\documentclass{amsart}

\usepackage{amsmath,comment}

\newtheorem{thm}{Theorem}[section]
\newtheorem{lem}[thm]{Lemma}
\newtheorem{cor}[thm]{Corollary}
\newtheorem{prop}[thm]{Proposition}

\theoremstyle{definition}
\newtheorem{defn}[thm]{Definition}
\newtheorem{rem}[thm]{Remark}
\newtheorem{ex}[thm]{Example}

\def\gl{{\rm GL}}
\def\M{{\rm M}}
\def\Ext{{\rm Ext}}

\def\rep{{\rm Rep}}
\def\rank{{\rm rank}}
\def\crys{{\rm crys}}
\def\fil{{\rm Fil}}
\def\gal{{\rm Gal}}
\def\val{{\rm val}_\pi}
\def\cyc{{\rm cyc}}
\def\ur{{\rm nr}}
\def\tr{{\rm tr}}
\def\bdd{{\rm bdd}}

\def\A{{\bf A}}
\def\B{{\bf B}}
\def\C{{\bf C}}
\def\D{{\bf D}}
\def\E{{\bf E}}
\def\F{{\bf F}}
\def\N{{\bf N}}
\def\Q{{\bf Q}}
\def\T{{\bf T}}
\def\V{{\bf V}}
\def\Z{{\bf Z}}

\def\CO{\mathcal{O}}

\newcommand{\gn}{\mathfrak{m}}

\begin{document}

\title
  {Extensions of rank one $(\varphi, \Gamma)$-modules
   and crystalline representations}

\author{Seunghwan Chang}
\address{Institute of Mathematical Sciences\\Ewha Womans University\\Seoul 120-750\\Republic of Korea}
\email{schang@ewha.ac.kr}      

\author{Fred Diamond}
\address{Department of Mathematics, King's College London, Strand, London WC2R 2LS, UK}
\email{Fred.Diamond@kcl.ac.uk}

\begin{abstract}  Let $K$ be a finite unramified extension of $\Q_p$.
We parametrize the $(\varphi, \Gamma)$-modules corresponding to 
reducible two-dimensional $\overline{\F}_p$-representations of $G_K$
and characterize those which have reducible crystalline lifts with
certain Hodge-Tate weights.
\end{abstract}

\thanks{}

\maketitle

\tableofcontents

\section{Introduction}
Buzzard, Jarvis and one of the authors \cite{BDJ05} have
formulated a generalization of Serre's conjecture for mod $p$
Galois representations over totally real fields unramified at $p$. To give a recipe
for weights, certain distinguished subspaces of local
Galois cohomology  groups in characteristic $p$ are defined in terms of 
the existence of ``crystalline lifts'' to characteristic zero.  
More precisely, let $K$ be a finite unramified extension of $\Q_p$ with residue
field $k$, $\F$ a finite extension of $\F_p$ containing $k$,
$\psi:G_K \to \F^\times$ a character, and denote by $S$ the set of embeddings
of $k$ in $\F$.  For each $J\subset S$, they define a subspace (or in certain
cases two subspaces) of $H^1(G_K,\F(\psi))$ which we denote $L_J$ (or $L_J^\pm$); 
with certain exceptions these subspaces have dimension $|J|$ (see
Remark~\ref{rmk:compare} below for the relation between our notation
and that of \cite{BDJ05}).  The definition of these subspaces in terms of crystalline
lifts is somewhat indirect, making it hard for example to compare the
spaces $L_J$ for different $J$.  Viewing the specification of the weights
in terms of a conjectural mod $p$ Langlands correspondence as in 
\cite[\S4]{BDJ05}, such a comparison provides
information about possible local factors at primes over $p$ of mod $p$
automorphic representations (see \cite{Bre09}).

The aim of this paper is to describe
them more explicitly using Fontaine's theory of $(\varphi, \Gamma)$-modules.
In particular, we prove that if $\psi$ is {\em generic}, as
defined in \S\ref{sec:generic}, then the subspaces are well-behaved
with respect to $J$ in the following sense:
\begin{thm} \label{thm:intro1} If $\psi$ is generic and $\psi|_{I_K} \neq 
\chi^{\pm1}$ where $\chi$ is the mod $p$ cyclotomic character, then 
$L_J = \oplus_{\tau\in J} L_{\{\tau\}}$.
\end{thm}
We remark that Theorem~\ref{thm:intro1} has been proved independently
by Breuil \cite[Prop.~A.3]{Bre09} using different methods.  We also
treat the case where $\psi|_{I_K} =  \chi^{\pm1}$; see 
Theorem~\ref{thm:generic} below for the statement.
 
We also give a complete description of the spaces $L_J$ (and $L_J^\pm$)
in terms of $(\varphi,\Gamma)$-modules when $K$ is quadratic,
without the assumption that $\psi$ is generic.  In particular,
we prove the following theorem which exhibits cases where the
spaces $L_J$ are not well-behaved as in Theorem~\ref{thm:intro1}:
\begin{thm}  \label{thm:intro2} Suppose that $[K:\Q_p] = 2$ and that $\psi$
is ramified. Writing $S = \{\tau,\tau'\}$, we have $L_{\{\tau\}}
 = L_{\{\tau'\}}$ if and only if $\psi|_{I_K} = \omega_2^i$
for some fundamental character $\omega_2$ of niveau $2$ and
some integer $i \in \{1,\ldots,p-1\}$.
\end{thm}
This is part of Theorem~\ref{thm:f2} below; see also Theorem~\ref{thm:unramified}
for the case when $\psi$ is unramified.

The paper is organized as follows: In \S 2 we review
preliminary facts on $p$-adic representations and $(\varphi,
\Gamma)$-modules, and set up the category of \'etale $(\varphi,
\Gamma)$-modules (corresponding to $\F[G_K]$-modules) in which we will
be working.  In \S 3 we give a parametrization of rank one objects in
the category, and identify them as reductions of crystalline characters
of $G_K$ using results of Dousmanis \cite{Dou07}.
In \S 4 we construct bases for the space of extensions of
rank ones.  (In a different but related direction, see
\cite{Her98, Her01, Liu07} for computation of $p$-adic
Galois cohomology  via $(\varphi, \Gamma)$-modules.)
In \S 5 we introduce the notion of bounded extensions, motivated
by the theory of Wach modules which characterizes those
$(\varphi,\Gamma)$-modules corresponding to crystalline representations
(see \cite{Wac96, Wac97, Ber03, Ber04}), and use this to define subspaces
$V_J^{(\pm)}$ which we compute in the generic and quadratic cases.
In \S 6 we treat certain exceptional cases
excluded from \S\S 4,5.  In \S 7 we relate the spaces $L_J^{(\pm)}$
and $V_J^{(\pm)}$ in the generic and quadratic cases and prove 
our main results.  We remark that a difficulty arises from the
fact that the integral Wach module functor is not right exact;
to overcome this we derive sufficient conditions for exactness
which may be of independent interest.

This paper grew out of the first author's Brandeis Ph.D.~thesis
\cite{thesis} written under the supervision of the second author.  The thesis
already contains most of the key technical results in \S\S 4,5.
The authors are grateful to Laurent Berger for helpful conversations
and correspondence, and to the referee for useful comments suggesting
corrections and improvements to the exposition.
The first author expresses his sincere gratitude to the second author
for suggesting the project and for the guidance and encouragement during doctoral study.
He would like to thank POSTECH, in particular Professor YoungJu Choie,
for the hospitality in the final stages of writing the paper.
He was supported by RP-Grant 2009 of Ewha Womans University.
The second author is grateful to the Isaac Newton
Institute for its hospitality in the final stages of writing this paper.
The research was supported by NSF grant \#0300434.

\section{Generalities on $p$-adic representations}
In this section we summarize (and expand a bit upon) basic facts on
$p$-adic representations, crystalline representations,
$(\varphi,\Gamma)$-modules and Wach modules.
We will give references for details and proofs along the way.
For excellent general introductions to the theory, see \cite{Ber04a}
and \cite{FO07}.

Let $p$ be a rational prime and fix an algebraic closure
$\overline{\Q}_p$ of $\Q_p$. If $K$ is a finite extension of $\Q_p$
contained in $\overline{\Q}_p$, $G_K$ denotes the Galois group ${\rm
Gal}(\overline{\Q}_p/K)$ and $K_0$ denotes the absolutely unramified
subfield of $K$. Let $\chi : G_K \rightarrow \Z_p^\times$ be the
cyclotomic character and let $\bar{\cdot}:\Z_p \to \F_p$ be the
reduction modulo $p$, so that $\overline{\chi}=\bar{\cdot}\circ\chi
: G_K \rightarrow \F_p^\times$ is the mod $p$ cyclotomic character.
We set $K_n = K(\mu_{p^n}) \subset \overline{\Q}_p$ for $n \ge 1$, 
and get a tower of fields
$$K=K_0 \subset K_1 \subset \cdots \subset K_n \subset \cdots
\subset K_\infty \subset \overline{\Q}_p$$ where $K_\infty =
\cup_{n\ge 1} K_n$. We define $H_K$ to be the kernel of $\chi$, i.e.,
$H_K=$ Gal$(\overline{\Q}_p/K_\infty)$ and set $\Gamma_K = G_K/H_K
=$ Gal$(K_\infty/K)$. In many cases where there is no confusion,
we will simply write $\Gamma$
for $\Gamma_K$ suppressing $K$. We set $\Gamma_n =
\Gamma_{K,n} =$ Gal$(K_\infty/K_n)$ for $n\ge 1$.

\subsection{Fontaine's rings}
Here we give a summary of the constructions of some of the rings
introduced by Fontaine that we will be using. See \cite{Col99, CC98, Fon94a}
for more details. Let $\C_p$ denote the $p$-adic completion of
$\overline{\Q}_p$ and $v_p$ the $p$-adic valuation normalized by
$v_p(p)=1$. The set $$\widetilde{\E}=\varprojlim_{x\mapsto x^p} \C_p
= \{x=(x^{(0)}, x^{(1)},\ldots) |\, x^{(i)}\in \C_p,
(x^{(i+1)})^p=x^{(i)}\}$$ together with the addition and the
multiplication defined by
$$(x+y)^{(i)}=\lim_{j\to \infty} (x^{i+j}+y^{i+j})^{p^j} \,\, {\rm
and} \,\, (xy)^{(i)}=x^{(i)}y^{(i)}$$ is an algebraically closed
field of characteristic $p$, complete for the valuation $v_\E$
defined by $v_\E(x)=v_p(x^{(0)})$. We endow $\widetilde{\E}$ a
Frobenius $\varphi$ and the action of $G_{\Q_p}$ by
$$\varphi((x^{(i)}))=((x^{(i)})^p) \,\, {\rm and} \,\, g((x^{(i)}))= (g(x^{(i)}))$$
if $g \in G_{\Q_p}$
We denote the ring of integers of $\widetilde{\E}$ by $\widetilde{\E}^+$;
it is stable under the actions of $\varphi$ and $G_{\Q_p}$.
Let $\varepsilon = (1, \varepsilon^{(1)}, \ldots, \varepsilon^{(i)},
\ldots)$ be an element of $\widetilde{\E}$ such that
$\varepsilon^{(1)}\neq 1$, so that $\varepsilon^{(i)}$ is a
primitive $p^i$-th root of unity for all $i\ge 1$. Then
$v_{\E}(\varepsilon-1)=p/(p-1)$ and $\E_{\Q_p}$ is defined to be the
subfield $\F_p((\varepsilon-1))$ of $\widetilde{\E}$. We define
$\E$ to be the separable closure of $\E_{\Q_p}$ in $\widetilde{\E}$ and
$\widetilde{\E}^+$ (resp. $\gn_\E$) to be the ring of integers (resp. the
maximal ideal) of $\E^+$. The field $\E$ is stable under the action of
$G_{\Q_p}$ and we have $\E^{H_{\Q_p}}=\E_{\Q_p}$. The theory of the field of
norms shows that $\E_K :=\E^{H_K}$ is a finite separable
extension of $\E_{\Q_p}$ of degree
$|H_{\Q_p}/H_K|=[K_\infty:\Q_p(\mu_{p^\infty})]$ and allows
one to identify Gal$(\E/\E_K)$ with $H_K$.
The ring of integers of $\E_K$ is denoted by $\E_K^+$.

Let $\widetilde{\A}=W(\widetilde{\E})$ be the ring of Witt vectors
with coefficients in $\widetilde{\E}$ and let $\widetilde{\B} =
\widetilde{\A}[1/p]= {\rm Fr}(\widetilde{\A})$. Then
$\widetilde{\B}$ is a complete discrete valuation field with ring of
valuation $\widetilde{\A}$ and residue field $\widetilde{\E}$. If $x
\in \widetilde{\E}$, $[x]$ denotes Teich\"muller representative of $x$
in $\widetilde{\A}$. Then every element of $\widetilde{\A}$ can be
written uniquely in the form $\sum_{i\ge 0}p^i[x_i]$ and that of
$\widetilde{\B}$ in the form $\sum_{i\gg\infty}p^i[x_i]$. We endow
$\widetilde{\A}$ with the topology which makes the map $x \mapsto
(x_i)_{i\in \N}$ a homeomorphism $\widetilde{\A}\to \widetilde{\E}^\N$ where
$\widetilde{\E}^\N$ is endowed with the product topology ($\widetilde{\E}$ is
endowed with the topology defined by the valuation $v_{\E}$). We
endow $\widetilde{\B}= \cup_{i\in\N}p^{-i}\widetilde{\A}$ with the topology of
inductive limit. The action of $G_{\Q_p}$ on $\widetilde{\E}$ induces
continuous actions on $\widetilde{\A}$ and $\widetilde{\B}$ which
commute with the Frobenius $\varphi$. Let $\pi=[\varepsilon]-1$. Define
$\A_{\Q_p}$ to be the closure of $\Z_p[\pi,\pi^{-1}]$ in
$\widetilde{\A}$. Then $$\A_{\Q_p}=\{\sum_{i\in \Z} a_n\pi^i|\, a_i \in
\Z_p, a_i \to 0 \,\,{\rm as}\,\, i\to -\infty\}$$ and $\A_{\Q_p}$ is
a complete discrete valuation ring with residue field $\E_{\Q_p}$.
As $$\varphi(\pi)=(1+\pi)^p-1 \,\,{\rm and}\,\,
\gamma(\pi)=(1+\pi)^{\chi(g)}-1 \,\,{\rm if}\,\, g \in G_{\Q_p},$$
the ring $\A_{\Q_p}$ and its field of fractions
$\B_{\Q_p}=\A_{\Q_p}[1/p]$ are stable under $\varphi$ and the action
of $G_{\Q_p}$. Let $\B$ be the closure of the maximal unramified
extension of $\B_{\Q_p}$ contained in $\widetilde{\B}$, and set
$\A=\B\cap \widetilde{\A}$, so that we have $\B=\A[1/p]$. Then $\A$ is a
complete discrete valuation ring with field of fractions $\B$ and
residue field $\E$. The ring $\A$ and the field $\B$ are stable under
$\varphi$ and $G_{\Q_p}$. If $K$ is a finite extension of $\Q_p$, we
define $\A_K=\A^{H_K}$ and $\B_K=\B^{H_K}$, which makes $\A_K$ a
complete discrete valuation ring with residue field $\E_K$ and the
field of fractions $\B_K=\A_K[1/p]$. When $K=\Q_p$, the two
definition of $\A_K$ and $\B_K$ coincide. If $F$ is a finite
extension of $K,$ then $\B_F$ is an unramified extension of $\B_K$
of degree $[F_\infty:K_\infty]$. If the extension $F/K$ Galois, then
the extensions $\widetilde{\B}_F/\widetilde{\B}_K$ and $\B_F/\B_K$
are also Galois with Galois group
$${\rm Gal}(\widetilde{\B}_F/\widetilde{\B}_K) = {\rm Gal}(\B_F/\B_K)
= {\rm Gal}(\E_F/\E_K) = {\rm Gal}(F_\infty/K_\infty) = H_K/H_F.$$
In particular, if $K$ is a finite unramified extension of $\Q_p$, we have
$$\A_K = \{\sum_{n\in \Z} a_n\pi^n|\, a_n \in \CO_K, a_n \to 0 \,\,{\rm as}\,\, n\to -\infty\}$$
with $\varphi$ acting as the Frobenius and $\Gamma$ acting trivially
on $\CO_K$.

The homomorphism $\theta : \widetilde{\A}^+ \to \CO_{\C_p}$,
$\sum_{n \ge 0}p^n[x_n] \mapsto \sum_{n \ge 0}p^nx_n^{(0)}$ is
surjective and its kernel is a principal ideal generated by $\omega =
\pi/\varphi^{-1}(\pi)$. We extend $\theta$ to a homomorphism
$\widetilde{\B}^+=\widetilde{\A}^+[1/p]\to \C_p$ and we set
$\B_{\rm dR}^+$ to be the ring
$\varprojlim\widetilde{\B}^+/(\ker\theta)^n$. Then $\theta$ extends
by continuity to a homomorphism $\B_{\rm dR}^+ \to \C_p$.
This makes $\B_{\rm dR}^+$ a discrete valuation ring with maximal
ideal $\ker\theta$ and residue field $\C_p$. The action of
$G_{\Q_p}$ on $\widetilde{\B}^+$ extends by continuity to a
continuous action of $G_{\Q_p}$ on $\B_{\rm dR}^+$. The series $\log
[\varepsilon]= \sum_{n\ge 1}(-1)^{n-1}\pi^n/n$ converges in $\B_{\rm
dR}^+$ to an element $t$, which is a generator of $\ker\theta$ on
which $\sigma \in G_{\Q_p}$ act via the formula
$\sigma(t)=\chi(\sigma)t$. We set $\B_{\rm dR} = \B_{\rm dR}^+[t^{-1}]
=$ Fr$\B_{\rm dR}^+$, and $\B_{\rm dR}$ comes with a decreasing, separated and exhaustive filtration
$\fil^i \B_{\rm dR}:= t^i\B_{\rm dR}^+$ for $i \in \Z$.
Let $\A_{\rm cris}=\{x=\sum_{n\ge 0}a_n \frac{\omega^n}{n!} \in
\B_{\rm dR}^+|\, a_n \in \widetilde{\A}^+, a_n \to 0 \}$. Then
$\B_{\rm cris}^+ = \A_{\rm cris}[1/p]$ is a subring of $\B_{\rm
dR}^+$ stable by $G_{\Q_p}$ and contains $t$, and the action of
$\varphi$ on $\widetilde{\B}^+$ extends by continuity to an action
of $\B_{\rm cris}^+$. We have $\varphi(t)=pt$ and we define $\B_{\rm
cris}$ to be the subring $\B_{\rm cris}^+[1/t]$ of $\B_{\rm dR}$, and define
the filtration $\fil^i \B_{\rm cris}:= \fil^i \B_{\rm dR} \cap \B_{\rm cris}$.

\subsection{Crystalline representations}
Let $K$ be a finite extension of $\Q_p$, and let $K_0$ denote its
maximal absolutely unramified subfield.

\begin{defn} A {\it $p$-adic representation of $G_K$} is a finite dimensional
$\Q_p$-vector space together with a linear and continuous action of $G_K$.
A {\it $\Z_p$-representation of $G_K$} is a $\Z_p$-module of finite type with a
$\Z_p$-linear and continuous action of $G_K$. A {\it mod $p$ representation of $G_K$}
is a finite dimensional $\F_p$-vector space with a linear and continuous action of $G_K$.
\end{defn}

\begin{rem} A $\Z_p$-representation $T$ of $G_K$ which is torsion-free over $\Z_p$ is naturally identified with a
({\it $G_K$-stable}) {\it lattice} of the $p$-adic representation $V:=\Q_p\otimes_{\Z_p} T$ of $G_K$.
\end{rem}

If $B$ is a topological $\Q_p$-algebra endowed with a continuous
action of $G_K$ and if $V$ is a $p$-adic representation of $G_K$, we
define $\D_B(V):=(B\otimes_{\Q_p}V)^{G_K}$, which is naturally a
module over $B^{G_K}$. If, in addition, $B$ is $G_K$-regular (i.e., $B$ is
a domain, $({\rm Fr}B)^{G_K}=B^{G_K}$, and every $b \in B-\{0\}$ such
that $\Q_pb$ is stable under $G_K$-action is a unit), then the map
$$\alpha_V: B\otimes_{B^{G_K}}\D_B(V)\to B\otimes_{\Q_p}V$$
defined by $b\otimes v \mapsto 1\otimes bv$ is an injection (see \S 1.3 of \cite{Fon94b}). 
In particular, we have $$\dim_{B^{G_K}}\D_B(V) \le \dim_{\Q_p}V.$$
(If $B$ is $G_K$-regular, $B^{G_K}$ is forced to be a field.)

\begin{defn} If $B$ is $G_K$-regular, we say that a $p$-adic representation $V$ of $G_K$ is
{\it $B$-admissible} if $\dim_{B^{G_K}}\D_B(V) = \dim_{\Q_p}V$. We say that $V$ is {\it crystalline}
if it is $\B_{\rm cris}$-admissible and that $V$ is {\it de Rham} if $\B_{\rm dR}$-admissible.
\end{defn}

\begin{rem} We have the following equivalent conditions:
\begin{enumerate}
\item $\dim_{B^{G_K}}\D_B(V) = \dim_{\Q_p}V$;
\item $\alpha_V$ is an isomorphism;
\item $B\otimes_{\Q_p}V \simeq B^{\dim_{\Q_p}V}$ as $B[G_K]$-modules.
\end{enumerate}
(cf. \S 1.4 of \cite{Fon94b})
\end{rem}

If $V$ is a $p$-adic representation of $G_K$, $\D_{\rm
dR}(V):=(\B_{\rm dR}\otimes_{\Q_p}V)^{G_K}$ is naturally a filtered $K$-vector space.  More precisely,
it is a finite dimensional $K$-vector space with a decreasing, separated and exhaustive filtration
$\fil^i \D_{\rm dR}(V):=(\fil^i\B_{\rm dR}\otimes_{\Q_p}V)^{G_K}$ of
$K$-subspaces for $i \in \Z$. If $V$ is de Rham, a {\it Hodge-Tate
weight} of $V$ is defined to be an integer $h \in \Z$ such that
$\fil^{h}\D_{\rm dR}(V)\neq \fil^{h+1}\D_{\rm dR}(V)$ with
multiplicity $\dim_K \fil^{h}\D_{\rm dR}(V)/\fil^{h+1}\D_{\rm
dR}(V)$. So there are $\dim_{\Q_p}V$ Hodge-Tate weights of $V$
counting multiplicities.  We remark that this differs from the
more standard convention (e.g., \cite{Ber04a}) of
defining $-h$ to be a Hodge-Tate weight of $V$ for $h$ as above.

\begin{defn} A {\it filtered $\varphi$-module over $K$} is a finite dimensional $K_0$-vector space $D$ together with
a $\sigma$-semilinear bijection $\varphi : D \to D$ and a
$\Z$-indexed filtration on $D_K:=D\otimes_{K_0}K$ of $K$-subspaces
which is decreasing, separated and exhaustive.
\end{defn}

If $V$ is a $p$-adic
representation of $G_K$, then $\D_{\rm cris}(V):= (\B_{\rm
cris}\otimes_{\Q_p} V)^{G_K}$ is a filtered $\varphi$-module over
$K$. More precisely, the Frobenius on $\B_{\rm cris}$ induces a
Frobenius map $\varphi:\D_{\rm cris}(V) \to \D_{\rm cris}(V)$ and
the filtration on $\B_{\rm dR}$ induces a filtration $\fil^i\D_{\rm
cris}(V):=D_K \cap (\fil^i\B_{\rm dR}\otimes_{\Q_p}V)^{G_K}$ on
$\D_{\rm cris}(V)$. Moreover, $\D_{\rm cris}(V)$ has finite
dimension over $K_0$ and $\varphi$ is bijective on $\D_{\rm
cris}(V)$. We get a functor $$\D_{\rm cris}: {\rm
Rep}_{\Q_p}G_K \to {\rm MF}_K^\varphi$$ from the category of
$p$-adic representations of $G_K$ to the category of filtered
$\varphi$-modules over $K$.

If $D$ is a filtered $\varphi$-module over $K$ of finite dimension
$d\ge 1$, then $\wedge^d D$ is a filtered $\varphi$-module of
dimension $1$. If $e \in \wedge^d_{K_0} D -\{0\}$ and
$\varphi(e)=\lambda e$ then ${\rm val}(\lambda)$ is independent of
choice of $e$ and we define $t_N(D):=v_p(\lambda)$. Also, we define
$t_H(D)=t_H(D_K)$ to be the largest integer such that
$\fil^{t_H(D)}(\wedge^d_K D_K)$ is nonzero, i.e. $\fil^i(\wedge^d_K
D_K)=\wedge^d_K D_K$ for $i\le t_H(D)$ and $\fil^i(\wedge^d_K
D_K)=0$ for $i> t_H(D)$.

\begin{defn} Let $D$ be a filtered $\varphi$-module over $K$. We say that $D$
is {\it weakly admissible} if $t_H(D)=t_N(D)$ and $t_H(D')\le t_N(D')$ for every
subobject $D'$ of $D$. We say that $D$ is {\it admissible} if $D
\simeq \D_{\rm cris}(V)$ for some $p$-adic representation $V$ of
dimension $\dim_{K_0}D$.
\end{defn}

One can show that if $V$ is a crystalline representation of $G_K$, then $\D_{\rm cris}(V)$ is weakly admissible.
The converse was conjectured by Fontaine, and proved by Colmez and Fontaine.

\begin{thm}[\cite{CF00}] Every weakly admissible filtered $\varphi$-module over $K$ is admissible.
\end{thm}

In sum, we have an equivalence of categories
$$\D_{\rm cris}: {\rm Rep}_{\Q_p}^{cris}G_K \to {\rm MF}_K^{\varphi, {w.a.}}$$
between crystalline representations of $G_K$ and weakly admissible filtered $\varphi$-modules over $K$
with a quasi-inverse given by $\V_{\rm cris}(\cdot):=(\fil^0(\cdot))^{\varphi=1}$.

\subsection{$(\varphi, \Gamma)$-modules}

\begin{defn} A $(\varphi, \Gamma)$-{\it module over $\A_K$}
(resp. $\B_K$, $\E_K$) is an $\A_K$-module of finite type (resp.
finite dimensional vector space over $\B_K$, $\E_K$) endowed with a
semilinear and continuous action of $\Gamma_K$ and with a semilinear map
$\varphi$ which commutes with the action of $\Gamma_K$. We say that
a $(\varphi, \Gamma)$-module $M$ over $\A_K$ (resp. $\E_K$) is {\it \'etale} if
$\varphi(M)$ generates $M$ over $\A_K$ (resp. $\E_K$). A $(\varphi, \Gamma)$-module
$M$ over $\B_K$ is {\it \'etale} if $M$ contains an $\A_K$-lattice
which is stable under $\varphi$ and is \'etale.
\end{defn}

\begin{rem} We identify a
(\'etale) $(\varphi,\Gamma)$-module over $\A_K$ killed by $p$ with 
the corresponing (\'etale) $(\varphi, \Gamma)$-modules over 
$\E_K$.
\end{rem}

If $T$ is a $\Z_p$-representation of $G_K$, we define
$\D(T)=(\A\otimes_{\Z_p}T)^{H_K}$. Then $\D(T)$ is naturally a
module over $\A_K$ of finite type. The Frobenius $\varphi$ on $\A$
induces a Frobenius map $\varphi : \D(T)\to \D(T)$ and the residual
action of $\Gamma_K$ on $\D(T)$ commutes with $\varphi$. One
can also check that $\D(T)$ is \'etale over $\A_K$. Conversely, if $M$ is an
\'etale $(\varphi, \Gamma)$-module over $\A_K$ we define $\T(M)=
(\A\otimes_{\A_K}M)^{\varphi=1}$, which is a $\Z_p$-representation
of $G_K$.

\begin{thm}[\cite{Fon91}]\label{equiv-fon91}
The functor $T \mapsto \D(T)$ defines an equivalence of categories
$$\D: {\rm Rep}_{\Z_p}G_K \to {\rm M}_{\A_K}^{\varphi, \Gamma, et}$$
between $\Z_p$-representations and \'etale $(\varphi, \Gamma)$-modules over $\A_K$
with $\T$ as a quasi-inverse. It induces, by inverting $p$, an equivalence of categories
$$\D: {\rm Rep}_{\Q_p}G_K \to {\rm M}_{\B_K}^{\varphi, \Gamma, et}$$
between $p$-adic representations and \'etale $(\varphi, \Gamma)$-modules over $\B_K$ with
$M \mapsto \V(M):=(\B\otimes_{\B_K}D)^{\varphi=1}$ as a quasi-inverse.
Moreover, if $T$ is a $\Z_p$-representation and $V$ a $p$-adic representation of $G_K$,
then
$$\begin{aligned}
{\rm rank}_{\Z_p}T & = {\rm rank}_{\A_K}\D(T), \\
\dim_{\Q_p}V & = \dim_{\B_K}\D(V).
\end{aligned}$$

\end{thm}

When we restrict the equivalence to the $p$-torsion objects we get
the following.

\begin{cor}\label{equiv-tor} The functor $T \mapsto \D(T)$ defines
an equivalence of categories between mod $p$ representations of
$G_K$ and \'etale $(\varphi, \Gamma_K)$-modules over $\E_K$.
\end{cor}

Now we introduce coefficients to representations of $G_K$ and
$(\varphi, \Gamma)$-modules to extend Theorem~\ref{equiv-fon91} and 
Corollary~\ref{equiv-tor}. We assume $K$ is absolutely unramified (of degree
$f$ over $\Q_p$) and let $F$ be a finite extension of $\Q_p$ with ring of integers $\CO_F$,
uniformizer $\varpi_F$ and residue field $\F$.   
Consider the ring $\A_{K,F} := \CO_F \otimes_{\Z_p} \A_K$
with the actions of 
$\varphi$ and $\Gamma_K$ extended to $\A_{K,F}$ by linearity,
i.e. $\varphi$ acts as $1\otimes \varphi$ and $\gamma \in \Gamma_K$ as $1\otimes \gamma$.
We assume there is an embedding $\tau_0: K \hookrightarrow F$, which we fix once and for all,
and put $\tau_i=\tau_0\circ {\rm \varphi}^i$ where $\varphi$ is the
Frobenius on $K$. We denote by $S$ the set of all
embeddings $K \hookrightarrow F$ and fix the identification
$S=\Z/f\Z$ via the map $\tau_i \mapsto i$.  We can then
identify $\A_{K,F}$ with $\A_{\Q_p,F}^S$
via the isomorphism defined by $a\otimes b\pi^n
\mapsto (a\tau(b)\otimes\pi^n)_{\tau}$.  Note that
$$\A_{\Q_p,F} = \{\sum_{n\in \Z} a_n\pi^n|\, a_n \in \CO_F, a_n \to 0 \,\,{\rm as}\,\, n\to -\infty\},$$
and the actions of $\varphi$ and
$\gamma \in \Gamma_K$ on $\A_{\Q_p,F}^S$ become
$$
\begin{aligned}
\varphi(g_0(\pi), g_1(\pi), \ldots, g_{f-1}(\pi))
&=(g_1(\varphi(\pi)), \ldots, g_{f-1}(\varphi(\pi)), g_0(\varphi(\pi^))),\\
\gamma(g_0(\pi), g_1(\pi), \ldots, g_{f-1}(\pi))
&=(g_0(\gamma(\pi)), g_1(\gamma(\pi)), \ldots,
g_{f-1}(\gamma(\pi))).
\end{aligned}$$
We similarly define $\B_{K,F}=F\otimes_{\Q_p}\B_K$ and 
$\E_{K,F} = \F\otimes_{\F_p}\E_K$ and endow them with
actions of $\varphi$ and $\Gamma$.  Note that $\B_{K,F}=\A_{K,F}[1/p]$
and $\E_{K,F} = \A_{K,F}/\varpi_F\A_{K,F}$.  Again
identifying $S$ with the set of embeddings $k \to \F$,
we have the isomorphism $\E_{K,F} = \F((\pi))^S$
with the actions of $\varphi$ and $\Gamma_K$ given by
the same formulas as above.

\begin{defn} An {\it $\CO_F$-representation} of $G_K$ is a finitely generated
$\CO_F$-module with a continuous $\CO_F$-linear action of $G_K$. A {\it $(\varphi,
\Gamma_K)$-module over $\A_{K,F}$} is a finitely generated $\A_{K,F}$-module $M$
endowed with commuting semilinear actions of $\Gamma_K$ and $\varphi$. A
$(\varphi, \Gamma_K)$-module $M$ over $\A_{K,F}$ is {\it \'etale} if $\varphi(M)$
generates $M$ over $\A_K$, or equivalently over $\A_{K,F}$.
\end{defn}

We write $\rep_{\CO_F} G_K$ for the category of $\CO_F$-representations
of $G_K$, and ${\rm M}_{\A_{K,F}}^{\varphi, \Gamma, et}$ for that of
 \'etale $(\varphi, \Gamma_K)$-modules over $\A_{K,F}$.
We use analogous definitions and notation for
representations of $G_K$ over $F$ and $\F$,
and $(\varphi,\Gamma_K)$-modules over $\B_{K,F}$ and $\E_{K,F}$.
The category of \'etale $(\varphi, \Gamma_K)$-modules over
$\E_{K,F}$ is the main category we will be working in. 
Theorem~\ref{equiv-fon91}
and Corollary~\ref{equiv-tor} immediately yield the following:

\begin{cor}\label{equiv-labeled} 
The functor $\D$ induces equivalences of categories
$\rep_{\CO_F} G_K  \to {\rm M}_{\A_{K,F}}^{\varphi, \Gamma, et}$,
$\rep_{F} G_K  \to {\rm M}_{\B_{K,F}}^{\varphi, \Gamma, et}$ and
$\rep_{\F} G_K  \to {\rm M}_{\E_{K,F}}^{\varphi, \Gamma, et}$.
\end{cor}

For each
embedding $\tau : K \hookrightarrow \F$, let 
$e_\tau : \A_{K,F}\rightarrow \A_{\Q_p,F}$ denote the projection 
to the $\tau$-component, defined by $a\otimes b\pi^i \mapsto a\tau(b)\pi^i$.
If $M$ is a $(\varphi, \Gamma)$-module over $\A_{K,F}$, then
$M = \prod_{\tau\in S} e_\tau M$, each $e_\tau M$ inherits an action
of $\Gamma$, and $\varphi$ induces semilinear morphisms
$e_{\tau\circ\varphi}M \to e_\tau M$
compatible with the action of $\Gamma$.
We use the same notation for $(\varphi,\Gamma)$-modules
over $\B_{K,F}$ and $\E_{K,F}$.

\begin{lem} \label{lem:free} If $M$ is an \'etale $(\varphi,\Gamma)$-module over $\A_{K,F}$, 
then the following are equivalent:
\begin{enumerate}
\item $\T(M)$ is free over $\CO_F$ of rank $d$;
\item $M$ is free over $\A_K$ of rank $d[F:\Q_p]$;
\item $M$ is free over $\A_{K,F}$ of rank $d$. 
\end{enumerate}
If $M$ is an \'etale $(\varphi, \Gamma)$-module over $\B_{K,F}$ 
(resp.~$\E_{K,F}$), then $M$ is free over $\B_{K,F}$ (resp.~$\E_{K,F}$)
of rank $\dim_F \T(M)$ (resp.~$\dim_{\F}\T(M)$).  
\end{lem}
\begin{proof} 
Suppose that $M$ is \'etale over $\A_{K,F}$.  Then multiplication by
$p$ is injective on $M$ if and only if it is injective on $\T(M)$.
Thus $M$ is torsion-free, and hence free, over $\A_K$ if and only if
$\T(M)$ is free over $\CO_F$.  Since the $\A_K$-rank of $M$ coincides
with the $\Z_p$ rank of $\T(M)$, the first two conditions
are equivalent.

If $M$ is free of rank $d$ over $\A_{K,F}$, then it is clearly
free of rank $d[F:\Q_p]$ over $\A_K$.  Conversely suppose that
$M$ is free over $\A_K$.  Then each $e_\tau M$ is torsion-free,
hence free, over the discrete valuation ring $\A_{\Q_p,F}$.
We need only show that each $e_\tau M$ has the same rank.
Since $M$ is \'etale, the maps 
$$e_{\tau\circ\varphi}M\otimes_{\A_{\Q_p,F},\varphi} \A_{\Q_p,F} \to e_\tau M$$
are surjective, so we have $\rank (e_{\tau_i} M )\le \rank (e_{\tau_{i+1}}M)$
for all $i\in \Z/f\Z$.  The equivalence between the last two conditions
follows.

The assertions for \'etale $(\varphi,\Gamma)$-modules over
$\B_{K,F}$ and $\E_{K,F}$ are similar, but simpler since
$\B_K$ and $\E_K$ are fields.
\end{proof}

Finally, there are tensor products and exact sequences in the various categories of
\'etale $(\varphi, \Gamma)$-modules, compatible via $\D$ with tensor products and
exact sequences in the corresponding categories of representations of $G_K$.

\subsection{Wach modules}
It is very useful to be able to characterize whether a $p$-adic representation
is crystalline in terms of the corresponding $(\varphi, \Gamma)$-module.
This can be done via the theory of Wach modules if $K$ is unramified over $\Q_p$.

Let $\A^+ =\A \cap \widetilde{\A}^+=\B \cap \widetilde{\A}^+$ and
$\B^+=\A^+[1/p]$. If $K$ is a finite unramified extension of $\Q_p$, we set
$\A_K^+ = (\A^+)^{H_K}=\CO_K[[\pi]]\subset \A_K$ and $\B_K^+ = (\B^+)^{H_K}=\A_K^+[p^{-1}] \subset \B_K$.

\begin{defn} Let $K$ be a finite unramified extension of $\Q_p$.
We say that a $\Z_p$-representation $T$ (resp. $p$-adic representation $V$) of $G_K$, is
{\it of finite height} if there exists a basis of $\D(T)$ (resp. $\D(V)$) such that
the matrices describing the action of $\varphi$ and the action of
$\Gamma_K$ are defined over $\A_K^+$ (resp. $\B_K^+$).
\end{defn}

Colmez \cite{Col99} proved that every crystalline representation is
necessarily of finite height. The converse is not true in general
and there are representations of finite height which are not
crystalline. However, Wach \cite{Wac96, Wac97} proved that finiteness
of height together with a certain condition (existence of a certain
$\A_K^+$-submodule of the corresponding $(\varphi, \Gamma)$-module)
implies crystallinity. Berger \cite{Ber03, Ber04} then refined the
results of Wach and Colmez as summarized below.

\begin{defn} Suppose $a \le b \in \Z$. A {\em Wach module} over $\A_K^+$ (resp. $\B_K^+$)
with weights in $[a,b]$ is a free $\A_K^+$-module (resp.
$\B_K^+$-module) $N$ of finite rank, endowed with an action of
$\Gamma_K$ which becomes trivial modulo $\pi$, and also
with a Frobenius map $\varphi : N[1/\pi] \rightarrow N[1/\pi]$ which
commutes with the action of $\Gamma_K$ and such that
$\varphi(\pi^{-a}N)\subset \pi^{-a}N$ and $\pi^{-a}N/\varphi(\pi^{-a}N)$ is
killed by $q^{b-a}$ where we define $q:=\varphi(\pi)/\pi$.
\end{defn}

\begin{thm}[\cite{Ber04}]\label{berger} 
\begin{enumerate}
\item A $p$-adic representation $V$ is crystalline with Hodge-Tate weights in $[a,b]$
if and only if $\D(V)$ contains a Wach module $\N(V)$ of rank $\dim_{\Q_p} V$ with weights
in $[a,b]$.  The association $V \mapsto \N(V)$ induces an equivalence of
categories between crystalline representations of $G_K$ and Wach
modules over $\B_K^+$, compatible with tensor products, duality and
exact sequences.
\item For a given crystalline representation $V$, the map $T
\mapsto \N(T):=\N(V)\cap \D(T)$ induces a bijection between $G_K$-stable lattices of
$V$ and Wach modules over $\A_K^+$ which are $\A_K^+$-lattices
contained in $\N(V)$.  Moreover $\D(T) = \A_K\otimes_{\A_K^+} \N(T)$.
\item If $V$ is a crystalline representation of $G_K$, and if we
endow $\N(V)$ with the filtration ${\rm Fil}^i\N(V)=\{x \in \N(V)|
\varphi(x) \in q^i\N(V)\}$, then we have an isomorphism $\D_{\rm
cris}(V)\to \N(V)/\pi\N(V)$ of filtered $\varphi$-modules (with the induced filtration on $\N(V)/\pi\N(V)$).
\end{enumerate}
\end{thm}

\begin{rem} If $0\to V_1 \to V \to V_2 \to 0$ is an exact sequence
of crystalline representations of $G_K$, then
$$0 \to \N(V_1) \to \N(V) \to \N(V_2) \to 0$$
is an exact sequence of $\B_K^+$-modules.  However
$\N$ does not define an exact functor from $G_K$-stable lattices
to $\A_K^+$-modules; indeed it fails to be right exact.  We return
to this point in more detail in \S\ref{sec:crys}.
\end{rem}

Again by introducing an action of $F$ to the categories, we get an analogous
equivalence of categories between crystalline $F$-representations and
Wach modules over $\B_{K,F}^+:=F \otimes_{\Q_p}\B_K^+$.
Here, by a crystalline $F$-representation we mean a finite
dimensional $F$-vector space with a continuous action of $G_K$ which
is crystalline considered as a $\Q_p$-linear representation (i.e.,
forgetting $F$-structure).  Similarly, for a fixed crystalline
$F$-representation of $G_K$, we have a corresponding equivalence of categories
between $G_K$-stable $\CO_F$-lattices and Wach modules over $\A_{K,F}^+:=\CO_F
\otimes_{\Z_p}\A_K^+$.

\begin{cor} Let $k \in \Z_{\ge 0}$.
An $F$-representation $V$ of $G_K$ is crystalline with Hodge-Tate weights in $[0,k]$ (i.e., positive crystalline)
if and only if there exists a $\B_{K,F}^+$-module $N$ free of rank $d:=\dim_F(V)$ contained in $\D(V)$ such that
\begin{enumerate}
\item the $\Gamma$-action preserves $N$ and is trivial on $N/\pi N$, and
\item $\varphi(N)\subset N$ and $N/\varphi^*(N)$ is killed by $q^k$.
\end{enumerate}
Moreover, if $N$ is given a filtration by
$$\fil^i(N):=\{x \in N | \varphi(x) \in q^iN\}$$ for $i \ge 0$,
then we have an isomorphism $$\D_{\rm cris}(V) \simeq N/\pi N$$
of filtered $\varphi$-modules over $F\otimes_{\Q_p}K$ where $N/\pi N$
is endowed with induced filtration.
\end{cor}

A standard argument (cf.~Lemma~\ref{lem:free})
shows that an $F$-representation $V$ of $G_K$ is crystalline if and only if the filtered
$\varphi$-module $\D_{\rm cris}(V)=(\B_{\rm cris}\otimes_{\Q_p}V)^{G_K}$
is free of rank $\dim_FV$ over $F\otimes_{\Q_p}K$. We have a
decomposition $\D_{\rm cris}(V)=\oplus_{\tau : K \hookrightarrow F}
e_\tau\D_{\rm cris}(V)$ where $e_\tau\D_{\rm cris}(V)$ is the
filtered $F$-vector space $\D_{\rm
cris}(V)\otimes_{K\otimes_{\Q_p}F, e_\tau}F$ with the filtration given
by $\fil^i e_\tau\D_{\rm cris}(V):=e_\tau\fil^i\D_{\rm cris}(V)$. A
\emph{labeled Hodge-Tate weight} with respect to the embedding
$\tau: K \hookrightarrow F$ is an integer $h \in \Z$ such that
$\fil^{h} e_\tau\D_{\rm cris}(V)\neq \fil^{h+1} e_\tau\D_{\rm
cris}(V)$, counted with multiplicity 
$$\dim_F \fil^{h} e_\tau\D_{\rm
cris}(V)/ \fil^{h+1} e_\tau\D_{\rm cris}(V).$$

\begin{lem} \label{Wach-free}  If $N$ is a Wach module over 
over $\A_{K,F}^+$ (resp.~$\B_{K,F}^+$), then $N$ is free
over $\A_{K,F}^+$ (resp.~$\B_{K,F}^+$). 
\end{lem}
\begin{proof} We just give the proof for Wach modules over $\A_{K,F}^+$;
the case of $\B_{K,F}^+$ can be deduced from this or proved
similarly.

Let $T$ denote the $\CO_F$-representation corresponding to $T$,
and let $d$ denote its rank.
The $\CO_F\otimes_{\Z_p}\CO_K$-module $N/\pi N$
is a lattice in $D_\crys(\Q_p\otimes_{\Z_p}T)$, which is
free of rank $d$ over $F\otimes_{\Q_p}K$.  It follows that
$N/\pi N$ is free of rank $d$ over $\CO_F\otimes_{\Z_p}\CO_K$.
Since $\pi$ is in the Jacobson radical of $\A_{K,F}^+$,
Nakayama's Lemma shows that $N$ is generated by $d$ 
elements over $\A_{K,F}^+$.
By Lemma~\ref{lem:free}, we know that 
$\A_{K,F}\otimes_{\A_{K,F}^+} N$ is free of rank $d$
over $\A_{K,F}$, so it follows that $N$ is free of rank
$d$ over $\A_{K,F}^+$.
\end{proof}

\section{Rank one modules} \label{sec:rk1}
In this section we give a parametrization of rank one \'etale $(\varphi,\Gamma)$-modules over $\E_{K,F}$
(with a view toward parametrizing their extensions) and then identify them with the
reduction modulo $p$ of Wach modules of rank one over $\A_{K,F}^+$.

\subsection{A parametrization} \label{subsec:param}
Denote by $\val : \F((\pi)) \to \Z$ the valuation normalized by $\val(\pi)=1$,
and let $\lambda_\gamma \in \F_p[[\pi]]$ be the unique $\frac{p^f-1}{p-1}$-th root of $\frac{\gamma(\pi)}{\overline{\chi}(\gamma)\pi}$ which is
$\equiv 1 \mod \pi$, if $\gamma \in \Gamma$.

\begin{prop}\label{rankone} For any $C \in \F^\times$ and any
$\vec{c}=(c_0,\ldots,c_{f-1}) \in \Z^S$, letting $M=\E_{K,f}e$
with
$$
\begin{aligned}
\varphi(e) &=Pe=(C\pi^{(p-1)c_0},\pi^{(p-1)c_1},\ldots,
\pi^{(p-1)c_{f-1}})\,e,\\
\gamma(e) &=G_\gamma e=(\lambda_\gamma^{\sum_0\vec{c}},
\lambda_\gamma^{\sum_1\vec{c}}, \ldots,
\lambda_\gamma^{\sum_{f-1}\vec{c}})\,e,
\end{aligned}
$$
where $\Sigma_l = \Sigma_l\vec{c} = \sum c_ip^j$ summing over $0\le
i,j\le f-1$, $i-j\equiv l \mod f$, defines an \'etale $(\varphi,
\Gamma)$-module of rank one over $\E_{K,F}$. Conversely, for any rank
one \'etale $(\varphi, \Gamma)$-module $M$ over $\E_{K,F}$ we can choose a basis $e$
so that $M=\E_{K,F}e$ with the action of $\varphi$ and $\Gamma$ given as above
for some $C$ and some $\vec{c}$.
Two such modules $M$ and $M'$ are isomorphic if and only if $C=C'$ and
$\Sigma_0\vec{c} \equiv \Sigma_0\vec{c'} \mod p^f-1$.
In particular, every rank one $(\varphi, \Gamma)$-module over $\E_{K,F}$ can
be written uniquely in this form with $0 \le c_i \le p-1$ and at least one $c_i <p-1$.
\end{prop}

\begin{proof} To show that the given formula actually defines an \'etale
$(\varphi, \Gamma)$-module we need to verify that
$P\varphi(G_\gamma)=G_\gamma \gamma(P)$ and
$G_{\gamma\gamma'}=G_\gamma\gamma(G_{\gamma'})$. The first identity
holds as $$\begin{aligned}\varphi(G_\gamma)/G_\gamma &=
(\lambda_\gamma^{p\Sigma_1-\Sigma_0},\ldots
,\lambda_\gamma^{p\Sigma_0-\Sigma_{f-1}})\\
& = (\lambda_\gamma^{c_0(p^f-1)},\ldots
,\lambda_\gamma^{c_{f-1}(p^f-1)})\\&=\left(\left(\frac{\gamma(\pi)}{\pi}\right)^{c_0(p-1)},\ldots,
\left(\frac{\gamma(\pi)}{\pi}\right)^{c_{f-1}(p-1)}\right)=\gamma(P)/P.\end{aligned}$$
To prove the second identity, as $\Gamma$ acting componentwise, we
need to show that $\lambda_{\gamma\gamma'}=\lambda\gamma(\lambda)$.
But note that
$$\left(\lambda_\gamma\gamma(\lambda_{\gamma'})\right)^{\frac{p^f-1}{f-1}}
={\frac{\gamma(\pi)}{\pi}\overline{\chi}(\gamma)}\,\gamma\left({\frac{\gamma'(\pi)}{\pi}\overline{\chi}(\gamma')}\right)
={\frac{\gamma\gamma'(\pi)}{\pi}\overline{\chi}(\gamma\gamma')}$$ and
$\lambda_\gamma\gamma(\lambda_{\gamma'})\equiv 1 \mod \pi$. The
claim follows from uniqueness of $\lambda_\gamma$'s. Note also that
the function $\gamma\mapsto \lambda_\gamma$ is continuous since it
is the composite of $\gamma(\pi)/\overline{\chi}(\gamma)\pi$ with
the inverse of the continuous bijective function $x\mapsto x^{(p^f-1)/(p-1)}$
on the compact Hausdorff space $1+\pi\F_p[[\pi]]$; it follows that the
$\Gamma$-action we have just defined is continuous.

We now prove that any rank one module can be written in this form.
Suppose we are given a rank one module $M=\E_{K,F}e$ such that
$\varphi(e)=(h_0(\pi),\ldots,h_{f-1}(\pi))e$ and $\gamma(e)=
(g_0(\pi),\ldots,g_{f-1}(\pi))e$. Note that if $u \in \E_{K,F}^\times$, by a
change of basis $e'=ue$ we get $P'=(\varphi(u)/u)P$ and $G_\gamma '=(\gamma(u)/u)G_\gamma$
where $\varphi(e')=P'e'$ and $\gamma(e')=G_\gamma ' e'$.
If $u = (\pi^j,\ldots,\pi^j)$,
then $\varphi(u)/u =(\pi^{(p-1)j}, \ldots,\pi^{(p-1)j})$. So we can
assume that $h_i(\pi) \in \F[[\pi]]$ by choosing a large enough $j > 0$.
We can ``shift'' between components by appropriate change of basis: if
$u=(1,\ldots,1,u_i(\pi),1,\ldots,1)$, then
$\varphi(u)/u=(1,\ldots,1,u_{i-1}(\pi^p),u_i(\pi)^{-1},1,\ldots,1)$.
By successive changes of basis we can make it into a form where
$\varphi(e)=(h(\pi),1,\ldots,1)e$ with $h(\pi) \in \F[[\pi]]$. Moreover
 for some choice of $e$, $\varphi(e)=(C\pi^v,1,\ldots, 1)e$ for
$C \in \F^\times$ and $v \ge 0$ as
$${\frac{\varphi(u(\pi),u(\pi^{p^{f-1}}),\ldots,u(\pi^{p^2}))}
{(u(\pi),u(\pi^{p^{f-1}}),\ldots,u(\pi^{p^2}))}}=(u(\pi^{p^f})/u(\pi),
1\ldots,1)$$ and the map $1+\pi\F[[\pi]]\rightarrow 1+\pi\F[[\pi]]$,
$u(\pi) \mapsto u(\pi^{p^f})/u(\pi)$, is surjective: as the map is
multiplicative and $1+\pi\F[[\pi]]$ is complete $\pi$-adically, it
suffices to prove that for any $s \ge 1$ and $\alpha \in \F^\times$,
$1+\alpha\pi^st(\pi)$ is in the image for some $t(\pi) \in
\F[[\pi]]^\times$, and indeed $1-\alpha\pi^s \mapsto
(1-\alpha\pi^{sp^f})/(1-\alpha\pi^s) \equiv 1+\alpha\pi^s \mod
\pi^{s+1}$.

To show that $(p-1) | v$, we note that $\varphi\gamma(e)=\gamma\varphi(e)$ if and only if
$${\frac{\varphi(g_0,\ldots,g_{f-1})}{(g_0,\ldots,g_{f-1})}}
={\frac{\gamma(C\pi^v,1,\ldots,1)}{(C\pi^v,1,\ldots,1)}}$$ where
$G_\gamma=(g_0,\ldots, g_{f-1})$. This is equivalent to
$$\left({\frac{g_1(\pi^p)}{g_0(\pi)}},\ldots,
{\frac{g_0(\pi^p)}{g_{f-1}(\pi)}}\right)
=\left(\left({\frac{\gamma(\pi)}{\pi}}\right)^v,1,\ldots,1\right),$$
which implies that $(\gamma(\pi)/\pi)^v=g_0(\pi^{p^f})/g_0(\pi)
\equiv 1 \mod \pi$. If $\delta \in \Gamma$ is such that
$\delta\Gamma_1$ generates $\Gamma/\Gamma_1 \simeq \mu_{p-1}$ then
$\delta(\pi)/\pi \equiv \chi(\delta) \mod \pi$. Thus
$\delta(\pi)/\pi$ has order $p-1$ modulo $\pi$ so that  $p-1 |v$ and
$\varphi(e)=(C\pi^{(p-1)w}, 1, \cdots ,1)$ where $(p-1)w=v$.

To determine the corresponding action of $\gamma \in \Gamma$, we note that
$\varphi\gamma(e)=\gamma\varphi(e)$ if and only if
$$\left({\frac{g_1(\pi^p)}{g_0(\pi)}},\ldots,{\frac{g_0(\pi^p)} {g_{f-1}(\pi)}}\right)
=\left(\left({\frac{\gamma(\pi)}{\pi}}\right)^{(p-1)w},1,\ldots,1\right)$$
if and only if $g_0(\pi^{p^f})/g_0(\pi)=(\gamma(\pi)/\pi)^{(p-1)w}
=(\gamma(\pi)/\pi\overline{\chi}(\gamma))^{(p-1)w}$ (the order of
$\overline{\chi}$ being $p-1$) and $g_1(\pi)=g_2(\pi^p), \ldots,
g_{f-2}(\pi)=g_{f-1}(\pi^p), g_{f-1}(\pi)=g_0(\pi^p)$. Thus, to get
$g_i$'s satisfying the above identity we just need to define
$g_0(\pi)$ such that
$g_0(\pi^{p^f})/g_0(\pi)=(\gamma(\pi)/\pi\overline{\chi}(\gamma))^{(p-1)w}$. If we set
$g_0(\pi)=\alpha_\gamma\lambda_\gamma(\pi)^w$ with $\alpha_\gamma
\in \F^\times$, we have
$g_0(\pi^{p^f})/g_0(\pi)
=\lambda_\gamma(\pi^{p^f})^w/\lambda_\gamma(\pi)^w
=\lambda_\gamma(\pi)^{w(p^f-1)}
=(\gamma(\pi)/\pi\overline{\chi}(\gamma))^{(p-1)w}$. Conversely
if $g_0'(\pi) \in \E_{K,F}^\times$ satisfies 
$$g_0'(\pi^{p^f})/g_0'(\pi) = (\gamma(\pi)/\pi\overline{\chi}(\gamma))^{(p-1)w}
 = g_0'(\pi^{p^f})/g_0'(\pi),$$
then $h(\pi) = g_0'(\pi)g_0^{-1}(\pi)$ satisfies $h(\pi^{p^f}) = h(\pi)$
and is therefore constant.  Thus we see that the identity implies
that $g_0(\pi)$ has the required form.

Since $G_{\gamma\gamma'}=G_\gamma\gamma(G_\gamma')$, the map $\gamma
\mapsto \alpha_\gamma$ must define a character $\Gamma \rightarrow
\F^\times$, from which we conclude that $\alpha_\gamma =
\overline{\chi}(\gamma)^{j_0}$ for some $0\le j_0 < p-1$.
Letting $u=(\pi,\pi^{p^{f-1}},\pi^{p^{f-2}},\ldots,\pi^p)$, we have
$$\begin{aligned}
{\frac{\varphi(u)}{u}}
&=(\pi^{p^f-1},1,\ldots,1),\\
{\frac{\gamma(u)}{u}} &=
\left({\frac{\gamma(\pi)}{\pi}},
\left({\frac{\gamma(\pi)}{\pi}}\right)^{p^{f-1}},
\left({\frac{\gamma(\pi)}{\pi}}\right)^{p^{f-2}},
\cdots,\left({\frac{\gamma(\pi)}{\pi}}\right)^p\right) \\
& \equiv (\overline{\chi}(\gamma),\ldots,\overline{\chi}(\gamma)) \mod \pi,
\end{aligned}$$
so replacing $e$ by $u^{-j}e$ for some $j\equiv j_0\bmod p-1$
gives $M=\E_{K,F}e$ with
$$\begin{aligned}
\varphi(e) &=(C\pi^{(p-1)w}, 1, \ldots, 1)e, \\
\gamma(e) &=(\lambda_\gamma(\pi)^w, \lambda_\gamma(\pi^p)^w, \ldots, \lambda_\gamma(\pi^{p^{f-1}})^w)e
\end{aligned}$$
where $0 \le w < p^f-1$.
Write $w=c_0+c_1p+\cdots+c_{f-1}p^{f-1}$ with $0 \le c_i \le p-1$.
Taking $e'=ue$ with $u=(1, \pi^{(p-1)({c_1+c_2p+\cdots+c_{f-1}p^{p-2}})}, 1, \ldots, 1)$ yields
$$\varphi(e')=(C\pi^{(p-1)c_0}, \pi^{(p-1)(c_1+c_2p+\cdots+c_{f-1}p^{p-2})}, 1, \ldots, 1)e.$$
Doing this successively gives 
$\varphi(e)=(C\pi^{(p-1)c_0},\pi^{(p-1)c_1},\ldots,\pi^{(p-1)c_{f-1}})e$ for some basis $e$.
It's easily checked that those changes of basis
that maintain $G_\gamma \equiv (1,\ldots,1) \mod \pi$ are $e'=u e$
such that $u = (u_0,\ldots,u_{f-1})$ with $(p-1)|\,\val(u_i)$ and that the corresponding
action of $\gamma \in \Gamma$ is given by $\gamma(e)=(\lambda_\gamma^{\sum_0\vec{c}},
\lambda_\gamma^{\sum_1\vec{c}}, \dots,
\lambda_\gamma^{\sum_{f-1}\vec{c}})e$.

Finally, we suppose that $M$ is isomorphic to $M'=\E_{K,F}e'$ with
$$\begin{aligned}\varphi(e') &=P'e'=(C'\pi^{(p-1)c_0'},\pi^{(p-1)c_1'},\ldots,
\pi^{(p-1)c_{f-1}'})\,e', \\
\gamma(e') &=G_\gamma' e'=(\lambda_\gamma^{\sum_0\vec{c'}},
\lambda_\gamma^{\sum_1\vec{c'}}, \ldots,
\lambda_\gamma^{\sum_{f-1}\vec{c'}})\,e'\end{aligned}$$
and determine when the two are isomorphic.  After appropriate changes of bases we can assume that
$$\begin{aligned}\varphi(e) &= Pe =(C\pi^{(p-1)w},1,\ldots,1)e,\\
\varphi(e') &= P'e'=
(C'\pi^{(p-1)w'},1,\ldots,1)e'\end{aligned}$$
where $w=\sum_0\vec{c}$ and $w' = \sum_0\vec{c'}$ satisfy $0\le w,w'\ < p^f-1$.

Suppose that $u =(u_0, \cdots, u_{f-1}) \in \E_{K,F}^\times$ is such that
$P'=(\varphi(u)/u) P$ and $G'_\gamma = (\gamma(u)/u) G_\gamma$ for
all $\gamma \in \Gamma$.   Then $\gamma(u_0)/u_0 \equiv 1 \bmod \pi\F[[\pi]]$,
so $(p-1)|{\rm val}_\pi(u_0)$.  It follows that
$u=(u_0(\pi),u_0(\pi^{p^{f-1}}), \ldots,
u_0(\pi^p))$ with $u_0(\pi)=u_0'(\pi)\pi^{(p-1)j}$ for some $u_0'(\pi) \in
\F[[\pi]]^\times$ and $j \in \Z$, in which case we have
$$\varphi(u)/u=(\pi^{(p-1)(p^f-1)j}u_0'(\pi)^{p^f-1},1,\ldots,1).$$ 
Thus, we conclude
that $M$ and $M'$ are isomorphic if only if $C=C'$ and $\sum c_ip^i
\equiv \sum c_i'p^i \mod p^f-1$. 

The last assertion is clear.
\end{proof}

We denote the module defined in the proposition by $M_{C\vec{c}}=M_{C(c_0,\ldots,c_{f-1})}$.
We simply write $M_{\vec{c}}$ for $M_{C\vec{c}}$ if $C=1$.
We also put
$$\begin{aligned}\kappa_\varphi(M_{C\vec{c}}) &=\kappa_\varphi(C,\vec{c})
=(C\pi^{(p-1)c_0},\pi^{(p-1)c_1},\ldots,\pi^{(p-1)c_{f-1}}), \\
\kappa_\gamma(M_{C\vec{c}})
&=\kappa_\gamma(C,\vec{c})=(\lambda_\gamma^{\sum_0\vec{c}},
\lambda_\gamma^{\sum_1\vec{c}}, \ldots,
\lambda_\gamma^{\sum_{f-1}\vec{c}}), \end{aligned}$$ and write
$\Sigma_l$ for $\Sigma_l \vec{c}$ where $c_i$'s are understood.

\subsection{Lifts in characteristic zero.}
We now construct rank one Wach modules over $\A_{K, F}^+$ 
following Dousmanis \cite[\S 2]{Dou07} and check that these
reduce modulo $\varpi_F$ to the $(\varphi,\Gamma)$-modules
$M_{C\vec{c}}$ over $\E_{K,F}$.

Let $q_1=q=\varphi(\pi)/\pi, q_n=\varphi^{n-1}(q) \in
\Z_p[[\pi]]$ and let $\Lambda_f=\prod_{j\ge 0} q_{1+jf}/p,
\Lambda_\gamma={\frac{\Lambda_f}{\gamma(\Lambda_f)}} \in \Q[[\pi]]$.
One then has that $\Lambda_f \in 1+\pi\Q_p[[\pi]]$ and
$\Lambda_\gamma \in 1+\pi\Z_p[[\pi]]$.

Suppose we want to construct a rank one Wach module $N=\A_{K,F}^+e$
such that
$$\begin{aligned}
\varphi(e) &=(\tilde{C}q^{c_0},q^{c_1},\ldots,q^{c_{f-1}})e,\\
\gamma(e) &=(g_0(\pi),\ldots,g_{f-1}(\pi))e
\end{aligned}$$ if $\gamma \in \Gamma$, where $\tilde{C} \in \CO_F^\times$
is any lift of $C\in\F^\times$ and each $g_i(\pi)=g_{\gamma,i}(\pi) \in \CO_F[[\pi]]$ depends on
$\gamma \in \Gamma$. Commutativity of the actions of $\varphi$ and
$\Gamma$ amounts to the following identities:
$$\begin{aligned}
\gamma(q)^{c_0}g_0(\pi) &= q^{c_0}\varphi(g_1(\pi)),\\
\gamma(q)^{c_1}g_1(\pi) &= q^{c_1}\varphi(g_2(\pi)),\\
&\vdots\\
\gamma(q)^{c_{f-2}}g_{f-2}(\pi) &= q^{c_{f-2}}\varphi(g_{f-1}(\pi)),\\
\gamma(q)^{c_{f-1}}g_{f-1}(\pi) &= q^{c_{f-1}}\varphi(g_0(\pi)).
\end{aligned}$$
Thus, we are looking for a solution $g_i(\pi)$ for each $\gamma$ of
the equation
$$g_0(\pi)=\left({\frac{q}{\gamma(q)}}\right)^{c_0}\varphi\left({\frac{q}{\gamma(q)}}\right)^{c_1}
\varphi^2\left({\frac{q}{\gamma(q)}}\right)^{c_2}\cdots
\varphi^{f-1}\left({\frac{q}{\gamma(q)}}\right)^{c_{f-1}}\varphi^f(g_0(\pi)).$$
It is straightforward to check that
$$g_0(\pi)= \Lambda_\gamma^{c_0}\varphi(\Lambda_\gamma)^{c_1}
\varphi^2(\Lambda_\gamma)^{c_2}\cdots
\varphi^{f-1}(\Lambda_\gamma)^{c_{f-1}}$$ gives the unique solution
which is $\equiv 1$ modulo $\pi$, and that the remaining
$g_i(\pi)$'s are uniquely determined by
$$\begin{aligned}
g_1(\pi) 
&=\left({\frac{q}{\gamma(q)}}\right)^{c_1}\varphi\left({\frac{q}{\gamma(q)}}\right)^{c_2} \cdots \varphi^{f-1}\left({\frac{q}{\gamma(q)}}\right)^{c_{f-1}}\varphi^{f-1}(g_0(\pi)), \\
& \vdots \\
g_{f-2}(\pi) 
&=\left({\frac{q}{\gamma(q)}}\right)^{c_{f-2}}\varphi\left({\frac{q}{\gamma(q)}}\right)^{c_{f-1}}\varphi^2(g_0(\pi)), \\
g_{f-1}(\pi) &=\left({\frac{q}{\gamma(q)}}\right)^{c_{f-1}}\varphi(g_0(\pi)).
\end{aligned}$$

Dousmanis \cite[\S 6]{Dou07} shows that $N=\A_{K,F}^+e$ endowed with
the actions of $\varphi$ and $\Gamma$ described above defines a Wach
module over $\A_{K,F}^+$ which we denote by $N_{\tilde{C}\vec{c}}$.
Furthermore, $(N_{\tilde{C}\vec{c}}/\pi N_{\tilde{C}\vec{c}}) \otimes_{\A_{K,F}^+}
\B_{K,F}^+$ is a filtered $\varphi$-module corresponding to a positive
character $G_K \to \F^\times$ with labeled Hodge-Tate weights
$(c_{f-1},c_0,c_1,\ldots,c_{f-2})$. One checks the
following by direct computation.

\begin{prop} We have an isomorphism
$M_{C\vec{c}}\simeq N_{\tilde{C}\vec{c}} \otimes_{\A_{K,F}^+} \E_{K,F}$ of
$(\varphi, \Gamma)$-modules over $\E_{K,F}$.
\end{prop}

Combined with Lemma~3.8 of \cite{BDJ05}, we obtain the following, where
$\omega_\tau$ denotes the fundamental character associated to $\tau$
(i.e., $\omega_\tau : I_K \to \F^\times$ is defined by composing $\tau$ with
the homomorphism $I_K \to k^\times$ obtained from local class field theory, with
the convention that uniformizers correspond to geometric Frobenius elements).
\begin{cor} \label{cor:fun} If $\psi: G_K \rightarrow \F^\times$ is the character 
defined by the action on $\V(M_{C\vec{c}})$, then $\psi|_{I_K} =
\prod_{\tau\in S} \omega_\tau^{-c_{\tau\circ\varphi^{-1}}}$.
\end{cor}

\section{Bases for the space of extensions}
We will assume $p > 2$ for the rest of the paper except in \S 6.3 and \S 7.
We fix a topological generator $\eta$ of the pro-cyclic group $\Gamma = \Gamma_K$,
and set $\xi=\eta^{p-1}$, so that $\xi$ topologically generates $\Gamma_1$.

Given $C \in \F^\times$ and $\vec{c}=(c_0,\ldots,c_{f-1}) \in \{0,1,\ldots, p-1\}^S$
with some $c_i < p-1$, we are going to parametrize the space of extension classes
$\Ext^1(M_{0}, M_{C\vec{c}})$ in the category of
\'etale $(\varphi, \Gamma)$-modules over $\E_{K,F}$. 
Here $M_0$ denotes the \'etale $(\varphi,\Gamma)$-module $\E_{K,F}$ with the
usual action of $\varphi$ and $\Gamma$, so $M_0 = M_{1,\vec{0}}$ corresponds
to the trivial character $G_K \to \F^\times$.
Recall that $M_{C\vec{c}}$, $\kappa_\varphi(M_{C\vec{c}})$ and 
$\kappa_{\gamma}(M_{C\vec{c}})$ were defined at the end of \S\ref{subsec:param};
since $M_{C\vec{c}}$ will be fixed in this section, we denote these simply
$\kappa_\varphi$ and $\kappa_\gamma$.

We start by noticing
that there is an $\F$-linear isomorphism
$$\beta: H/H_0 \rightarrow \Ext^1(M_{0}, M_{C\vec{c}}),$$ where $H$ is the subgroup of 
$\E_{K,F}\times \{\Gamma \to \E_{K,F} \}$
consisting of elements $(\mu_\varphi, (\mu_\gamma)_{\gamma
\in \Gamma})$  such that $\gamma \mapsto \mu_\gamma$ is continuous and
satisfies\footnote{We will frequently use the following notation:
for an element $\kappa$ and a ring endomorphism $\psi$ of $\E_{K,F}$,
we denote by $\kappa\psi - 1$ the $\F$-linear endomorphism of $\E_{K,F}$
defined by $(\kappa\psi-1)(x) = \kappa\psi(x) - x$.  We do the same for
$\F((\pi))$ in place of $\E_{K,F}$.  Thus for example, if $\Sigma,s\in\Z$,
then $(\lambda_\gamma^\Sigma \gamma - 1 ) (\pi^s)$ denotes
$\lambda_\gamma(\pi)^{\Sigma} \cdot \gamma(\pi^{s})-\pi^{s}$.}
$$\begin{aligned}
(\dagger)\qquad\  &\, (\kappa_\varphi\varphi-1)(\mu_\gamma)
=(\kappa_\gamma\gamma-1)(\mu_\varphi) \,\, \forall \gamma \in
\Gamma,\\
(\ddagger)\qquad\  &\, \mu_{\gamma\gamma'} =
\kappa_\gamma\gamma(\mu_{\gamma'})+\mu_\gamma \,\, \forall
\gamma,\gamma' \in \Gamma,
\end{aligned}
$$ and $H_0=\{
(\kappa_\varphi\varphi(b)-b,
(\kappa_\gamma\gamma(b)-b)_{\gamma\in\Gamma})
 | \, b\in \F((\pi))^S\} \subset H$.

We call elements of $H$ {\it cocyles} and those of $H_0$ {\it coboundaries}.
The map $\beta$ is defined as follows:
given a cocycle $\mu = (\mu_\varphi, (\mu_\gamma)_{\gamma
\in \Gamma}) \in H$, we define an extension
$$0 \to M_{C\vec{c}} \to E \to M_0 \to 0$$
basis $\{e,e'\}$ such that the action 
$\varphi$ and $\gamma \in \Gamma$ are given by the matrices $P= \left(
\begin{matrix}
\kappa_\varphi & \mu_\varphi \\
0  &   1
\end{matrix}
\right)$ and $G_\gamma= \left(
\begin{matrix}
\kappa_\gamma & \mu_\gamma \\
0  &   1
\end{matrix}
\right)$. It is straightforward to check that the
matrices $P$ and $G_\gamma$ define an extension if and only if
$\mu \in H$, that every extension arises this way,
and that a change of basis for an extension $E$
corresponds to adding an element of $H_0$ to $\mu$.
If $\mu \in H$, then we write $[\mu]$ for the corresponding
extension class $\beta(\mu)$.

By Corollary~\ref{equiv-labeled}, we get an isomorphism $\Ext^1(M_{\vec{0}},
M_{C\vec{c}}) \simeq H^1(K, \F(\psi))$ where $\psi:
G_K \rightarrow \F^\times$ is the character defined by the action
on $\V(M_{C\vec{c}})$.

\begin{lem}
Via Corollary~\ref{equiv-labeled}, $M_{\vec{0}}$ corresponds to the
trivial character and $M_{\overrightarrow{p-2}}$ to the mod $p$
cyclotomic character.
\end{lem}

\begin{proof} The assertion is clear for the trivial character.
The mod $p$ cyclotomic character factors as $G_K \rightarrow
\Z_p^\times \rightarrow \F_p^\times \hookrightarrow \F^\times$ where
the arrow in the middle is the reduction mod $p$. If $T=\Z_p(1)$,
its Wach module is given by $\N(\Z_p(1)) = \A_K^+e$ where
$\varphi(e)={\frac{\pi}{\varphi(\pi)}}e \,\,{\rm and}\,\,
\gamma(e)={\frac{\chi(\gamma)\pi}{\gamma(\pi)}}e \,\,{\rm if}\,\,
\gamma \in \Gamma$ (cf. \cite[Appendice A]{Ber04}). Working modulo $p$ and
extending scalars to $\F$ we see that the \'etale $(\varphi,
\Gamma)$-module over $\E_{K,F}$ corresponding to the mod $p$ cyclotomic
character is given by $M=\E_{K,F}e$ with
$\varphi(e)=\pi^{1-p}e=(\pi^{1-p},\ldots,\pi^{1-p})e$. By a change
of basis $e'=ue$ with $u=(\pi^{p-1},\ldots,\pi^{p-1})$, we get
$M\simeq M_{\overrightarrow{p-2}}$.
\end{proof}

Since $$\dim_\F H^1 (K,\F(\psi))
=\left\{\begin{array}{lll} f+1 & {\rm if} &  \psi=1 \,{\rm or}\, \overline{\chi} \\
                           f & {\rm if} & \psi \not\in \{1, \overline{\chi}\},
\end{array} \right.$$
we have
$$\dim_\F \Ext^1(M_{\vec{0}}, M_{C\vec{c}})
=\left\{\begin{array}{lll} f+1 & {\rm if} \,\, C=1, \,\,{\rm and}\,\,
\vec{c}=\vec{0} \,\,{\rm or}\,\, \vec{c}=\overrightarrow{p-2}
\\ f & {\rm otherwise}.
\end{array} \right.$$

We are about to define elements
$B_0,\ldots,B_{f-1} \in H $ such that the associated extension classes
form a basis for $\Ext^1(M_{\vec{0}}, M_{C\vec{c}})$ except for the
two cases where $C=1, \vec{c}=\vec{0}$ or $C=1,
\vec{c}=\overrightarrow{p-2}$, for which a separate treatment will
be given in \S 6. Thanks to the isomorphism $\beta$ we only need to define
$\mu_\varphi$ and $\mu_\gamma$'s satisfying the desired
properties $(\dagger)$ and $(\ddagger)$. According to whether the
parameter $c_i$ is equal to $p-1$ or not the extension $B_i$ is
constructed in a slightly different manner.

\subsection{Construction of $B_i$ when $c_i<p-1$}
Recall that we have fixed a topological generator $\eta$ for $\Gamma$,
and we let $\xi = \eta^{p-1}$, a topological generator for $\Gamma_1$.
\label{subsec:B_i1}
\begin{lem}\label{delta}
Suppose that $\Sigma, s \in \Z$, and $v = v_p(\Sigma+s(p^f-1)/(p-1)) < \infty$.
Then
$$(\lambda_\eta^\Sigma \eta - 1 ) (\pi^s) 
\in ({\overline{\chi}(\eta)}^s-1)\pi^s+
\overline{s_v}{\frac{{\overline{\chi}(\eta)}^s({\overline{\chi}(\eta)}-1)}{2}}\pi^{s+
p^v}+\pi^{s+2p^{v}}\F_p[[\pi^{p^v}]],$$ where $\Sigma+s(p^f-1)/(p-1)
= \sum_{j\ge v}s_jp^j$.
\end{lem}

\begin{proof}
It is easy to see that in
$\F_p[[\pi]]/\pi^{p-1}$, we have
$$\lambda_\eta = \lambda_\eta^{\frac{p^f-1}{p-1}} =
{\frac{\eta(\pi)}{\overline{\chi}(\eta)\pi}} =
\overline{\chi}(\eta)^{-1}\sum_{j=1}^{d_0-1}
\overline{\frac{d_0!}{j!(d_0-j)!}}\pi^{j-1} 
= 1+\sum_{j=2}^{d_0-1}\overline{\frac{d_0!}{d_0j!(d_0-j)!}}\pi^{j-1},$$ 
where $\chi(\eta)=\sum_{j\ge
0}d_jp^j \in \Z_p^\times$. Noting that, if $s \in \Z$, then
$$(\lambda_\eta ^{\Sigma}\eta-1)(\pi^s)
=\left(\overline{\chi}(\eta)^s\lambda_\eta^{\Sigma}\cdot
\left({\frac{\eta(\pi)}{\overline{\chi}(\eta)\pi}}\right)^s
-1\right)\pi^s,$$
and the result follows as
$$\begin{aligned} \overline{\chi}(\eta)^s\lambda_\eta^{\Sigma} \cdot \left({\frac{\eta(\pi)}{\overline{\chi}(\eta)\pi
}}\right)^s -1 &= \overline{\chi}(\eta)^s\lambda_\eta^{\Sigma+s(p^f-1)/(p-1)}-1 \\
&= \overline{\chi}(\eta)^s\lambda_\eta^{\sum_{j\ge v}s_jp^j}-1 \\
&=\overline{\chi}(\eta)^s\lambda_\eta(\pi^{p^v})^{s_v}\lambda_\eta(\pi^{p^{v+1}})^{s_{v+1}}\cdots
 -1 \\
& \equiv (\overline{\chi}(\eta)^s-1)+
\overline{\chi}(\eta)^s\left(
\left(1+{\frac{\overline{\chi}(\eta)-1}{2}}\pi^{p^v}\right)^{s_v}-1\right) \\
& \equiv (\overline{\chi}(\eta)^s-1)+
\overline{s_v}{\frac{\overline{\chi}(\eta)^s(\overline{\chi}(\eta)-1)}{2}}\pi^{p^v} \mod \pi^{2p^v}
\end{aligned}$$
and
$$\overline{\chi}(\eta)^s\lambda_\eta^\Sigma \cdot \left({\frac{\eta(\pi)}{\overline{\chi}(\eta)\pi}}\right)^s-1
-\left((\overline{\chi}(\eta)^s-1)+\overline{s_v}{\frac{\overline{\chi}(\eta)^s(\overline{\chi}(\eta)-1)}{2}}\pi^{p^v}
\right) \in \pi^{2p^v}\F_p[[\pi^{p^v}]].
$$
\end{proof}

We note the following lemma, whose straightforward proof we omit:
\begin{lem}\label{gamma n} If $n\ge 1$, $\gamma \in \Gamma_n$ and 
$\chi(\gamma) \equiv 1 + zp^n\bmod p^{n+1}$,
then $\lambda_{\gamma} \equiv 1 + z\pi^{p^n-1} + z\pi^{p^{n}} \bmod \pi^{2p^n-2}$.
\end{lem}

\begin{lem}\label{gamma} Let $\chi(\xi) \equiv 1+z p
\mod p^2$ with $0 < z \le p-1$ and let $\Sigma, s\in \Z$.
If $v = v_p(\Sigma+s(p^f-1)/(p-1)) < \infty$,
then $$(\lambda_\xi ^{\Sigma}\xi - 1 ) (\pi^s) \in \overline{
s_vz}(\pi^{s+(p-1)p^v}+\pi^{s+p^{v+1}}) +
\pi^{s+2p^v(p-1)}\F_p[[\pi^{p^v}]],$$ where $\Sigma+s(p^f-1)/(p-1) =
\sum_{j\ge v}s_jp^j$.
\end{lem}

\begin{proof}
By Lemma~\ref{gamma n}, we have
$$\lambda_\xi \equiv {\frac{\xi(\pi)}{\pi\overline{\chi}(\xi)}}
\equiv 1+z\pi^{p-1}+z\pi^p \mod \pi^{2p-2}.$$ Noting that if $s \in \Z$, then
$$(\lambda_\xi^{\Sigma}\xi-1)(\pi^s) =
\left(\lambda_\xi^{\Sigma} \cdot \left(
{\frac{\xi(\pi)}{\pi\overline{\chi}(\xi)}}\right)^s-1\right)\pi^s,$$
and the result follows as
$$\begin{aligned}
\lambda_\xi^{\Sigma} \cdot \left(
{\frac{\xi(\pi)}{\pi\overline{\chi}(\xi)}}\right)^s-1
&= \lambda_\xi^{\Sigma+s(p^f-1)/(p-1)}-1 \\
&= \lambda_\xi(\pi^{p^v})^{s_v}\lambda_\xi(\pi^{p^{v+1}})^{s_{v+1}}
\cdots-1 \\
&\equiv \lambda_\xi(\pi^{p^v})^{s_v}-1 \mod \pi^{(p-1)p^{v+1}}\\
&\equiv \overline{s_vz}(\pi^{(p-1)p^v}+ \pi^{p^{v+1}}) \mod
\pi^{2(p-1)p^v}
\end{aligned}$$
and
$$\lambda_\xi^{\Sigma} \cdot \left(
{\frac{\xi(\pi)}{\pi\overline{\chi}(\xi)}}\right)^s-1-\overline{s_vz}(\pi^{(p-1)p^v}+
\pi^{p^{v+1}}) \in \pi^{2(p-1)p^v}\F_p[[\pi^{p^v}]].$$
\end{proof}

We now assume $c_i < p-1$ and construct an element $B_i \in H$.
(For the case $c_i =p-1$, we will need to use a modified construction 
described in \S\ref{subsec:B_i2}.)

Suppose for the moment that we have
successfully defined $B_i$ with $\mu_\varphi(B_i)$ of the form
$(0,\ldots,0,H_i(\pi),0,\ldots,0)$, $H_i(\pi)$ being the $i$th
component. For each $\gamma \in \Gamma$, by the condition $(\dagger)$ there should exist
$\mu_\gamma(B_i)=(G_0(\pi),\ldots,G_{f-1}(\pi))$ such
that $$(\kappa_\varphi\varphi-1)(\mu_\gamma(B_i))
=(\kappa_\gamma\gamma-1)(\mu_\varphi(B_i)),$$ i.e.,
$$\begin{aligned} & (C\pi^{(p-1)c_0}G_1(\pi^p)-G_0(\pi),
\pi^{(p-1)c_1}G_2(\pi^p)-G_1(\pi),\ldots,\pi^{(p-1)c_{f-1}}G_0(\pi^p)-G_{f-1}(\pi))
\\ = \, & (0,\ldots,0, (\lambda_\gamma^{\Sigma_i}\gamma-1)
(H_i(\pi)),0,\ldots,0). \end{aligned}$$ This is true if and only if
$$\begin{aligned}(C\pi^{(p-1)\Sigma_i}\Phi-1)(G_i(\pi))
& =(\lambda_\gamma^{\Sigma_i}\gamma-1)(H_i(\pi)), \\
G_{i+1}(\pi) &=\pi^{(p-1)c_{i+1}}G_{i+2}(\pi^p), \\ & \vdots \\
G_{f-1}(\pi) &=\pi^{(p-1)c_{f-1}}G_0(\pi^p), \\
G_0(\pi) &=C\pi^{(p-1)c_0}G_1(\pi^p), \\
G_1(\pi) &=\pi^{(p-1)c_1}G_2(\pi^p), \\ & \vdots \\
G_{i-1}(\pi) &=\pi^{(p-1)c_{i-1}}G_i(\pi^p), \end{aligned}$$ where
$\Phi(G(\pi))=G(\pi^{p^f})$. Except for the case $C=1$,
$\vec{c}=\vec{0}$, the map $C\pi^{(p-1)\Sigma_i}\Phi-1$ defines a
bijection $\F[[\pi]]\rightarrow\F[[\pi]]$. So the trick is to find
$H_i(\pi)$ so that we have
$(\lambda_\gamma^{\Sigma_i}\gamma-1)(H_i(\pi)) \in \F[[\pi]]$. The
corresponding $G_i(\pi)$ and so $\mu_\gamma(B_i)$'s are automatically and
uniquely determined by the bijectivity.  Moreover since the bijection
$C\pi^{(p-1)\Sigma_i}\Phi-1$ on the compact Hausdorff
space $\F[[\pi]]$ is continuous, so is its inverse, and it follows
that $\gamma \mapsto \mu_\gamma(B_i)$ is continuous.

To find such $H_i(\pi)$, we observe via Lemma \ref{delta} 
that\footnote{Note that if $c_i = p-1$, then $\Sigma_i \equiv - 1 \bmod p$ and $v \ge 1$
in the notation of Lemma~\ref{delta}, which is why we need to modify the construction
in that case.}
$$\val(\lambda_\eta^{\Sigma_i}\eta-1)(\pi^{1-p})=2-p \,\,{\rm
and} \,\, \val(\lambda_\eta^{\Sigma_i}\eta-1)(\pi^s)=s \,\,{\rm
if}\,\, 2-p\le s\le-1.$$ 
  Then there exist unique
$\epsilon_{2-p},\ldots,\epsilon_{-1}\in \F_p$ such that
$$\left(\lambda_\eta^{\Sigma_i}\eta-1\right)(\pi^{1-p}+\epsilon_{2-p}\pi^{2-p}
+\ldots+\epsilon_{-1}\pi^{-1}) \in\F[[\pi]].$$
We set
$$H_i(\pi)= \pi^{1-p} +h_i(\pi)= \pi^{1-p}+\epsilon_{2-p}\pi^{2-p}
+\ldots+\epsilon_{-1}\pi^{-1}$$ and claim that
$$(\lambda_{\gamma}^{\Sigma_i}{\gamma}-1)(H_i(\pi)) \in\F[[\pi]]$$ for all $\gamma \in
\Gamma$. Note that by Corollary \ref{gamma n}
we have $\lambda_{\gamma_1} \equiv 1 \mod \pi^{p-1}$, so that
$(\lambda_{\gamma_1}\gamma_1-1)(\pi^s) \equiv 0 \mod \pi^{p-1}$  for all $1-p \le s
\le -1$ if $\gamma_1 \in \Gamma_1$.
Since any given $\gamma \in \Gamma$ can be written as $\gamma = \eta^m
\gamma_1$ where $m\in \N_{\ge 0}$ and $\gamma_1 \in \Gamma_1$, we have
$$(\lambda_{\gamma}^{\Sigma_i}{\gamma}-1)(H_i(\pi)) \in\F[[\pi]]$$ by the following.

\begin{lem}\label{val} Let $\Sigma$ and $v$ be integers and $H(\pi) \in
\F((\pi))$. For any $\gamma, \gamma' \in \Gamma$, if the valuations (in $\pi$)
of $(\lambda_\gamma^{\Sigma}\gamma-1)(H(\pi))$ and
$(\lambda_{\gamma'}^{\Sigma}\gamma'-1)(H(\pi))$ are $\ge v$, so
is that of $(\lambda_{\gamma\gamma'}^{\Sigma} \gamma\gamma'-1)(H(\pi))$.
\end{lem}

\begin{proof} If both
$\lambda_\gamma^{\Sigma}\gamma(H(\pi))-H(\pi)$ and
$\lambda_{\gamma'}^{\Sigma}\gamma'(H(\pi))-H(\pi)$ are in
$\pi^v\F[[\pi]]$, then
$$\begin{aligned}(\lambda_{\gamma\gamma'}^{\Sigma} \gamma\gamma'-1)(H(\pi))
&
=\left({\frac{\gamma\gamma'(\pi)}{\pi\overline{\chi}(\gamma\gamma')}}\right)
^{{\frac{p-1}{p^f-1}} \Sigma}\gamma\gamma'(H(\pi))-H(\pi) \\
&
=\left(\gamma\left({\frac{\gamma'(\pi)}{\pi\overline{\chi}(\gamma')}}\right)
{\frac{\gamma(\pi)}{\pi\overline{\chi}(\gamma)}}\right)
^{{\frac{p-1}{p^f-1}} \Sigma}\gamma(\gamma'(H(\pi)))-H(\pi) \\ &
=\lambda_\gamma^{\Sigma}
\gamma\left(\lambda_{\gamma'}^{\Sigma}\gamma'(H(\pi))-H(\pi)\right)
+\lambda_\gamma^{\Sigma}\gamma(H(\pi))-H(\pi) \in
\pi^v \F[[\pi]].\end{aligned}$$
\end{proof}

So far we have defined $\mu_\varphi=\mu_\varphi(B_i)$ and
$\mu_\gamma=\mu_\gamma(B_i)$ satisfying the condition ($\dagger$),
and need to verify the condition ($\ddagger$). It is easily checked
that if $\gamma, \gamma' \in \Gamma$, both $\mu_{\gamma\gamma'}$ and
$\mu'_{\gamma\gamma'}=
\kappa_\gamma\gamma(\mu_{\gamma'})+\mu_\gamma$ satisfy ($\dagger$).
Since when we fix $\mu_\varphi$ the solution of ($\dagger$) for
$\gamma\gamma'$ is unique (by the bijectivity of the map $C\pi^{\Sigma_i}\Phi-1$), we must have ($\ddagger$)
$\mu_{\gamma\gamma'} =
\kappa_\gamma\gamma(\mu_{\gamma'})+\mu_\gamma$.

\subsection{Construction of $B_i$ when $c_i=p-1$}
\label{subsec:B_i2}
\begin{prop}\label{trick}
If $c_i=p-1$ and $c_{i+1} \neq p-2$, we have 
$$(\lambda^{\Sigma_i}_\eta\eta-1)\left(\pi^{1-p^2} + h_i''(\pi)
+ \epsilon(\pi^{2-2p} +h_i'(\pi) +h_i(\pi)) \right)  \in
\F[[\pi]]$$ for some unique $\epsilon \in \F^\times$ and some unique Laurent polynomials
$h_i''(\pi)  =\sum_{s=1}^{p-2}\epsilon_s''\pi^{1-p^2+sp}$,
$h_i'(\pi)  =\sum_{s=1}^{p-2}\epsilon_s'\pi^{2-2p+s}$,
$h_i(\pi)  =\sum_{s=1}^{p-2}\epsilon_s\pi^{1-p+s} \in \F[\pi][1/\pi]$.
\end{prop}

\begin{proof} By Lemma \ref{delta} we have,
$$(\lambda_\eta^{\Sigma_i}\eta-1)(\pi^{1-p^2}) \in \,\F^\times\pi^{1-p^2+p}+\pi^{1-p^2+2p}\F[[\pi^p]]$$
and $$(\lambda_\eta^{\Sigma_i}\eta-1)(\pi^{1-p^2+sp}) \in  \,
\F^\times\pi^{1-p^2+sp}+\sum_{j=1}^{p-s}\F\pi^{1-p^2+(s+j)p} +\F[[\pi]]$$ for $1\le s \le p-2$.
Thus there exist unique $\epsilon_s'', \nu' \in \F$ such that $$(\lambda^{\Sigma_i}_\eta\eta-1)\left(\pi^{1-p^2}+ \sum_{s=1}^{p-2}\epsilon_s''\pi^{1-p^2+sp}\right)
\in \nu'\pi^{1-p}+\F[[\pi]].$$
Similarly, there exist unique $\epsilon_s', \epsilon_s, \nu \in F$ such that
$$(\lambda^{\Sigma_i}_\eta\eta-1)\left(\pi^{2-2p}+ \sum_{s=1}^{p-2}\epsilon_s'\pi^{2-2p+s} +
\sum_{s=1}^{p-2}\epsilon_s\pi^{1-p+s} \right) \in \nu\pi^{1-p}+\F[[\pi]].$$
Put $h_i''(\pi)  =\sum_{s=1}^{p-2}\epsilon_s''\pi^{2-p^2+sp},
h_i'(\pi)  =\sum_{s=1}^{p-2}\epsilon_s'\pi^{2-2p+s},
h_i(\pi)  =\sum_{s=1}^{p-2}\epsilon_s\pi^{1-p+s}.$

The point then is to show that both $\nu$ and $\nu'$ are nonzero so that
$$ (\lambda^{\Sigma_i}_\eta\eta-1)\left(\pi^{1-p^2}+ h_i''(\pi)\right)
+  \epsilon\,(\lambda^{\Sigma_i}_\eta\eta-1)\left(\pi^{2-2p}+ h_i'(\pi)+h_i(\pi) \right) \in \F[[\pi]]$$
where $\epsilon= -\nu'/\nu \in\F -\{0\}$. 
So let us prove nonvanishing of $\nu'$ and $\nu$.
Suppose $\nu'=0$ so that $\val(\lambda^{\Sigma_i}_\eta\eta-1) \left(\pi^{1-p^2}+h_i''(\pi)\right)\ge 0.$
By Lemma \ref{val}, recalling $\xi=\eta^{p-1}$, it follows that
$\val(\lambda^{\Sigma_i}_{\xi}\xi-1) \left(\pi^{1-p^2}+ h_i''(\pi)\right)\ge 0,$
However, by Lemma \ref{gamma} we have 
${\rm \val} (\lambda^{\Sigma_i}_\eta\eta-1)\left(\pi^{1-p^2}+h_i''(\pi)\right)=1-p.$
Thus $\nu'$ cannot be zero. Similarly, we get $\nu \neq 0$.
\end{proof}

\begin{prop}\label{trick+}
If $c_i=p-1$, and $r \in \{0,\ldots,f-1\}$ is such that $c_{i+1}=\cdots=c_{i+r}=p-2$ and $c_{i+r+1} \neq p-2$,
we have $$(\lambda^{\Sigma_i}_\eta\eta-1)\left( \pi^{1-p^{r+2}} + \sum_{j=0}^{r+1}h_i^{(j)}
+ \sum_{j=0}^{r}\epsilon^{(j)} h_i'^{(j)} \right)  \in \F[[\pi]]$$ for some unique Laurent polynomials 
$$h_i^{(j)}(\pi)  =\sum_{s=1}^{p-2}\epsilon_s^{(j)}\pi^{1-p^{j+1}+sp^j} \,\,(0\le j \le r+1),$$
$$h_i'^{(j)}(\pi)  = \pi^{1+p^j-2p^{j+1}} + \sum_{s=1}^{p-2}\epsilon_s'^{(j)}\pi^{1+p^j-2p^{j+1}+sp^j} \,\,(0\le j \le r)$$
in $\F[\pi][1/\pi]$ with $\epsilon_1^{(r+1)},\epsilon^{(r)}\neq 0$.
\end{prop}

\begin{proof} By Lemma~\ref{delta} (with $v=r+1, s_v=c_{i+r+1}+2$) we get
$$(\lambda_\eta^{\Sigma_i}\eta-1)\pi^{1-p^{r+2}}
\in  \,\F^\times\pi^{1-p^{r+2}+p^{r+1}} + \pi^{1-p^{r+2}+2p^{r+1}}\F[[\pi^{p^{r+1}}]],$$ and
$$(\lambda_\eta^{\Sigma_i}\eta-1)\pi^{1-p^{r+2}+sp^{r+1}} \in  \,
\F^\times\pi^{1-p^{r+2}+sp^{r+1}}+\sum_{t=1}^{p-s-1}\F\pi^{1-p^{r+2}+(s+t)p^{r+1}} +\F[[\pi]]$$
for $1 \le s \le p-2$, so that there exist unique $\epsilon_1^{(r+1)},\ldots,\epsilon_{p-2}^{(r+1)}, \nu^{(r+1)} \in \F$
such that $$(\lambda^{\Sigma_i}_\eta\eta-1)\left(\pi^{1-p^{r+2}}+ \sum_{s=1}^{p-2}\epsilon_s^{(r+1)}\pi^{1-p^{r+2}+sp^{r+1}}\right) \in \nu^{(r+1)}\pi^{1-p^{r+1}}+\F[[\pi]].$$
We set $h_i^{(r+1)}=\sum_{s=1}^{p-2}\epsilon_s^{(r+1)}\pi^{1-p^{r+2}+sp^{r+1}}$.

Again by Lemma~\ref{delta}, we get
$$(\lambda_\eta^{\Sigma_i}\eta-1)\pi^{1-2p^{r+1}+p^r}
\in  \,\F^\times\pi^{1-2p^{r+1}+2p^r} + \pi^{1-2p^{r+1}+3p^r}\F[[\pi^{p^{r}}]],$$
and 
$$(\lambda_\eta^{\Sigma_i}\eta-1)\pi^{1-2p^{r+1}+(1+s)p^{r}} \in  \,
\F^\times\pi^{1-2p^{j+1}+(1+s)p^{j}}+\sum_{t=1}^{p-s-2}\F\pi^{1-2p^{r+1}+(1+s+t)p^{r}}
+\pi^{1-p^{r+1}}\F[[\pi^{p^r}]]$$ 
for $1 \le s \le p-2$, so that there exist unique $\epsilon_1'^{(r)},\ldots,\epsilon_{p-2}'^{(r)}, \nu'^{(r)} \in \F$
such that 
$$(\lambda^{\Sigma_i}_\eta\eta-1)\left(\sum_{s=0}^{p-2}\epsilon_s'^{(r)}\pi^{1+p^r-2p^{r+1}+sp^r} \right) \in
\nu'^{(r)}\pi^{1-p^{r+1}}+\pi^{1-p^{r+1}+p^r}\F[[\pi^{p^r}]],$$
where we have set $\epsilon_0^{'(r)}=1$.

As in the proof of Proposition~\ref{trick},
one can show that both $\nu^{(r+1)}$ and $\nu^{'(r)}$ are not zero, so that
$$(\lambda^{\Sigma_i}_\eta\eta-1)
\left( \pi^{1-p^{r+2}} + h_i^{(r+1)} + \epsilon^{(r)} h_i^{'(r)} \right)
\in \pi^{1-p^{r+1}+p^r}\F[[\pi^{p^r}]]$$
where $\epsilon^{(r)}=-\nu^{(r+1)}/\nu^{'(r)} \neq 0$.
Then again
$$(\lambda^{\Sigma_i}_\eta\eta-1)
\left( \pi^{1-p^{r+2}} + h_i^{(r+1)} + \epsilon^{(r)} h_i^{'(r)} + h_i^{(r)} \right)
\in \F\pi^{1-p^r} + \F[[\pi]]$$
for some $h_i^{(r)}=\sum_{s=1}^{p-2}\epsilon_s^{(r)}\pi^{1-p^{r+1}+sp^r}$.
Iterating the process proves the proposition.
\end{proof}

When $c_i =p-1, c_{i+1}=\cdots=c_{i+r}=p-2, c_{i+r+1} \neq p-2$ we define $$\mu_\varphi(B_i)=(0,\ldots,0,H_i(\pi),0,\ldots,0)$$ where
$$H_i(\pi)= \pi^{1-p^{r+2}} + \sum_{j=1}^{r+1}h_i^{(j)}(\pi)
+ \sum_{j=0}^{r}\epsilon^{(j)} h_i^{'(j)}(\pi)$$
is the $i$th
component. By Proposition~\ref{trick+}, we get
$(\lambda^{\Sigma_i}_\gamma\gamma-1)(\mu_\varphi(B_i))\in \F[[\pi]]$
and then $\mu_\gamma(B_i)$ is determined by bijectivity of the map
$C\pi^{(p-1)\Sigma_i}\Phi-1: \F[[\pi]]\rightarrow \F[[\pi]]$. The
condition $(\ddagger)$ is checked in an analogous fashion as in \S 4.1.

\begin{rem} The cocycle $B_i$ for the case 
$c_i=p-1$, $c_{i+1}=\cdots=c_{i+r}=p-2, c_{i+r+1} \neq p-2$
is cohomologous to a cocycle $B_i'$ defined by
$$\begin{aligned}
\mu_\varphi(B_i') & = \left(\epsilon^{(0)}\pi^{2-2p}\sum_{s=0}^{p-2}\epsilon_s^{'(0)}\pi^{s} 
                      + \pi^{1-p}\sum_{s=1}^{p-2}\epsilon_s^{(0)}\pi^{s} \right)e_i  \\
                  & + \left(\epsilon^{(1)}\pi^{3-3p}\sum_{s=0}^{p-2}\epsilon_s^{'(1)}\pi^{s} 
                    + \pi^{2-2p}\sum_{s=1}^{p-2}\epsilon_s^{(1)}\pi^{s} \right)e_{i+1} \\
                  & \vdots \\
                  & + \left(\epsilon^{(r)}\pi^{3-3p}\sum_{s=0}^{p-2}\epsilon_s^{'(r)}\pi^{s} 
                    + \pi^{2-2p}\sum_{s=1}^{p-2}\epsilon_s^{(r)}\pi^{s} \right)e_{i+r} \\
                  & + \left(\pi^{2-2p}\sum_{s=0}^{p-2}\epsilon_s^{(r+1)}\pi^{s} \right)e_{i+r+1},
\end{aligned}$$
where $\epsilon_0^{'(0)} = \epsilon_0^{'(1)} = \cdots = \epsilon_0^{'(r)} = \epsilon_0^{(r+1)}=1$
and $\epsilon^{(r)}\neq 0$.
See Lemma~\ref{f2.2} for the proof in the case $f=2$.
\end{rem}

\subsection{Linear independence of $B_i$'s} 
Throughout this subsection we assume $C\neq 1$ if $\vec{c}=\vec{0}$,
so that $C\pi^{(p-1)\Sigma_i}\Phi-1 : \F[[\pi]] \rightarrow \F[[\pi]]$
defines a valuation-preserving bijection for all $i \in S$.
From the constructions in \S\S \ref{subsec:B_i1}, \ref{subsec:B_i2} we have at hand the extensions
$[B_0],\ldots,[B_{f-1}] \in \Ext^1(M_{0}, M_{C\vec{c}})$ such that, if $i \in S$,
$$\mu_\varphi(B_i)=(0,\ldots,0, H_i(\pi), \ldots, 0)$$ has $i$th component
$$H_i(\pi) = \pi^{1-p^{r+2}} + \sum_{j=0}^{r+1}h_i^{(j)}(\pi) + \sum_{j=0}^{r}\epsilon^{(j)} h_i^{'(j)}(\pi),$$
where if $c_i\neq p -1$, then we set $r=-1$ and $h_i^{(0)} (\pi) = h_i(\pi)$ was defined in \S \ref{subsec:B_i1},
and if $c_i=p-1$, then $r$ is the least non-negative integer such that $c_{i+r+1}\neq p-2$
and $h_i^{(j)}$, $h_i^{(j)}$ and $\epsilon^{'(j)}$ were defined in \S \ref{subsec:B_i2}.

To prove linear independence of $[B_i]$'s, suppose that
$\beta_0B_0+\cdots+\beta_{f-1}B_{f-1}$ is a coboundary
for some $\beta_0,\ldots,\beta_{f-1} \in \F$. We want to
show that $\beta_0 = \cdots = \beta_{f-1}=0$. By the cyclic nature of
the indexing, it is enough to show that $\beta_{f-1}=0$. Since
$\beta_0\mu_\varphi(B_0)+\cdots+\beta_{f-1}\mu_\varphi(B_{f-1})=(\beta_0H_0(\pi),\ldots,\beta_{f-1}H_{f-1}(\pi)),$
by adding another coboundary, we see that
$$\begin{aligned}
& (\beta_0H_0(\pi)+\beta_1C\pi^{(p-1)c_0}H_1(\pi^p)
+\cdots+\beta_{f-1}C\pi^{(p-1)\sum_{j=0}^{f-2}c_jp^j}H_{f-1}
(\pi^{p^{f-1}}), 0, \ldots, 0) \\
= & (C\pi^{(p-1)c_0}b_1(\pi^p)-b_0(\pi),
\pi^{(p-1)c_1}b_2(\pi)-b_1(\pi), \ldots,
\pi^{(p-1)c_{f-1}}b_0(\pi^p)-b_{f-1}(\pi))\end{aligned}$$
for some $(b_0(\pi),\ldots,b_{f-1}(\pi))\in \F((\pi))^S$.
It follows that
$$\begin{aligned}
& \beta_0H_0(\pi)+\beta_1C\pi^{(p-1)c_0}H_1(\pi^p) +\cdots
+\beta_{f-1}C\pi^{(p-1)\sum_{j=0}^{f-2}c_jp^j}H_{f-1} (\pi^{p^{f-1}}) \\
= & (C\pi^{(p-1)\Sigma_0}\Phi-1)(b_0(\pi))
\end{aligned}$$ and
$$\begin{aligned}
b_1(\pi) & =\pi^{(p-1)c_1}b_2(\pi^p), \\
b_2(\pi) & =\pi^{(p-1)c_2}b_3(\pi^p), \\
& \vdots \\
b_{f-2}(\pi) & =\pi^{(p-1)c_{f-2}}b_{f-1}(\pi^p), \\
b_{f-1}(\pi) & =\pi^{(p-1)c_{f-1}}b_0(\pi^p).
\end{aligned}$$
As the map
$C\pi^{(p-1)\Sigma_0}\Phi-1 : \F[[\pi]]\rightarrow \F[[\pi]]$ is a
bijection, we get a congruence
$$\begin{aligned} & \beta_0H_0(\pi)+\beta_1C\pi^{(p-1)c_0}H_1(\pi^p)
+\cdots+\beta_{f-1}C\pi^{(p-1)\sum_{j=0}^{f-2}c_jp^j}H_{f-1}
(\pi^{p^{f-1}}) \\ \equiv & (C\pi^{(p-1)\Sigma_0}\Phi-1)(b(\pi)) \mod
\F[[\pi]] \end{aligned}$$ for some $b(\pi)=b_{-s}\pi^{-s}+ \sum_{j
=1}^{s-1}b_{-s+j}\pi^{-s+j} \in \F[1/\pi]$ with $s > 0$ and $b_{-s} \neq 0$.
Suppose $\beta_{f-1}\neq 0$ and we will get contradictions.

First assume $c_{f-1}=p-1, c_f = \cdots = c_{f-1+r} = p-2, c_{f+r} \neq p-2$ with $r>0$,
in which case we have
$$H_{f-1}(\pi)=\pi^{1-p^{r+2}} + \sum_{j=0}^{r+1}h_i^{(j)}(\pi)
+ \sum_{j=0}^{r}\epsilon^{(j)} h_i'^{(j)}(\pi).$$
One checks that the lowest degree term (in $\pi$) of the LHS of the congruence is
$$\beta_{f-1}C\pi^{(p-1)\sum_{j=0}^{f-2}c_jp^j}\pi^{(1-p^{r+2})p^{f-1}},$$
so that the valuation of the LHS is $(p-1)(\sum_{j=0}^{f-2}c_jp^j-{(1+p+\cdots+p^{r+1})p^{f-1}})$. On the other
hand, we have three possibilities for the RHS: $(p-1)\Sigma_0-sp^f<-s$,
$-s<(p-1)\Sigma_0-sp^f$ and $(p-1)\Sigma_0-sp^f=-s$.

If $(p-1)\Sigma_0-sp^f<-s$, the leading term of the RHS is $b_{-s}C\pi^{(p-1)\Sigma_0}\pi^{-sp^f}$
and we have $s=(p-1)(2+p+\cdots+p^r)$  and $\beta_{f-1}=b_{-s}$.
Now the term 
$$\beta_{f-1}C \pi^{(p-1)\sum_{j=0}^{f-2}c_jp^j}\epsilon^{(r)}\pi^{(1+p^r-2p^{r+1})p^{f-1}}$$
is alive on the LHS and must match a term on the RHS.
Considering possible matching valuations on the RHS we get either
$$(p-1)\sum_{j=0}^{f-2}c_jp^j+{(1+p^r-2p^{r+1})p^{f-1}} = -t$$
or
$$(p-1)\sum_{j=0}^{f-2}c_jp^j+{(1+p^r-2p^{r+1})p^{f-1}} = (p-1)\Sigma_0-tp^f$$
for some $0<t<s=(p-1)(2+p+\cdots+p^r)$.  The former equation contradicts 
the inequality $t < s$ and the latter implies that $t=2p^r-p^{r-1}+p-2$.
Since $p^f \not| t+(p-1)\Sigma_0$, there must be a term of degree $-t$ on the LHS.
However if $m < r$, then the leading term of $\pi^{(p-1)\sum_{j=0}^{m-1}c_jp^j}H_m(\pi^{p^m})$
has degree $> -t$, and if $m \ge r$, then its terms cannot be congruent to $-t\bmod p^r$,
and we again arrive at a contradiction.

If $-s<(p-1)\Sigma_0-sp^f$, the leading term of the RHS is $-b_{-s}\pi^{-s}$.
Then $(p-1)|s$ and $s(p^f-1)/(p-1) < \Sigma_0 < p^f-1$, so that $s<1$,
which is impossible.

Lastly, if $(p-1)\Sigma_0-sp^f=-s$,
working modulo powers of $p$, we get $s=c_0=\cdots=c_{f-1}=p-1$,
a contradiction.

Now we may assume that $\beta_j=0$ for all $j \in S$ such that $c_j=p-1$ and $c_{j+1}=p-2$.
Suppose now that $c_{f-1} = p - 1$ and $c_0\neq p - 2$.  We then proceed to show that
$\beta_{f-1}=0$ by induction on $m$ where $m\ge 1$ is such that $c_{f-m-1} \neq p-1$
and $c_{f-m}=c_{f-m+1} = \cdots c_{f-1} = p-1$.  We may thus assume that
$\beta_{f-m} = \cdots = \beta_{f-2} = 0$ if $m\ge 2$.  The argument used in the
case $r>0$ then goes through with the following two changes:  1) the induction
hypothesis is used to show that the term 
$$\beta_{f-1}C \pi^{(p-1)\sum_{j=0}^{f-2}c_jp^j}\epsilon^{(0)}\pi^{(2-2p)p^{f-1}}$$
is alive on the LHS, and 2) the equality 
$$(p-1)\sum_{j=0}^{f-2}c_jp^j+{(2-2p)p^{f-1}} = (p-1)\Sigma_0-tp^f$$
immediately gives a contradiction without considering more terms.

Now we may assume that $\beta_j=0$ for all $j \in S$ such that $c_j=p-1$, and
suppose $c_{f-1}<p-1$. The leading term of the LHS then is $$\beta_{f-1}C\pi^{(p-1)\sum_{j=0}^{f-2}c_jp^j}\pi^{(1-p)p^{f-1}}.$$
If $(p-1)\Sigma_0-sp^f<-s$,
$sp^f=(p-1)(\Sigma_0-\sum_{j=0}^{f-2}c_jp^j + p^{f-1})=(p-1)(c_{f-1}+1)p^{f-1}$,
ans so $p|(c_{f-1}+1)$, which is impossible as $0\le c_{f-1}<p-1$.
If $(p-1)\Sigma_0-sp^f \ge -s$, then $-s \le (p-1)(\sum_{j=0}^{f-2}c_jp^j - p^{f-1})
 \le 1-p$, contradicting that $s(p^f-1)/(p-1) \le \Sigma_0 < p^f-1$.

This completes the proof that the $[B_i]$ are linearly independent, hence form a basis
for $\Ext^1(M_0,M_{C\vec{c}})$ (unless $C=1, \vec{c}=\vec{0}$ or $C=1, \vec{c}=\overline{p-2}$).

\section{The space of bounded extensions}
In this section we define bounded extensions, which we will later relate to extensions
arising from crystalline representations.

\subsection{Bounded extensions}

\begin{defn}\label{bounded}
Suppose $A, B \in \F^\times$ and $0 \le a_i, b_i \le p$ with exactly
one of $a_i$ or $b_i$ is zero for each $i \in S$. We say that
an extension (class) $E\in\Ext^1(M_{A\vec{a}},M_{B\vec{b}})$ is {\it
bounded} if there exists a basis for $E$ such that the defining
matrices $P$ and $G_\gamma$ satisfy the following:
\begin{enumerate}
\item $P= \left(
\begin{matrix}
\kappa_\varphi(B,\vec{b}) & * \\
0  &  \kappa_\varphi(A,\vec{a})
\end{matrix}
\right)$ and $G_\gamma= \left(
\begin{matrix}
\kappa_\gamma(B,\vec{b}) & * \\
0  &  \kappa_\gamma(A,\vec{a})
\end{matrix}
\right)$ if $\gamma \in \Gamma$,
\item $P\in {\rm M}_2(\F[[\pi]]^S)$,
\item $G_\gamma - I_2 \in \pi{\rm M}_2(\F[[\pi]]^S)$ if $\gamma \in
\Gamma_1$.
\end{enumerate} Bounded extensions form a subspace, denoted by
$\Ext^1_{\bdd}(M_{A\vec{a}},M_{B\vec{b}})$, of the full space
$\Ext^1(M_{A\vec{a}},M_{B\vec{b}})$ of extensions.
\end{defn}

\begin{rem} Note that the space $\Ext^1_{\bdd}(M_{A\vec{a}},M_{B\vec{b}})$ depends on $\vec{a}$ and $\vec{b}$
and not just on the isomorphism classes of the $(\varphi,\Gamma)$-modules $M_{A\vec{a}}$ and $M_{B\vec{b}}$.
\end{rem}

\begin{lem}
The condition (3) can be replaced by a weaker condition (3$'$)
$G_\xi - I_2 \in \pi{\rm M}_2(\F[[\pi]]^S)$.
(Recall that $\xi$ is a topological generator of $\Gamma_1$.)
\end{lem}
\begin{proof}
If $\gamma, \gamma' \in \Gamma_1$, then
$$
G_{\gamma\gamma'} = G_\gamma\gamma G_{\gamma'} \equiv
\left(
  \begin{array}{cc}
    1 & \mu_\gamma(E)+\gamma\mu_{\gamma'}(E) \\
    0 & 1 \\
  \end{array}
\right) \mod \pi$$
by Corollary~\ref{gamma n}.
So, if $G_\xi =I_2 \mod \pi$, we have by induction that $G_{\xi^n}\equiv I_2$ for all $n \ge 1$.
Since $\langle \xi \rangle$ is dense in $\Gamma_1$,  continuity of the action gives
that $G_\gamma-I_2 \in \pi{\rm M}_2(\F[[\pi]]^S)$ for all $\gamma \in \Gamma_1$.
\end{proof}

We now describe a way to analyze extensions systematically and to
check for boundedness. Given $J \subset S$ and $n \in \Z/(p^f-1)\Z$
we can always find $a_i, b_j$ for $i \in J, j \in S-J$ with
$1 \le a_i, b_j \le p$ such that $$n \equiv \sum_{j\not\in
J}b_jp^j - \sum_{i \in J}a_ip^i \mod p^f-1.$$ 
The congruence has a unique solution if $n \not\equiv n_J
\mod p^f-1$, and has two solutions if $n \equiv n_J \mod
p^f-1$ where $n_J:=\sum_{i\in J}p^{i+1}-\sum_{i\not\in J}p^i$ (cf.
\cite[\S 3]{BDJ05}). To compute the solutions explicitly in the
double solution case suppose $n \equiv n_J \mod p^f-1$ and
we have two solutions $a_i, b_j$ and $a_i', b_j'$.
Then $\Sigma:= \sum_{j\not\in J}(b_j-b_j')p^j-\sum_{i\in J}(a_i-a_i')p^i
\equiv 0 \mod p^f-1$. 
Note that $|\Sigma|\le p^f-1$, as $|a_i-a_i'|, |b_j-b_j'|\le p-1$, so that $\Sigma = 0$ or $\pm(p^f-1)$,
and in the latter case we can exchange $\vec{a},\vec{b}$ and $\vec{a}',\vec{b}'$ if necessary
in order to assume $\Sigma = p^f - 1$.
If $\Sigma = 0$, then reducing modulo powers of $p$ shows that $\vec{a} = \vec{a}'$
and $\vec{b} = \vec{b}'$.
If $\Sigma = p^f-1$, then we have solutions $a_i = 1, b_j= p$ and $a_i' = p, b_j'= 1$
($i \in J, j\in S-J$).  

Now fix $J \subset S$, $C \in \F^\times$ and $\vec{c} \in \{0,1, \ldots, p-1\}^S$ with some $c_i < p-1$.
If $\Sigma_0\vec{c} \not\equiv n_J \mod p^f-1$, we can solve
the congruence $\Sigma_0\vec{c} \equiv \sum_{i\not\in J}b_ip^i -
\sum_{i \in J}a_ip^i \mod p^f-1$ with unique solution,
and get an isomorphism
$$\Ext^1(M_{\vec{0}},M_{C\vec{c}}) \simeq \Ext^1(M_{\vec{0}},M_{C\vec{d}}),$$
where $d_i=-a_i$ if $i\in J$ and $d_j=b_j$ if $j\not\in J$.
The isomorphism is (not canonical but) well-defined up to
${\rm Aut} M_{C\vec{c}}=\F^\times$ and the valuations of entries of
the matrices defining the $(\varphi, \Gamma)$-module extensions are
invariant, which suffices for our purposes. Tensoring $M_{A\vec{a}}$ with the
subobject and the quotient of the extension gives an isomorphism
$$\iota : \Ext^1(M_{\vec{0}},M_{C\vec{c}}) \rightarrow
\Ext^1(M_{A\vec{a}},M_{B\vec{b}})$$ where $CA=B$ and $a_i=0 $ if $i
\not\in J$ and $b_i=0$ if $i \in J$.  
Note that if $J=\emptyset$, $A=1$ and $c_i > 0$ for all $i$,  then 
$M_{A\vec{a}} = M_0$, $M_{B\vec{b}} = M_{C\vec{c}}=M_{C\vec{d}}$ and
$\iota$ is the identity.
In general, we have a commutative diagram
$$\begin{array}{ccc} H/H_0 & \longrightarrow & H'/H_0'
 \\\downarrow & & \downarrow \\
\Ext^1(M_{\vec{0}},M_{C\vec{c}}) & \longrightarrow &
\Ext^1(M_{A\vec{a}},M_{B\vec{b}}).\end{array}$$ The vertical arrows
are $\beta$'s, and the bottom arrow, which we also denote $\iota$, is induced by
$$(\mu_\varphi,(\mu_\gamma)_{\gamma \in \Gamma})\mapsto
(\kappa_\varphi(A,\vec{a})\langle\vec{c}\rangle_J\mu_\varphi,
(\kappa_\gamma(A,\vec{a})\langle\vec{c}\rangle_J\mu_\gamma)_{\gamma
\in \Gamma}),$$ where the isomorphism $\E_{K,F} e =M_{C\vec{c}} \simeq
M_{C\vec{d}} = \E_{K,F} e'$ is defined by the change of basis $e =
\langle\vec{c}\rangle_J e'$ with $\langle\vec{c}\rangle_J \in \E_{K,F}^\times$.
It is straighforward to check the following formula for $\langle\vec{c}\rangle_J$,
which we will need in order to compute spaces of bounded extensions: 
$$\langle\vec{c}\rangle_J = 
(\pi^{(p-1)\varepsilon_0}, \ldots, \pi^{(p-1)\varepsilon_{f-1}}),$$
where $(p^f - 1)\varepsilon_i = \Sigma_i(\vec{c} - \vec{d})$.
   
We define $$V_J=\iota^{-1}(\Ext^1_{\bdd} (M_{A\vec{a}},M_{B\vec{b}}))
\subset \Ext^1(M_{\vec{0}},M_{C\vec{c}}),$$ so that $\dim_\F V_J =
\dim_\F \Ext^1_{\bdd} (M_{A\vec{a}},M_{B\vec{b}})$.

If $\Sigma_0\vec{c} \equiv n_J \mod p^f-1$, we can assume that $n_J=
\sum_{i \not\in J}b_ip^i - \sum_{i \in J}a_ip^i$ and
$n_J+1-p^f=\sum_{i \not\in J}b_i'p^i - \sum_{i \in J}a_i'p^i$ where
$a_i, b_j$ and $a_i', b_j'$ are two solutions. 
Then we have $a_i=p, b_j=1$ and $a_i'=1, b_j'=p$ (for $i \in J$ and $j \in S-J$).
 
As in the case of a unique solution, we have isomorphisms (but now there are two)
$$\begin{aligned}
\iota_+ &: \Ext^1(M_{\vec{0}},M_{C\vec{c}}) \rightarrow
\Ext^1(M_{A\vec{a}},M_{B\vec{b}}),\\
\iota_- &: \Ext^1(M_{\vec{0}},M_{C\vec{c}}) \rightarrow
\Ext^1(M_{A\vec{a'}},M_{B\vec{b'}})
\end{aligned}$$ and define
$$\begin{aligned}
V_J^+ &=\iota_+^{-1}(\Ext^1_{\bdd} (M_{A\vec{a}},M_{B\vec{b}}))\subset
\Ext^1(M_{\vec{0}},M_{C\vec{c}}),\\
V_J^- &=\iota_-^{-1}(\Ext^1_{\bdd} (M_{A\vec{a'}},M_{B\vec{b'}}))
\subset \Ext^1(M_{\vec{0}},M_{C\vec{c}}).
\end{aligned}$$
Note that we always use $+$ to denote the case where all $a_i = p$, $b_j = 1$,
and $-$ for the case where all $a_i = 1$ and $b_j = p$.

In the next two subsections we will compute explicitly the spaces of bounded
extensions in the generic case and in the case $f=2$.

\subsection{Generic case} \label{sec:generic}

For each $i\in S$, let $e_i: \E_{K,F}=\F((\pi))^S \to \F((\pi))$ denote
the projection $(g_0,\ldots,g_{f-1})\mapsto g_i$.

\begin{prop}\label{gen}
If  $0< c_i <p-1$ for all $i \in S$, then 
\begin{enumerate}
\item $V_{\{i\}}=\F [B_{i+1}]$ for all $i \in S$;
\item $V_J=\oplus_{i\in J}V_{\{i\}}$ if $J \subset S$;
\item $V_S^+=V_S^-=\Ext^1(M_{\vec{0}},M_{C\overrightarrow{p-2}})$ if $C\neq 1, \vec{c}=\overrightarrow{p-2}$.
\end{enumerate}
\end{prop}

\begin{rem} Proposition~\ref{gen} does not say anything about the cyclotomic case 
$C=1,\vec{c}=\overrightarrow{p-2}$, which will be treated in \S 6.1.
\end{rem}

\begin{proof} First consider the case $J \neq \emptyset$. 
We may assume that $f-1 \in J$; even though we do not have complete symmetry
due to the presence of the constant $C$, we will see that the argument
goes through independently of which component $C$ lies in. 
As $0< c_i <p-1$ for all $i \in S$ we have ${\frac{p^f-1}{p-1}} \le \Sigma_0\vec{c} \le (p-2){\frac{p^f-1}{p-1}}$.
We claim that the congruence
$$\Sigma_0\vec{c} \equiv \sum_{i \not\in J}b_ip^i - \sum_{i \in J}a_ip^i \mod p^f-1,$$
has a unique solution $a_i$, $b_j$ ($i \in J, j \not\in J$) such that $1 \le a_i, b_j \le p$,
except when $J=S$ and $\vec{c}=\overrightarrow{p-2}$.
If there were another distinct solution, we have either
$\Sigma_0\vec{c}=\sum_{j\not\in J}p^{j+1}-\sum_{j\in J}p^j$ or
$\Sigma_0\vec{c}=\sum_{j\not\in J}p^{j+1}-\sum_{j\in J}p^j + p^f-1$.
The former is impossible since modulo $p$ we have $c_0\equiv -1$ if $0 \in J$ and 
$c_0\equiv 0$ if $j \not\in J$, and thus $c_0=p-1$ or $0$, contradicting the assumption.
In the latter case, we have $0 \in J$ and $c_0=p-2$. 
Computations modulo $p^2$ show that $1 \in J$ and $c_1=p-2$.
By induction we get $J=S$ and $c_i =p-2$ for all $i\in S$.
Thus, unless $J=S$, $\vec{c}=\overrightarrow{p-2}$, we have unique $a_i, b_j$ ($i \in J, j \not\in J$)
satisfying the equation $\sum_{i=0}^{f-1}c_ip^i = \sum_{i\not\in J}b_ip^i-\sum_{i \in J}a_ip^i +p^f-1$. 
Letting $u= (\pi^{(p-1)\delta_{f-1J}},\pi^{(p-1)\delta_{0J}},
\ldots,\pi^{(p-1)\delta_{f-2J}})$ with $\delta_{iJ}=1$ if $i \in J$
and $\delta_{iJ}=0$ otherwise, one checks that $\E_{K,F} e=
M_{C\vec{c}} \simeq M_{C\vec{d}}=\E_{K,F} e'$ via the change of basis $e = ue'$ where
$d_i=b_i$ if $i\not\in J$ and $d_i=-a_i$ if $i\in J$, and that $\langle\vec{c}\rangle_J=u$.
Note that
$$\begin{aligned}{\frac{\varphi(\langle\vec{c}\rangle_J)}{\langle\vec{c}\rangle_J}} &= {\frac{(\pi^{(p-1)p\delta_{0J}},\pi^{(p-1)p\delta_{1J}},
\ldots,\pi^{(p-1)p\delta_{f-1J}})}{(\pi^{(p-1)\delta_{f-1J}},\pi^{(p-1)\delta_{0J}},
\ldots,\pi^{(p-1)\delta_{f-2J}})}} \\
& =(\pi^{(p-1)(p\delta_{0J}-\delta_{f-1J})},
\pi^{(p-1)(p\delta_{1J}-\delta_{0J})},
\ldots,\pi^{(p-1)(p\delta_{f-1J}-\delta_{f-2J})})\end{aligned}$$ and
that
$$(p\delta_{0J}-\delta_{f-1J})+p(p\delta_{1J}-\delta_{0J})
+\cdots+p^{f-1}(p\delta_{f-1J}-\delta_{f-2J})
=(p^f-1)\delta_{f-1J}=p^f-1.$$

Recall that we have a basis $[B_0], \ldots, [B_{f-1}]$ for $\Ext^1(M_{\vec{0}},
M_{C\vec{c}})$ such that
$$\begin{aligned}\mu_\varphi(B_i) &=(0,\ldots,0, \pi^{1-p}+h_i(\pi), 0,\ldots,0), \\
\mu_\xi(B_i) &=(G_0^{(i)}, \ldots, G_{f-1}^{(i)}),\end{aligned}$$
where $h_i(\pi) \in \F\pi^{2-p}+\cdots+\F\pi^{-1}$ and
$$
\begin{aligned}
G_i^{(i)}(\pi) &= -\alpha_i+g_i(\pi), \\
G_{i-1}^{(i)}(\pi) &= \pi^{(p-1)c_{i-1}}(-\alpha_i+g_i(\pi^p)), \\
G_{i-2}^{(i)}(\pi) &=
\pi^{((p-1)(c_{i-2}+c_{i-1}p)}(-\alpha_i+g_i(\pi^{p^2})), \\
&\vdots \\
G_{0}^{(i)}(\pi) &=
\pi^{(p-1)(c_0+c_1p+\cdots+c_{i-1}p^{i-1})}(-\alpha_i+g_i(\pi^{p^i})), \\
G_{f-1}^{(i)}(\pi) &=
\pi^{(p-1)(c_{f-1}+c_0p+c_1p^2+\cdots+c_{i-1}p^i)}(-\alpha_i+g_i(\pi^{p^{i+1}})), \\
&\vdots \\
G_{i+1}^{(i)}(\pi) &= \pi^{(p-1)(c_{i+1}+\cdots+c_{i-1}p^{f-2})}
(-\alpha_i+g_i(\pi^{p^{f-1}})),
\end{aligned}
$$
with $\alpha_i=\overline{s_0z} \in \F^\times$ as in Lemma \ref{gamma}.

To show that $\iota [B_{i+1}] \in \Ext^1(M_{A\vec{a}},M_{B\vec{b}})$ is bounded if $i \in J$ is straightforward:
As $\langle\vec{c}\rangle_J=(\pi^{(p-1)\delta_{f-1J}},\pi^{(p-1)\delta_{0J}},
\ldots,\pi^{(p-1)\delta_{f-2J}})$ and $\delta_{iJ}=1$, we have
$$\begin{aligned}\mu_\varphi(\iota B_{i+1}) &=\kappa_\varphi(A,\vec{a})\langle\vec{c}\rangle_J\mu_\varphi(B_{i+1}) \\
& =(0,\ldots,0,\pi^{(p-1)(a_{i+1}+1)}(\pi^{1-p}+h_{i+1}(\pi)),0,\ldots,0) \in \F[[\pi]]^S
\end{aligned}$$ with the nonzero entry in the ($i+1$)-component, and 
$$\begin{aligned}\mu_\xi(\iota B_{i+1})
& =\kappa_\xi(A,\vec{a})\langle\vec{c}\rangle_J\mu_\xi(B_{i+1}) \in \pi\F[[\pi]]^S 
\end{aligned}$$
as $e_{i+1}\mu_\xi(\iota B_{i+1})=\lambda_\xi^{(p-1)\Sigma_{j+1}\vec{a}}\pi^{p-1}G_{i+1}^{(i)}$ and 
$e_{j+1}\mu_\xi(\iota B_{i+1})=\lambda_\xi^{(p-1)\Sigma_{j+1}\vec{a}}\pi^{(p-1)\delta_{jJ}}G_{j+1}^{(i)}$
is divisible by $\pi^{(p-1)c_{j+1}}$ if $j\neq i$.

Next we need to show that if $E =\sum_{j=0}^{f-1}\beta_jB_j$ and $[E]\in V_J$,
then $\beta_{i+1}=0$ for all $i \notin J$.
Suppose $\iota [E]$ is bounded, $i \not\in J$ and $\beta_{i+1} \neq 0$.
Then $\mu_\varphi\iota(E+B) \in \F[[\pi]]^S$ and $\mu_\xi\iota(E+B) \in \pi\F[[\pi]]^S$
for some coboundary $B$, for which we have
$$\begin{aligned}\mu_\varphi(B) &=(C\pi^{(p-1)c_0}b_1(\pi^p)-b_0,
\pi^{(p-1)c_1}b_2(\pi^p)-b_1,\ldots, \pi^{(p-1)c_{f-1}}b_0(\pi^p)-b_{f-1}), \\
\mu_\xi(B) & =((\lambda_\xi^{\Sigma_0\vec{c}}\xi-1)b_0(\pi), (\lambda_\xi^{\Sigma_1\vec{c}}\xi-1)b_1(\pi)
\ldots,(\lambda_\xi^{\Sigma_{f-1}\vec{c}}\xi-1)b_{f-1}(\pi)).
\end{aligned}$$
for some $(b_0(\pi), \ldots, b_{f-1}(\pi)) \in \E_{K,F}$.
Note that $\kappa_\xi(A,\vec{a}) \equiv 1=(1,\ldots,1) \mod \pi$
and $e_{i+1}\langle\vec{c}\rangle_J = \pi^{(p-1)\delta_{iJ}}=1$.
As $\val e_{i+1}\mu_\xi(E+B)=\val e_{i+1}\mu_\xi\iota(E+B) \ge 1$ while $\val e_{i+1}\mu_\xi(E)=0$,
the valuation of $e_{i+1}\mu_\xi(B)=(\lambda_\xi^{\Sigma_{i+1}}\xi-1)b_{i+1}(\pi)$ has to be zero.
Letting $s= \val b_{i+1}(\pi)$, Lemma \ref{gamma} implies that 
$(\lambda_\xi^{\Sigma_{i+1}\vec{c}}\xi-1)b_{i+1}(\pi) \in \pi\F[[\pi]]$ if $s \ge 0$
and $\val(\lambda_\xi^{\Sigma_{i+1}\vec{c}}\xi-1)b_{i+1}(\pi) =s+(p-1)p^v$
if $s<0$ and $\Sigma_{i+1}\vec{c}+s(p^f-1)/(p-1)$ is divisible by $p^v$ but not $p^{v+1}$.
Thus $\val b_{i+1}(\pi)$ must be negative and divisible by $p-1$.
Looking at the $i$-th component, we have
$$\begin{aligned}
e_i\mu_\varphi(\iota(E+B)) & =\pi^{(p-1)(a_i+\delta_{i-1J})} (e_i\mu_\varphi(E) + e_i\mu_\varphi(B)) \\
& =\pi^{(p-1)\delta_{i-1J}}(\pi^{1-p}+h_i(\pi) +\pi^{(p-1)c_i}b_{i+1}(\pi^p)-b_i(\pi)),\end{aligned}$$
whose valuation has to be non-negative. Since $(p-1)c_i+p \,
\val b_{i+1}(\pi) < 1-p=\val(\pi^{1-p}+h_i(\pi))$, we get $\val b_i(\pi)=(p-1)c_i+p \, \val b_{i+1}(\pi)$.
Cycling this through all $j \in S$ leads to
$\val b_{i+1}(\pi) =(p-1)\Sigma_i\vec{c} +p^f \val b_{i+1}(\pi)$, so that
$\val b_{i+1}(\pi) =-{\frac{p-1}{p^f-1}}\Sigma_i\vec{c} >1-p$, which is a contradiction.

Now suppose $J=S, C\neq 1, \vec{c}=\overrightarrow{p-2}$. In this case we have two solutions
$\vec{a}=\vec{p}, \vec{b}=\vec{0}$ and $\vec{a'}=\vec{1}, \vec{b'}=\vec{0}$ of the congruence
and the corresponding isomorphisms
$$\begin{aligned}
\iota_+ &: \Ext^1(M_{\vec{0}},M_{C\overrightarrow{p-2}}) \rightarrow
\Ext^1(M_{A\vec{p}},M_{B\vec{0}}),\\
\iota_- &: \Ext^1(M_{\vec{0}},M_{C\overrightarrow{p-2}}) \rightarrow
\Ext^1(M_{A\vec{1}},M_{B\vec{0}}).
\end{aligned}$$
One shows that $V_J^+=V_J^-=\Ext^1(M_{\vec{0}},M_{C\overrightarrow{p-2}})$ by straightforward computations.

If $J=\emptyset$, the congruence equation has a unique solution unless $\vec{c}=\vec{1}$,
in which case we have two solutions $\vec{a}=\vec{0}, \vec{b}=\vec{1}$ and $\vec{a'}=\vec{0},\vec{b'}=\vec{p}$.
The proof that $V_\emptyset=0$ (when $\vec{c}\neq\vec{1}$) and
$V_\emptyset^+=V_\emptyset^-=0$ (when $\vec{c}=\vec{1}$) is
similar to the case $J\neq\emptyset$.
\end{proof}

\subsection{Case $f=2$} \label{sec:f2}
Throughout this subsection we assume that $f=2$, $0 \le c_0, c_1 \le p-1$, not both $p-1$.
If $\vec{c}=\vec{0}$ or $\overrightarrow{p-2}$, we further assume $C\neq 1$; 
the cases $\vec{c}=\vec{0}$ and $\vec{c}=\overrightarrow{p-2}$ when $C=1$
are dealt with in \S\S \ref{sec:cyclo}, \ref{sec:triv}.
Before determining which extensions are bounded, we describe the basis
elements in the form we will need.  Recall that we defined a basis $\{[B_0],[B_1]\}$
for $ \Ext^1(M_{\vec{0}},M_{C\vec{c}})$ where $B_0$ and $B_1$ are cocycles of the
following form:

$$\begin{aligned}
\mu_\varphi(B_0) & =(H_0(\pi),0), \\
\mu_\xi(B_0) & =(G_0^{(0)}(\pi),G_1^{(0)}(\pi)) \\
             & =((C\pi^{(p-1)\Sigma_0}\Phi-1)^{-1}((\lambda_\xi^{\Sigma_0}\xi-1)H_0(\pi)), \pi^{(p-1)c_1}G_0^{(0)}(\pi^p)), \\
\mu_\varphi(B_1) & =(0, H_1(\pi)), \\
\mu_\xi(B_1) & =(G_0^{(1)}(\pi),G_1^{(1)}(\pi)) \\
 &=(\pi^{(p-1)c_0}G_1^{(1)}(\pi^p),(C\pi^{(p-1)\Sigma_1}\Phi-1)^{-1}((\lambda_\xi^{\Sigma_1}\xi-1)H_1(\pi))),
\end{aligned}$$
where
$$H_0(\pi) = \left\{\begin{array}{ll}\pi^{1-p}+h_0 & {\rm if} \, c_0<p-1, \\
\pi^{1-p^2}+h_0^{(1)}+\epsilon^{(0)}h_0'^{(0)}+h_0^{(0)} & {\rm if} \, c_0=p-1, c_1\neq p-2, \\
\pi^{1-p^3}+h_0^{(2)}+\epsilon^{(1)}h_0'^{(1)}+h_0^{(1)}+\epsilon^{(0)}h_0'^{(0)}+h_0^{(0)}
& {\rm if} \, c_0=p-1, c_1= p-2,
\end{array} \right.$$
$$H_1(\pi)
=\left\{\begin{array}{ll}\pi^{1-p}+h_1 & {\rm if} \, c_1<p-1, \\
\pi^{1-p^2}+h_1^{(1)}+\epsilon^{(0)}h_1'^{(0)}+h_1^{(0)}  & {\rm if} \, c_1=p-1, c_0\neq p-2, \\
\pi^{1-p^3}+h_1^{(2)}+\epsilon^{(1)}h_1'^{(1)}+h_1^{(1)}+\epsilon^{(0)}h_1'^{(0)}+h_1^{(0)}
& {\rm if} \, c_1=p-1, c_0= p-2.
\end{array} \right.$$

\begin{lem}\label{f2.1}
Suppose that $i \in \{0,1\}$ is such that $0 \le c_i < p - 1$.  Then 
for some $\alpha_i \in F^\times, g_i(\pi) \in 1 + \pi\F[[\pi]]$, we have
$$\begin{array}{ll}
\left.\begin{array}{rl}
\mu_\varphi(B_0) &=(\pi^{1-p}+h_0(\pi),0),\\
\mu_\xi(B_0) &=(\alpha_0g_0(\pi),
\pi^{(p-1)c_1}\alpha_0g_0(\pi^p))\end{array}\right\}&\mbox{if $i=0$,}\\
\ & \ \\
\left.\begin{array}{rl}
\mu_\varphi(B_1) &=(0, \pi^{1-p}+h_1(\pi)),\\
\mu_\xi(B_1) &=(C\pi^{(p-1)c_0}\alpha_1g_1(\pi^p),
\alpha_1g_1(\pi))\end{array}\right\}&\mbox{if $i=1$.}
\end{array}$$
\end{lem}
\begin{proof}  We assume $i=0$; the case $i=1$ is similar.  As in Lemma~\ref{gamma}, we have
$$L_0(\pi) :=
(\lambda_\xi^{\Sigma_0}\xi-1)(\pi^{1-p}+h_0(\pi)) \equiv \overline{s_0z} \mod \pi\F[[\pi]],$$
so that $e_0\mu_\xi(B_0)=(C\pi^{(p-1)\Sigma_0}\Phi-1)^{-1}(L_0(\pi))
=\alpha_0g_0(\pi)$ for some $g_0(\pi) \in 1+\pi\F[[\pi]]$ with
$\alpha_0 = (C-1)^{-1}\overline{s_0z}$ if $c_0 = c_1 = 0$, and
$\alpha_0 = -\overline{s_0z}$ otherwise. (Recall that we assume for now that $C\neq 1$
if $c_0 = c_1 = 0$.)
\end{proof}
 
If $c_i = p-1$, we introduce a cocycle $B_i'$ cohomologous to $B_i$ which we will work with.
\begin{lem}\label{f2.2}
Suppose that $\{i,j\} = \{0,1\}$ with $c_i = p - 1$ and $c_j < p - 2$.
Then there is a cocycle $B_i'$ such that $[B_i'] = [B_i]$ and
$$\begin{array}{c} \val(e_i\mu_\varphi(B_i')) = \val(e_j\mu_\varphi(B_i')) = 2-2p,\\
\val e_i\mu_\xi(B_i')) \ge 0\quad\mbox{and} \quad\val (e_j\mu_\xi(B_i')) = 1 - p.\end{array}$$
\end{lem}
\begin{proof}  Again assume $i=0$, the case $i=1$ being similar.  
By the very construction of $H_0(\pi)$, we have
$L_0(\pi) := (\lambda_\xi^{\Sigma_0}\xi-1)H_0(\pi) \in \F[[\pi]],$ so that
$g_0'(\pi) := e_0\mu_\xi(B_0) =(C\pi^{(p-1)\Sigma_0}\Phi-1)^{-1}(L_0(\pi)) \in \F[[\pi]]$
and $\mu_\xi(B_0)=(g_0'(\pi),\pi^{(p-1)c_1}g_0'(\pi^p)).$
Now let $B_0' = B_0 - B$ where $B$ is a coboundary such that 
$$\begin{aligned} 
\mu_\varphi(B) &= (C\pi^{(p-1)c_0}b_1(\pi^p)-b_0(\pi), \pi^{(p-1)c_1}b_0(\pi^p)-b_1(\pi)) \\
&=(\pi^{1-p^2}+h_0^{(1)}, -C^{-1}(\pi^{2-2p}+\widetilde{h}_0^{(1)}),\\
\mu_\xi(B) &=(0, (\lambda_\xi^{\Sigma_1}\xi-1)b_1(\pi)),
\end{aligned}$$
where $b_0(\pi) =0, b_1(\pi) = C^{-1}(\pi^{2-2p}+\widetilde{h}_0^{(1)})$ with
$\widetilde{h}_0^{(1)} :=\sum_{s=1}^{p-2}\epsilon_s^{(1)}\pi^{2-2p+s}$. Then
$$\begin{aligned}
\mu_\varphi(B_0') &= \left(\epsilon^{(0)}h_0'^{(0)}+h_0^{(0)}, C^{-1}(\pi^{2-2p}+h_0^{(1)})\right)\\
\mu_\xi(B_0') & =(g_0'(\pi), \pi^{(p-1)c_1}g_0'(\pi^p)-(\lambda_\xi^{\Sigma_1}\xi-1)b_1(\pi)),
\end{aligned}$$
so that $B_0'$ has the required form.
\end{proof}

\begin{lem}\label{f2.3}
Suppose that $\{i,j\} = \{0,1\}$ with $c_i = p - 1$ and $c_j = p - 2$.
Then there is a cocycle $B_i'$ such that $[B_i'] = [B_i]$ and
$$\begin{array}{c} \val(e_i\mu_\varphi(B_i')) \ge 2 - 2p,\quad \val(e_j\mu_\varphi(B_i')) = 3-3p,\\
\val(e_i\mu_\xi(B_i')) = 1 - p\quad\mbox{and} \quad\val(e_j\mu_\xi(B_i')) \ge 2 - 2p.\end{array}$$
\end{lem}
\begin{proof}  This is similar to Lemma~\ref{f2.2} but choose the coboundary $B$  such that
$$\begin{aligned} 
\mu_\varphi(B) &= (C\pi^{(p-1)c_0}b_1(\pi^p)-b_0(\pi), \pi^{(p-1)c_1}b_0(\pi^p)-b_1(\pi)) \\
               &= (\pi^{1-p^3}+\epsilon^{(1)}h_0'^{(1)}+h_0^{(1)}+h_0^{(2)}-C^{-1}\widetilde{h}_0^{(2)},
                   -C^{-1}(\epsilon^{(1)}\widetilde{h}_0'^{(1)}+\widetilde{h}_0^{(1)})),\\
\mu_\xi(B) &=((\lambda_\xi^{\Sigma_0}\xi-1)b_0(\pi), (\lambda_\xi^{\Sigma_1}\xi-1)b_1(\pi)),
\end{aligned}$$
by taking
$b_0(\pi)= C^{-1}\widetilde{h}_0^{(2)}, b_1(\pi)= C^{-1}\left(\epsilon^{(1)}\widetilde{h}_0'^{(1)}
            + \widetilde{h}_0^{(1)} \right)
            + \pi^{(p-1)(p-2)}b_0(\pi^p)$,
where $\widetilde{h}_0^{(2)}:=\sum_{s=0}^{p-2}\epsilon_s^{(2)}\pi^{2-2p+s}$ (with $\epsilon_0^{(2)}=1$),
$\widetilde{h}_0'^{(1)}:=\sum_{s=0}^{p-2}\epsilon_s'^{(1)}\pi^{3-3p+s}$ (with $\epsilon'^{(1)}=1$) and
$\widetilde{h}_0^{(1)}:=\sum_{s=1}^{p-2}\epsilon_s^{(1)}\pi^{2-2p+s}$.
\end{proof}

\begin{prop} If $f=2$, then
$$V_S = V_S^\pm = \Ext^1(M_{\vec{0}},M_{C\vec{c}})$$ with $\pm$ occuring when $\vec{c}=\overrightarrow{p-2}$.
\end{prop}

\begin{proof} By straightforward calculations one can check that both $\iota [B_0]$ and $\iota [B_1]$
are bounded in each of the following cases to consider:
\begin{itemize}
\item $0\le c_0, c_1\le p-2$, $1 \le a_0, a_1 \le p-1$, $\langle\vec{c}\rangle_S=(\pi^{p-1}, \pi^{p-1})$;
\item $c_0=p-1, 0\le c_1< p-2$, $a_0=1, 1\le a_1 < p-1$, $\langle\vec{c}\rangle_S=(\pi^{p-1}, \pi^{2p-2})$;
\item $0\le c_0< p-2, c_1=p-1$, $1\le a_0 < p-1, a_1=1$, $\langle\vec{c}\rangle_S=(\pi^{2p-2}, \pi^{p-1})$;
\item $p-2\le c_0, c_1\le p-1$, $p-1 \le a_0, a_1 \le p$, $\langle\vec{c}\rangle_S=(\pi^{2p-2}, \pi^{2p-2})$;
\item $c_0=c_1=p-2$, $a_0=a_1=p$, $\langle\vec{c}\rangle_S=(\pi^{2p-2}, \pi^{2p-2})$ (for $V_S^+$);
\item $c_0=c_1=p-2$, $a_0=a_1=1$, $\langle\vec{c}\rangle_S=(\pi^{p-1}, \pi^{p-1})$ (for $V_S^-$).
\end{itemize}
\end{proof}

\begin{prop} If $f=2$, then $$V_\emptyset = V_\emptyset^\pm =0$$
with $\pm$ occurring when $\vec{c}=\vec{1}$.
\end{prop}

\begin{proof}
We have the following cases to consider:
\begin{itemize}
\item $1\le c_0, c_1 \le p-1$,$1 \le b_0, b_1 \le p-1$, $\langle\vec{c}\rangle_\emptyset=(1,1)$;
\item $c_0 =0, 2\le c_1 \le p-1$,$b_0=p, 1\le b_1 \le p-2$, $\langle\vec{c}\rangle_\emptyset=(1,\pi^{1-p})$;
\item $2 \le c_0 \le p-1, c_1 = 0$,$1\le b_0 \le p-2, b_1=p$, $\langle\vec{c}\rangle_\emptyset=(\pi^{1-p},1)$;
\item $0 \le c_0, c_1 \le 1$, $p-1\le b_0, b_1 \le p$, $\langle\vec{c}\rangle_\emptyset=(\pi^{1-p},\pi^{1-p})$.
\end{itemize}
If $E$ is a cocycle such that $\iota[E]$ is bounded, then there is a coboundary $B$
associated to some $(b_0(\pi),b_1(\pi))\in \F((\pi))^S$ such that
$\iota(E+B)$ has $\mu_\varphi\in \F[[\pi]]^S$ and $\mu_\xi\in\pi\F[[\pi]]^S$.
As $\kappa_\varphi(A,\vec{a}) \in (\F^\times)^S$ and
$\langle\vec{c}\rangle_\emptyset = (\pi^{(1-p)\epsilon_0}, \pi^{(1-p)\epsilon_1})$ for some $\epsilon_j \in \{0,1\}$,
we get $\mu_\varphi(E+B) \in \F[[\pi]]^S$ and $\mu_\xi(E+B) \in \pi\F[[\pi]]^S$.

First consider the case $0\le c_0,c_1<p-1$ and $E = B_0+\beta B_1$ for some $\beta\in \F^\times$.
As $\val e_{0}\mu_\varphi(E)=1-p$ and $\val e_1\mu_\varphi(E)\ge 1-p$, we have 
$\val(C\pi^{(p-1)c_0}b_1(\pi^p)-b_0(\pi))=1-p$ and $\val(\pi^{(p-1)c_1}b_0(\pi^p)-b_1(\pi))\ge 1-p$.
If $\val b_0(\pi) > 1-p$, then $(p-1)c_0 + p\val b_1(\pi) = 1-p$, which implies that $p|(c_0+1)$,
contradicting $c_0 < p-1$.  If $\val b_0(\pi) \le 1-p$, then $(p-1)c_1+p\val b_0(\pi) < 1-p$,
which implies that $\val b_1(\pi) = (p-1)c_1+p\val b_0(\pi) < 1-p$, which in turn implies
$(p-1)c_0+p\val b_1(\pi) < 1-p$, so that $\val b_0(\pi) = (p-1)c_0+p\val b_1(\pi) =
(p-1)\Sigma_0 + p^2\val b_0(\pi)$, yielding a contradiction.  The proof that $\iota[B_1]$
is not bounded is the same.

Next suppose $c_0=p-1$ and $0<c_1<p-2$.   First consider the case $E = B_0' +\beta B_1$.
As $\val e_1\mu_\xi(E)=1-p$,
we have $\val(\lambda_\xi^{\Sigma_1}\xi-1)b_1(\pi)=1-p$, so that $\val b_1(\pi) \le 2-2p$.
Then $\val\pi^{(p-1)c_0}b_1(\pi^p)=(p-1)c_0+p\val b_1(\pi)<(1-p)(1+p)
                                                 <2-2p=\val e_0\mu_\varphi(E)$,
and so $\val b_0(\pi) = \val \pi^{(p-1)c_0}b_1(\pi^p) = (p-1)c_0+p\val b_1(\pi)$.
Then again $\val\pi^{(p-1)c_{1}}b_0(\pi^p)
            = (p-1)\Sigma_{1}+p^2\val b_{1}(\pi)<2-2p = \val e_1\mu_\varphi(E)$,
so that $\val b_{1}(\pi) = \val\pi^{(p-1)c_{1}}b_0(\pi^p) = (p-1)\Sigma_1+p^2\val b_1(\pi)$,
or $\val b_{1}(\pi) = -\frac{p-1}{p^2-1}\Sigma_{1}> 2-2p$, a contradiction.
The proof that $\iota [B_1]$ is not bounded is the same as in the case $c_0 < p-1$.

If $c_0=p-1, c_1=p-2$, the proof is similar to the preceding case,
except that we start by noting that $\val e_0\mu_\xi(E)=1-p$ if $E = B_0' +\beta B_1$.

The proof in the case that $c_1 = p-1$ is the same as the case $c_0 = p - 1$.
\end{proof}

\begin{prop}\label{dim1} If $f=2$, then
$$\begin{aligned}
V_{\{1\}} &= 
\left\{
\begin{array}{ll}
\F [B_1]  & {\it if} \,\, c_0=p-1, \\
\F[\alpha_1B_0-\alpha_0B_1] & {\it if} \,\, 0<c_0<p-1, c_1=0, \\
\F [B_0]  & 0\le c_0<p-1, 0<c_1\le p-1; 
\end{array}\right. \\
V_{\{1\}}^+ &= \F[\alpha_1B_0-\alpha_0B_1]; \\
V_{\{1\}}^- &= 0,
\end{aligned}$$
with $\pm$ occuring when $\vec{c}=\vec{0}$.
(See Lemma \ref{f2.1} for the definition of the $\alpha_i$.)
\end{prop}

\begin{proof} Unless $\vec{c}=\vec{0}$, $\vec{c}$ gives rise to unique 
$\vec{a}=(0,a_1), \vec{b}=(b_0,0)$ with $1\le a_1, b_0 \le p$.
If $\vec{c}=\vec{0}$, we have $\vec{a}=(0,p), \vec{b}=(1,0)$ (for $V_J^+$)
or $\vec{a}=(0,1), \vec{b}=(p,0)$ (for $V_J^-$).
We always have $\langle\vec{c}\rangle_{\{1\}}=(\pi^{p-1},1)$ except when $\vec{c}=\vec{0}$, $b_0=p$, $a_0=1$,
in which case we have $\langle\vec{c}\rangle_{\{1\}}=(1,\pi^{1-p})$.

(1) Assume $c_0=p-1$. It is straightforward to check that $\iota [B_1+\beta B]$ is bounded 
for some $\beta \in \F^\times$ where $B$ is a coboundary
such that 
$$\mu_\varphi(B)=(C\pi^{(p-1)c_0}\pi^{(1-p)p}, -\pi^{1-p})\quad\mbox{and}\quad
\mu_\xi(B)=(0, (\lambda_\xi^{\Sigma_1}\xi-1)(\pi^{1-p})).$$

Suppose $\iota [B_0']$ is bounded. There exists a coboundary $B$ such that
$\mu_\varphi\iota(B_0'+B) \in \F[[\pi]]^S$, $\mu_\xi\iota(B_0'+B) \in \pi\F[[\pi]]^S$, and so
$$\begin{aligned}
\mu_\varphi(B_0') +(C\pi^{(p-1)c_0}b_1(\pi^p)-b_0(\pi), \pi^{(p-1)c_1}b_0(\pi^p)-b_1(\pi)) 
                  & \in \pi^{1-p}\F[[\pi]]\times\pi^{(1-p)a_1}\F[[\pi]], \\
\mu_\xi(B_0')     + ((\lambda_\xi^{\Sigma_0}\xi-1)b_0(\pi),(\lambda_\xi^{\Sigma_1}\xi-1)b_1(\pi))    
                  & \in \pi^{2-p}\F[[\pi]]\times\pi\F[[\pi]]
\end{aligned}$$
for some $b_0(\pi), b_1(\pi) \in \F((\pi))$.

If $c_1<p-2$, we have $\val (\lambda_\xi^{\Sigma_1}\xi-1)b_1(\pi)=1-p$ as $\val e_1\mu_\xi(B_0')=1-p$,
so that $\val b_1(\pi) \le 2-2p$.
Then 
$$\val(\pi^{(p-1)c_0}b_1(\pi^p))=(p-1)c_0+p\val b_1(\pi)<(1-p)(1+p)<\val e_0 \mu_\varphi(B_0'),$$
and so $\val b_0(\pi)=(p-1)c_0+p\val b_1(\pi)$.
Then again 
$$\begin{aligned}
\val \pi^{(p-1)c_1}b_0(\pi^p)&=(p-1)c_1+p\val b_0(\pi)=(p-1)\Sigma_1+p^2\val b_1(\pi)\\
 &<(1-p)(1+p)<(1-p)a_1,\end{aligned}$$
so that $\val b_1(\pi)=(p-1)\Sigma_1+p^2\val b_1(\pi)$,
or $\val b_1(\pi)=-\frac{p-1}{p^2-1}\Sigma_1>1-p$, a contradiction.

If $c_1=p-2$, start with $\val (\lambda_\xi^{\Sigma_0}\xi-1)b_0(\pi)=1-p$
and the same argument as above (for the case $c_1<p-2$) goes through.

(2) Assume $0<c_0<p-1, c_1=0$.
Straightforward calculations show that
$\mu_\varphi\iota B_0, \mu_\varphi\iota B_1 \in \F[[\pi]]^S$ but
$\mu_\xi\iota B_0(\pi), \mu_\xi\iota B_1 \not\in \pi\F[[\pi]]^S$.
If, however, we take $\alpha_1B_0-\alpha_0B_1$, it has $\mu_\varphi$ obviously in $\F[[\pi]]^S$ and
$\mu_\xi = \kappa_\xi(A,\vec{a})(\pi^{p-1}\alpha_0\alpha_1(g_0(\pi)-C\pi^{(p-1)c_0}g_1(\pi^p)), \alpha_0\alpha_1(g_0(\pi^p)-g_1(\pi))) \in \pi\F[[\pi]]^S$.

Now suppose $\iota [B_1]$ is bounded, and so we have, for some coboundary $B$, that
$\mu_\varphi\iota(B_1+B) \in \F[[\pi]]^S$ and $\mu_\xi\iota(B_1+B) \in \pi\F[[\pi]]^S$,
which implies
$$\begin{aligned}
\mu_\varphi(B_1) +(C\pi^{(p-1)c_0}b_1(\pi^p)-b_0(\pi), \pi^{(p-1)c_1}b_0(\pi^p)-b_1(\pi)) 
                  & \in \pi^{1-p}\F[[\pi]]\times\pi^{(1-p)a_1}\F[[\pi]], \\
\mu_\xi(B_1)     + ((\lambda_\xi^{\Sigma_0}\xi-1)b_0(\pi),(\lambda_\xi^{\Sigma_1}\xi-1)b_1(\pi))    
                  & \in \pi^{2-p}\F[[\pi]]\times\pi\F[[\pi]]
\end{aligned}$$
for some $b_0(\pi), b_1(\pi) \in \F((\pi))$.

We have $\val(\lambda_\xi^{\Sigma_1}\xi-1)b_1(\pi)=0$ and so $\val b_1(\pi) \le 1-p$,
so that $\val \pi^{(p-1)c_0}b_1(\pi^p)=(p-1)c_0+p\val b_1(\pi)<1-p$.
Then $\val b_0(\pi) = (p-1)c_0+p\val b_1(\pi)$ and 
$\val \pi^{(p-1)c_1}b_0(\pi^p)=(p-1)\Sigma_1+p^2\val b_1(\pi)<(1-p)a_1$, so that
$\val b_1(\pi)=\Sigma_1+p^2\val b_1(\pi)$, or $\val b_1(\pi)=-\frac{p-1}{p^2-1}\Sigma_1>1-p$,
 a contradiction.

(3) Assume $0\le c_0<p-1, 0<c_1 \le p-1$.
It is straightforward to check that $\iota [B_0]$ is bounded:
$$\begin{aligned}
\mu_\varphi\iota(B_0) &=(A, \pi^{(p-1)a_1})(\pi^{p-1}, 1)(\pi^{1-p}+h_0(\pi)) \in \F[[\pi]]^S, \\
\mu_\xi\iota(B_0)     &=\kappa_\xi(A, \vec{a})(\pi^{p-1}, 1)(\alpha_0g(\pi), \pi^{(p-1)c_1}\alpha_0g_0(\pi^p))
                       \in \pi\F[[\pi]]^S
\end{aligned}$$
as $c_1>0$. 

Now suppose $\iota [B_1]$ is bounded. Then there exists a coboundary $B$ such that
$\mu_\varphi\iota(B_1+B) \in \F[[\pi]]^S$, $\mu_\xi\iota(B_1+B) \in \pi\F[[\pi]]^S$, and so
$$\begin{aligned}
\mu_\varphi(B_1+B) +(C\pi^{(p-1)c_0}b_1(\pi^p)-b_0(\pi), \pi^{(p-1)c_1}b_0(\pi^p)-b_1(\pi)) 
                  & \in \pi^{1-p}\F[[\pi]]\times\pi^{(1-p)a_1}\F[[\pi]], \\
\mu_\xi(B_1+B)     + ((\lambda_\xi^{\Sigma_0}\xi-1)b_0(\pi),(\lambda_\xi^{\Sigma_1}\xi-1)b_1(\pi))    
                  & \in \pi^{2-p}\F[[\pi]]\times\pi\F[[\pi]].
\end{aligned}$$
If $c_1<p-1$, then the argument is the same as in case (2).

If $c_1=p-1, c_0<p-2$, then as 
$\val e_0\mu_\xi(B_1'+B) \ge 2-p$ and $\val e_0\mu_\xi(B_1')=1-p$,
we have $\val e_0\mu_\xi(B)=\val(\lambda_\xi^{\Sigma_0}\xi-1)b_0(\pi)=1-p$, 
so that $\val b_0(\pi) \le 2-2p$.
Then 
$$\begin{aligned}\val \pi^{(p-1)c_1}b_0(\pi^p) &= (p-1)c_1+p\val b_0(\pi)\\
&\le (1-p)(1+p)
<{\rm min}(\val e_1 \mu_\varphi B_1', (1-p)a_1),\end{aligned}$$
 so that $\val b_1(\pi)=(p-1)c_1+p\val b_0(\pi)$.
So $$\val \pi^{(p-1)c_0}b_1(\pi^p)=(p-1)\Sigma_0+p^2\val b_0(\pi)<(1-p)(1+p) < \val e_0\mu_\varphi(B_1'),$$
which implies $\val b_0(\pi)=(p-1)\Sigma_0+p^2\val b_0(\pi)$,  
or $\val b_0(\pi)=-\frac{p-1}{p^2-1}\Sigma_0>1-p$, a contradiction.

If $c_1=p-1, c_0=p-2$, then as 
$\val e_1\mu_\xi(B_1'+B) \ge 1$ and $\val e_1\mu_\xi(B_1')=1-p$,
we have $\val e_1\mu_\xi(B)=\val(\lambda_\xi^{\Sigma_1}\xi-1)b_1(\pi)=1-p$, 
so that $\val b_1(\pi) \le 2-2p$.
Then $$\val \pi^{(p-1)c_0}b_1(\pi^p) = (p-1)c_0 + p\val b_1(\pi)\le (1-p)(1+p)<
\val e_0 \mu_\varphi B_1',$$ so that $\val b_0(\pi)=(p-1)c_0+p\val b_1(\pi)$.
So 
$$\val \pi^{(p-1)c_1}b_0(\pi^p)=(p-1)\Sigma_1+p^2\val b_1(\pi)<(1-p)(1+p) < \val e_1\mu_\varphi(B_1'),$$
which implies $\val b_1(\pi)=(p-1)\Sigma_1+p^2\val b_1(\pi)$,  
or $\val b_1(\pi)=-\frac{p-1}{p^2-1}\Sigma_1>1-p$, a contradiction.

(4) Assume $c_0=c_1=0$, $b_0=1, a_1=p$. Straightforward calculations show that
$\mu_\varphi\iota B_0(\pi), \mu_\varphi\iota B_1 \in \F[[\pi]]^S$ but
$\mu_\xi\iota B_0(\pi), \mu_\xi\iota B_1 \not\in \pi\F[[\pi]]^S$.
If, however, we take $\alpha_1B_0-\alpha_0B_1$, it has $\mu_\varphi$ obviously in $\F[[\pi]]^S$ and
$$\mu_\xi = (\pi^{p-1}\alpha_0\alpha_1(g_0(\pi)-Cg_1(\pi^p)), \alpha_0\alpha_1(g_0(\pi^p)-g_1(\pi))) \in \pi\F[[\pi]]^S.$$

Now suppose $\iota [B_1]$ is bounded, and so we have, for some coboundary $B$, that
$\mu_\varphi\iota(B_1+B) \in \F[[\pi]]^S$ and $\mu_\xi\iota(B_1+B) \in \pi\F[[\pi]]^S$,
which implies
$$\begin{aligned}
\mu_\varphi(B_0) +(C b_1(\pi^p)-b_0(\pi), b_0(\pi^p)-b_1(\pi)) 
                  & \in \pi^{1-p}\F[[\pi]]\times\pi^{(1-p)p}\F[[\pi]], \\
\mu_\xi(B_0)     + ((\xi-1)b_0(\pi),(\xi-1)b_1(\pi))    
                  & \in \pi^{2-p}\F[[\pi]]\times\pi\F[[\pi]]
\end{aligned}$$
for some $b_0(\pi), b_1(\pi) \in \F((\pi))$.
We have $\val(\xi-1)b_1(\pi)=0$ and so $\val b_1(\pi) \le 1-p$,
so that $\val b_1(\pi^p)=p\val b_1(\pi)<1-p $.
Then $\val b_0(\pi) = p\val b_1(\pi)$ and 
$\val b_0(\pi^p)=p^2\val b_1(\pi)<(1-p)p<\val e_0\mu_\varphi (B_1)$, giving
 $\val b_1(\pi)=0$ and a contradiction.

(5) Assume $c_0=c_1=0$, $b_0=p, a_1=1$.
Suppose $\iota[B_0+\beta B_1]$ is bounded for some $\beta \in \F$.
There exist a coboundary $B$ such that
$\mu_\varphi\iota(B_0+\beta B_1+B) \in \F[[\pi]]^S$ and $\mu_\xi\iota(B_0+\beta B_1+B) \in \pi\F[[\pi]]^S$.
As $\kappa_\varphi(A,\vec{a})\langle\vec{c}\rangle \in (\F^\times)^S$, we have 
$$\begin{aligned}
\mu_\varphi(B_0+\beta B_1+B) &= \mu_\varphi(B_0+\beta B_1)+(Cb_1(\pi^p)-b_0(\pi),b_0(\pi^p)-b_1(\pi))
                                \in \F[[\pi]]^S, \\
\mu_\xi(B_0+\beta B_1+B) &= \mu_\xi(B_0+\beta B_1)+((\xi-1)b_0(\pi), (\xi-1)b_1(\pi)) \in \pi\F[[\pi]]^S
\end{aligned}$$
for some $b_0(\pi), b_1(\pi) \in \F((\pi))$. 

Note that
$$\begin{aligned}
\val e_0\mu_\varphi(B_0+\beta B_1) &= 1-p \le \val e_1\mu_\varphi(B_0+\beta B_1), \\
\val e_0\mu_\xi(B_0+\beta B_1) &\ge 0, \val e_1\mu_\xi(B_0+\beta B_1) \ge 0.
\end{aligned}$$
Then $\val e_0\mu\varphi(B)=\val(Cb_1(\pi^p)-b_0(\pi))=1-p$, and we get either 
$\val b_0(\pi)=1-p < \val b_1(\pi^p)$ or $\val b_1(\pi^p)=\val b_0(\pi) < 1-p$.
In either case, we have 
$\val b_0(\pi^p) < 1-p$, so $\val b_0(\pi^p)=\val b_1(\pi)$, giving a contradiction.

The same argument proves that $\iota [B_1]$ is not bounded.

\end{proof}

Similarly one proves the following.

\begin{prop}\label{dim1'} If $f=2$, then
$$\begin{aligned}
V_{\{0\}} &= \left\{
\begin{array}{ll}
\F [B_0]  & {\it if} \,\, c_1=p-1, \\
\F [C\alpha_1B_0-\alpha_0B_1] & {\it if} \,\,  c_0=0, 0<c_1<p-1, \\
\F [B_1]  & 0< c_0 \le p-1, 0\le c_1 < p-1;
\end{array}\right. \\
V_{\{0\}}^+ &= \F[C\alpha_1B_0-\alpha_0B_1]; \\
V_{\{0\}}^- &= 0.
\end{aligned}$$
with $\pm$ occuring when $\vec{c}=\vec{0}$.
(See Lemma \ref{f2.1} for the definition of the $\alpha_i$.)
\end{prop}

In proving Propositions \ref{dim1} and \ref{dim1'} we have shown the
following, which exhibits instances of coincidence of $V_J$'s for
distinct $J$'s.

\begin{cor} Suppose $f=2$ and recall $(c_0,c_1) \neq (p-1,p-1)$.
\begin{enumerate}
\item If $c_0=p-1$, then $V_{\{1\}}=V_{\{0\}}=\F [B_1].$
\item If $c_1 =p-1$, then $V_{\{1\}}=V_{\{0\}}=\F [B_0].$
\item If $c_0=c_1=0$, then $V_{\{1\}}^+$ and $V_{\{0\}}^+$
 are distinct and one-dimensional,
 and $V_{\{1\}}^-= V_{\{0\}}^- = 0$.
\item In all other cases, $V_{\{1\}}$ and $V_{\{0\}}$
 are distinct and one-dimensional.
\end{enumerate}
\end{cor}

\section{Exceptional cases}
\subsection{Cyclotomic character} \label{sec:cyclo}
Assume $C=1,\vec{c}=\overrightarrow{p-2}$, so that
$$\begin{aligned}
\kappa_\varphi(C,\vec{c}) &= (\pi^{(p-1)(p-2)},\ldots,\pi^{(p-1)(p-2)}),\\
\kappa_\gamma(C,\vec{c}) 
&=\left(\left({\frac{\gamma(\pi)}{\pi\overline{\chi}(\gamma)}}\right)^{p-2},\ldots,
\left({\frac{\gamma(\pi)}{\pi\overline{\chi}(\gamma)}}\right)^{p-2}\right)
\end{aligned}$$
if $\gamma \in \Gamma$. Recall that $B_i$'s for all $i \in S$ have
already been constructed in \S 4.1 and we just need to construct an
additional basis element which we will denote $B_{\rm tr}$ (for
{\it tr\`es ramifi\'e}). Before we do this for
arbitrary $f\ge 1$, let's first consider the situation where $f=1$ (i.e.,
$K=\Q_p$) and  $\F = \F_p$ as a foundation for the general construction.
(We will go back to the general case $f \ge 1$ in the paragraph preceding Lemma~\ref{tr_basis}.)

\begin{lem}\label{cyc}
Let $\eta \in \Gamma$ be such that $\eta\Gamma_1$ generates $\Gamma/\Gamma_1 \simeq \F_p^\times$
and let $\chi(\xi) \equiv 1+z p \mod p^2$ with $0 < z \le p-1$.
If $s \in \Z$ is divisible by $p^v$ but not by $p^{v+1}$
for some $v \in \Z$, then
$$\begin{aligned}
\overline{\chi}(\eta)\eta(\pi^s)-\pi^s &\in ({\overline{\chi}(\eta)}^{s+1}-1)\pi^s+
\overline{s_v}{\frac{{\overline{\chi}(\eta)}^{s+1}({\overline{\chi}(\eta)}-1)}{2}}\pi^{s+
p^v}+\pi^{s+2p^{v}}\F_p[[\pi^{p^v}]], \\
\overline{\chi}(\xi)\xi(\pi^s)-\pi^s &\in \overline{s_vz}(\pi^{s+(p-1)p^v}+\pi^{s+p^{v+1}})
+ \pi^{s+p^{v+1}(p-1)}\F_p[[\pi^{p^v}]],
\end{aligned}$$
where $s = \sum_{j\ge v}s_jp^j$.
\end{lem}

\begin{proof} Similar to Lemma \ref{delta} and \ref{gamma}.
\end{proof}

\begin{lem}
There exists $h'(\pi) \in \pi^{1-2p}+\pi^{2-2p}\F[[\pi]]$ such that
$$(\overline{\chi}(\eta)\eta-1)(h'(\pi)) \in \F(\pi^{-p}-\pi^{-1})+\pi\F[[\pi]].$$
(Recall that $\eta$ is a topological generator of $\Gamma$.)
\end{lem}

\begin{proof} By Lemma \ref{cyc}, there exist
$\epsilon_{2-2p}, \ldots, \epsilon_{-1}, \epsilon_0 \in \F$
(unique if we set $\epsilon_{-p}=\epsilon_{-1}=0$) such that
$$(\overline{\chi}(\eta)\eta-1)(\pi^{1-2p}+\epsilon_{2-2p}\pi^{2-2p}+\cdots+\epsilon_{-1}\pi^{-1}+\epsilon_0)
\in \F\pi^{-p}+\F\pi^{-1}+\pi\F[[\pi]].$$
Set $h'(\pi)= \pi^{1-2p} +\epsilon_{2-2p}\pi^{2-2p} +\cdots+\epsilon_{-1}\pi^{-1}+\epsilon_0$, so
$$(\overline{\chi}(\eta)\eta-1)(h'(\pi)) \in  \alpha \pi^{-p} + \beta \pi^{-1} + \pi\F[[\pi]]$$
for some $\alpha,\beta\in\F$.  Writing $(\overline{\chi}(\xi)\xi-1) = 
\left(\sum_{i=0}^{p-2} \overline{\chi}(\eta)^i\eta^i\right)(\overline{\chi}(\eta)\eta-1)$
we find that
$$(\overline{\chi}(\xi)\xi-1)(h'(\pi)) \in - (\alpha \pi^{-p} + \beta \pi^{-1}) + \F[[\pi]].$$
On the other hand a direct computation shows that
$$(\overline{\chi}(\xi)\xi-1)(h'(\pi)) \in z(\pi^{-p}-\pi^{-1})+ \F[[\pi]]$$
where $z\in \F^\times$, so that $\alpha = \beta = -z$ and the lemma follows.
\end{proof}

Let $h'(\pi)$ be as in the lemma.  
Since $\varphi-1$ is bijective on $\pi\F[[\pi]]$, it follows that 
$$(\overline{\chi}(\eta)\eta-1)(h'(\pi)) \in  (\varphi-1)(g'_\eta(\pi))$$
for a unique $g'_\eta \in -z^{-1}\pi^{-1} + \pi\F[[\pi]]$.
We now extend the definition to construct elements
$g'_\gamma(\pi) \in \pi^{-1}\F[[\pi]]$ for all $\gamma\in \Gamma$.   We let
$$g'_{\eta^n}(\pi) = \sum_{i=0}^{n-1} \overline{\chi}(\eta)^i\eta^i(g'_\eta(\pi))$$
for $n\in \N$.  If $\gamma'\in \Gamma_2$, then
$(\overline{\chi}(\gamma')\gamma'-1)(h'(\pi))$ is in $\pi\F[[\pi]]$
and can therefore be written as $(\varphi-1)(g'_{\gamma'}(\pi))$ for a unique
$g'_{\gamma'}(\pi)\in \pi\F[[\pi]]$.  If $\eta^n\in \Gamma_2$, then $p(p-1)|n$
and the definitions coincide.  Moreover, an arbitrary $\gamma\in \Gamma$ can
be written as $\gamma'\eta^n$ for some $\gamma'\in \Gamma_2$ and $n\in\N$, and
$$g'_{\gamma}(\pi) := g'_{\gamma'}(\pi) + \gamma'(g'_{\eta^n}(\pi))$$
is independent of the choice of $\gamma'$ and $n$.

One then checks that
$\mu=(h'(\pi),(g_\gamma'(\pi))_{\gamma' \in \Gamma})$
satisfies conditions $(\dagger)$ and $(\ddagger)$, giving an
extension $$0\rightarrow M_{\rm cyc} \rightarrow E' \rightarrow M_0
\rightarrow 0$$ in the category of \'etale $(\varphi, \Gamma)$-modules over
$\E_K$, where $M_{\rm cyc}=\E_Ke_1$ is a rank one defined by
$\varphi(e_1')=e_1'$ and $\gamma(e_1')=\chi(\gamma)e_1'$ if $\gamma
\in \Gamma$ (and, of course, $M_0=\E_Ke_0$ by $\varphi(e_0)=e_0$ and
$\gamma(e_0)=e_0$). Using the isomorphism $M_{\rm cyc} \simeq M_{p-2}=\E_Ke_1$ defined
by $e_1'=\pi^{2-p}e_1$ we get an extension
$$0\rightarrow M_{p-2} \rightarrow E \rightarrow M_0
\rightarrow 0$$
defined by the cocycle
$\mu = (\pi^{3(1-p)}h(\pi),(\pi^{1-p}g_\gamma(\pi))_{\gamma\in \Gamma})$ with
$h(\pi)=\pi^{2p-1}h'(\pi)$, $g_\gamma(\pi)=\pi g_\gamma'(\pi)$.

Now we go back to the context of arbitrary $f\ge 1$, and define
$\mu_\varphi(B_{\rm tr})=(\pi^{3(1-p)}h(\pi),\ldots,\pi^{3(1-p)}h(\pi))$ and
$\mu_\gamma(B_{\rm tr})=(\pi^{1-p}g_\gamma(\pi),\ldots,\pi^{1-p}g_\gamma(\pi))$ for all $\gamma \in \Gamma$.
It is straightforward to check that $B_{\rm tr} \in H$, so that
$[B_{\rm tr}] \in \Ext^1(M_{\vec{0}},M_{\overrightarrow{p-2}})$.

\begin{rem} The class $[B_{\rm tr}]$ is not canonical.
Choosing $\epsilon_{-p} = - \epsilon_{-1} \neq 0$ in the proof of Lemma~6.2
gives different extension classes $[B_{\rm tr}]$ differing by a multiple of
$[B_0] + [B_1] + \dots + [B_{f-1}]$.
\end{rem}

\begin{lem}\label{tr_basis} The extensions $[B_0], \ldots, [B_{f-1}], [B_{\rm tr}]
\in \Ext^1(M_{\vec{0}},M_{\overrightarrow{p-2}})$ are linearly
independent, and therefore form a basis.
\end{lem}

\begin{proof}
It suffices to show that $[B_{\rm tr}]$ is not contained in the span
of $[B_i]$'s. Suppose $B_{\rm tr} = \beta_0B_0+ \cdots
+\beta_{f-1}B_{f-1}$ for some $\beta_i\in \F$. Then $E:=B_{\rm tr} -
(\beta_0B_0 + \cdots + \beta_{f-1}B_{f-1})$ is a coboundary, so that $$\mu_\varphi(E)=(\pi^{(p-1)(p-2)}b_1(\pi^p)-b_0(\pi),
\ldots,\pi^{(p-1)(p-2)}b_0(\pi^p)-b_{f-1}(\pi))$$ for some $b_i(\pi) \in \F((\pi))$. As
$$\begin{aligned}
\mu_\varphi(B_{\rm tr}) &= (\pi^{3(1-p)}h(\pi), \ldots, \pi^{3(1-p)}h(\pi)), \\
\mu_\xi(B_{\rm tr}) &= (\pi^{1-p}g_\xi(\pi), \ldots, \pi^{1-p}g_\xi(\pi))
\end{aligned}$$ where $h(\pi), g_\xi(\pi) \in \F[[\pi]]^\times$, we have
$\val e_i \mu_\varphi(E) = \val(\pi^{(p-1)(p-2)}b_{i+1}(\pi^p)-b_i(\pi))=3(1-p)$ for all $i \in S$. 
For each $i \in S$, letting $s_i:=\val(b_i(\pi))$, we have $s_{i} \le 3(1-p)$ or $(p-1)(p-2)+s_{i+1}p = 3(1-p)$.
The latter is impossible looking at divisibility by $p$, and so $s_i \le 3(1-p)$ for all $i \in S$,
which yields a contradiction after cycling.
\end{proof}

The determination of which linear combinations of $[B_0],[B_1],\ldots,[B_{f-1}]$ are bounded
is exactly as in the generic case.  We now extend this to include $[B_{\rm tr}]$.

\begin{prop} \label{prop:cyclo}
Suppose that $C=1$, $\vec{c}=\overrightarrow{p-2}$ and let $A \in \F^\times$ be given.
\begin{enumerate}
\item If $J=S$, then $$\iota [B_{\rm tr}] \in \Ext^1_{\bdd}(M_{A\vec{p}},M_{A\vec{0}}),$$ so that 
$ V_S^+=\Ext^1(M_{\vec{0}},M_{\overrightarrow{p-2}})$.
\item $V_S^- = \oplus_{i\in S} \F [B_i]$, and if $J\neq S$, then $V_J = \oplus_{i\in J} \F [B_{i+1}]$.
\end{enumerate}
\end{prop}

\begin{proof}
(1) Straightforward: as $\vec{a}=\vec{p}$ and $\langle\vec{c}\rangle_S = (\pi^{2(p-1)},\ldots,\pi^{2(p-1)})$, we have
$$\begin{aligned}
\mu_\varphi(\iota B_{\rm tr}) &= (A\pi^{(p-1)^2}h(\pi), \pi^{(p-1)^2}h(\pi),\ldots,\pi^{(p-1)^2}h(\pi)) \in \F[[\pi]]^S, \\
\mu_\xi(\iota B_{\rm tr}) &= (\lambda_\xi^{(p-2)\frac{p^f-1}{p-1}}, \ldots, \lambda_\xi^{(p-2)\frac{p^f-1}{p-1}})(\pi^{p-1}g_\xi(\pi), \ldots, \pi^{p-1}g_\xi(\pi))
\in \pi\F[[\pi]]^S.
\end{aligned}$$
(2) Let $E:=\beta_0B_0+\cdots+\beta_{f-1}B_{f-1}+B_{\rm tr}$ for some $\beta_0,\ldots,\beta_{f-1}\in \F$.
We must show that in all other cases where $\iota: \Ext^1(M_{\vec{0}},M_{\overrightarrow{p-2}})
\rightarrow \Ext^1(M_{A\vec{a}},M_{A\vec{b}})$ was defined, we have that $\iota [E]$ is not bounded.

So suppose that $\iota [E]$ is bounded.  Then
there exists a coboundary $B$ defined by $(b_0(\pi),\ldots,b_{f-1}(\pi))$ such that
$\mu_\varphi(\iota (E+B))\in\F[[\pi]]^S$ and $\mu_\xi(\iota (E+B))\in \pi\F[[\pi]]^S$.
We have $e_i\langle\vec{c}\rangle_J=1$ or $\pi^{p-1}$ and $\val e_i\mu_\xi(\iota E)\le 0$.
It follows that $\val e_i\mu_\xi(B)=\val e_i \mu_\xi(E) = 1-p$, so
by Lemma \ref{gamma}, we must have $s_i:=\val (b_i(\pi))\le 2(1-p)$.
Then $\val(\pi^{(p-1)c_i}b_i(\pi^p))=(p-1)(p-2)+s_ip \le (1-p)(p+2)$,
so that $s_{i-1}=(p-1)(p-2)+s_ip$. Cycling this through indices leads to a contradiction.

\end{proof}

\subsection{Trivial character} \label{sec:triv}
In this subsection, we assume that $C=1, \vec{c}=\vec{0}$, so that
$\kappa_\varphi(C,\vec{c})=\kappa_\gamma(C,\vec{c})=(1,\ldots,1) \in \F((\pi))^S$.

Using Lemma \ref{delta} we can find unique $\epsilon_{2-p},\ldots,\epsilon_{-1} \in \F$ such that
$$(\eta-1)(\pi^{1-p}+\epsilon_{2-p}\pi^{2-p}+\cdots+\epsilon_{-1}\pi^{-1}) \in \F[[\pi]].$$
Set $H(\pi)=\pi^{1-p}+\epsilon_{2-p}\pi^{2-p}+\cdots+\epsilon_{-1}\pi^{-1}$.
By Lemma \ref{gamma}, we get
$$(\xi-1)(H(\pi)) \in \F^\times+\pi\F[[\pi]],$$ which implies, via Lemma \ref{trick}, that
$$(\eta-1)(H(\pi)) \in \nu+\pi\F[[\pi]]$$ for some $\nu \in \F-\{0\}$. Likewise we have
$$(\eta-1)(H(\pi^p)) \in \nu + \pi\F[[\pi]],$$ so that
$$(\eta-1)(-H(\pi^p)+H(\pi)) \in \pi\F[[\pi]].$$

Note that if $\gamma' \in \Gamma_2$,  then
$(\gamma'-1)(H(\pi)) \in \pi\F[[\pi]]$, and it follows that
$(\gamma'-1)(-H(\pi^p)+H(\pi)) \in \pi\F[[\pi]]$.
Now for each $\gamma \in \Gamma$, writing $\gamma=\eta^n\gamma'$ where
$\gamma' \in \Gamma_2$, we get by Lemma \ref{val} that
$$(\gamma-1)(-H(\pi^p)+H(\pi)) \in \pi\F[[\pi]].$$
As the map $g(\pi) \mapsto g(\pi^{p^f})-g(\pi)$ defines a bijection
$\pi\F[[\pi]]\to\pi\F[[\pi]]$, for each $\gamma \in \Gamma$ there
exists a unique $g_{\gamma}(\pi) \in \pi\F[[\pi]]$ such
that
$$g_\gamma(\pi^{p^f})-g_\gamma(\pi) = (\gamma-1)(-H(\pi^p)+H(\pi)),$$
or equivalently,
$$(\varphi-1)(g_\gamma(\pi), g_\gamma(\pi^{p^{f-1}}),\ldots, g_\gamma(\pi^p))=(\gamma-1)(-H(\pi^p)+H(\pi), 0,\ldots,0).$$

If we set
$$\begin{aligned}
\mu_\varphi(B_0) &= (-H(\pi^p)+H(\pi), 0,\ldots, 0), \\
\mu_\gamma(B_0) &= (g_{\gamma}(\pi), g_{\gamma}(\pi^{p^{f-1}}),
\ldots, g_{\gamma}(\pi^p)),
\end{aligned}$$
$\mu(B_0)=(\mu_\varphi(B_0),(\mu_\gamma(B_0))_{\gamma \in \Gamma})$
satisfies the condition $(\dagger)$ by the considerations above. We
note that $\mu_\gamma(B_0)$ are uniquely determined so that they
satisfy $(\dagger)$. As both $\mu_{\gamma\gamma'}(B_0)$ and
$\mu_{\gamma\gamma'}':=\gamma(\mu_{\gamma'}(B_0))+\mu_\gamma(B_0)$
satisfy $(\dagger)$ for $\gamma\gamma'$, they must coincide, so that
($\ddagger$) is satisfied.

For each $1 \le i \le f-1$, we construct $[B_i] \in
\Ext^1(M_{\vec{0}},M_{\vec{0}})$ in a similar way, i.e., by setting
$$\begin{aligned}
\mu_\varphi(B_i) &= (0,\ldots, 0, -H(\pi^p)+H(\pi), 0, \ldots, 0), \\
\mu_\gamma(B_i) & = (g_\gamma(\pi^{p^{i}}), \ldots,
g_\gamma(\pi^{p}), g_\gamma(\pi), g_\gamma(\pi^{p^{f-1}}), \ldots,
g_\gamma(\pi^{p^{i+1}})).
\end{aligned}$$

\begin{rem} For each $0 \le i \le f-1$, consider the coboundary $B_i''$ by
$$\begin{aligned}
\mu_\varphi(B_i'') &=(0,\ldots,0,H(\pi^p),-H(\pi),0,\ldots,0), \\
\mu_\gamma(B_i'') &= (0,\ldots,0,0,(\gamma-1)(H(\pi)),0,\ldots,0),
\end{aligned}$$
where $H(\pi)$ is the $i$-th component and $-H(\pi)$ the $(i+1)$-th
component of $\mu_\varphi(B_i'')$ and $(\gamma-1)(H(\pi))$ is the
$(i+1)$-th component of $\mu_\gamma(B_i'')$. Define $B_i'=B_i+B_i''$
for each $0\le i \le f-1$. Then $\F[B_i]=\F [B_i']$ in $\Ext^1(M_{\vec{0}},M_{\vec{0}})$ for all $0 \le i
\le f-1$, where we have
$$\begin{aligned}
\mu_\varphi(B_i') &=(0,\ldots,0,H(\pi),-H(\pi),0,\ldots,0), \\
\mu_\gamma(B_i') &=  (g_\gamma(\pi^{p^{i}}), \ldots,
g_\gamma(\pi^{p}), g_\gamma(\pi), g_\gamma(\pi^{p^{f-1}})+(\gamma-1)H(\pi), g_\gamma(\pi^{p^{f-2}}), \ldots,
g_\gamma(\pi^{p^{i+1}})).
\end{aligned}$$
\end{rem}

Next, we define $B_{\rm nr}$ (for {\it non-ramifi\'e})
by setting
$$\begin{aligned}
\mu_\varphi(B_{\rm nr}) &= (1, 0,\ldots,0), \\
\mu_\gamma(B_{\rm nr}) &= (0,0,\ldots,0)
\end{aligned}$$
for all $\gamma \in \Gamma$. It is straightforward to check that
this defines an extension $[B_{\rm nr}] \in
\Ext^1(M_{\vec{0}},M_{\vec{0}})$. We can ``move'' the 1 in $\mu_\varphi$
to any component, i.e., taking any of $(1, 0,\ldots,0)$, $(0,1,0,\ldots,0), \ldots, (0,\ldots,0,1)$
to be $\mu_\varphi$ defines the same cocycle class (up to coboundaries).

Set $B_{\rm cyc} = \sum_{i=0}^{f-1} B_i'$.  Then we have
$$\begin{aligned}
\mu_\varphi(B_{\rm cyc}) &= (0,\ldots,0), \\
\mu_\gamma(B_{\rm cyc}) &= (g_\gamma'(\pi),\ldots,g_\gamma'(\pi))
\end{aligned}$$
for some $g_\gamma'\in\F[[\pi]]$.  Since $(\varphi-1)g'_\gamma(\pi)=0$,
we must have in fact $g'_\gamma(\pi) = g'_\gamma\in\F$.  In particular
$g'_\eta = \nu$.  Moreover $\gamma\mapsto g'_\gamma$ defines a homomorphism
$\Gamma \to \F$.  Thus if $\gamma=\eta^{n_\gamma}$ modulo
$\Gamma_2$, then
$$\mu_\gamma(B_{\rm cyc}) = \nu \overline{n}_\gamma(1,\ldots,1).$$

\begin{lem}\label{trv} The extensions $[B_{\rm nr}], [B_0], \ldots, [B_{f-1}]
\in \Ext^1(M_{\vec{0}},M_{\vec{0}})$ are linearly independent, and therefore
form a basis.
\end{lem}

\begin{proof}
Suppose that $E = \beta B_{\rm nr} + \beta_0 B_0+ \cdots + \beta_{f-1} B_{f-1}$ is a coboundary.
By adding some coboundary $B$ we have
$$\begin{aligned}
e_0\mu_\varphi(E+B) 
&=\beta + \beta_0(-H(\pi^p) + H(\pi)) + \beta_1(-H(\pi^{p^2})+H(\pi^p)))+\cdots \\
&\,\,\,\,\,\,\,+\beta_{f-2}(-H(\pi^{p^{f-1}})+H(\pi^{p^{f-2}}))
+\beta_{f-1}(-H(\pi^{p^{f}})+H(\pi^{p^{f-1}})) \\
&= \beta + \beta_0H(\pi) + (\beta_1-\beta_0)H(\pi^p) + \cdots \\
&\,\,\,\,\,\,\,+ (\beta_{f-1}-\beta_{f-2})H(\pi^{p^{f-1}})-\beta_{f-1}H(\pi^{p^f}) \\
&=(\Phi-1)(\sum_{j\ge s}b_j\pi^j)\end{aligned}$$
for some $\sum_{j\ge s}b_j\pi^j \in \F((\pi))$.
Equating constant terms gives $\beta = 0$.
If $\beta_{f-1} \neq 0$, then $s = 1-p$, 
$\beta_0H(\pi) = -(b_{1-p}\pi^{1-p} + \cdots + b_{-1}\pi^{-1})$ and
$\beta_0=\beta_1= \cdots = \beta_{f-1}$.
It follows that $E = \beta_{f-1}\sum_{i=0}^{f-1}B_i$  is cohomologous to
$\beta_{f-1}B_{\rm cyc}$, and therefore that $B_{\rm cyc}$ is coboundary.  
Thus there exists $(b_0(\pi), \ldots, b_{f-1}(\pi)) \in \F((\pi))^S$ such that
$$\begin{aligned}
(\varphi-1)(b_0(\pi), \ldots, b_{f-1}(\pi)) &=(0, \ldots, 0), \\ 
(\xi-1)(b_0(\pi), \ldots, b_{f-1}(\pi)) &= -\nu(1,\ldots,1),
\end{aligned}$$
which is impossible as the former implies $b_0(\pi)=\cdots=b_{f-1}(\pi) \in \F$,
so that $(\xi-1)(b_0(\pi), \ldots, b_{f-1}(\pi)) = 0$.
Thus, $\beta_{f-1}=0$.

If $0\le i \le f-2$ is the largest such that $\beta_i\neq 0$, then $\val(e_0\mu_\varphi(E+B))=p^{i+1}(1-p)$,
which leads to an easy contradiction. Thus, $\beta_i=0$ for all $0\le i \le f-2$.
\end{proof}

We now assume $f=2$ and compute the spaces of bounded extensions. 
We then have the following cases to consider:
\begin{itemize}
\item $J=S$, $a_0 = a_1 = p -1$;
\item $J = \{1\}$, $b_0 = 1$, $a_1=p$ (for $V_J^+$);
\item $J = \{1\}$, $b_0 = p$, $a_1=1$ (for $V_J^-$);
\item $J = \{0\}$, $a_0 = p$, $b_1=1$ (for $V_J^+$);
\item $J = \{0\}$, $a_0 = 1$, $b_1=p$ (for $V_J^-$);
\item $J=\emptyset$, $b_0 = b_1 = p - 1$.
\end{itemize}

\begin{prop} Suppose that $f=2$, $C=1$, $\vec{c}=\vec{0}$ and $A\in\F^\times$.
\begin{enumerate}
\item $V_S = \Ext^1(M_{\vec{0}},M_{\vec{0}})$;
\item $V_{\{i\}}^+ = \langle [B_{\rm nr}], [B_i] \rangle$ for $i = 0,1$;
\item $V_{\{i\}}^- = \langle [B_{\rm nr}] \rangle$ for $i = 0,1$;
\item $V_\emptyset = \{0\}$.
\end{enumerate}
\end{prop}
\begin{proof}
(1) We have $\langle \vec{c} \rangle_S = (\pi^{p-1},\pi^{p-1})$ and it is
strightforward to check that $\iota [B_0]$, $\iota [B_1]$ and $\iota [B_{\rm nr}]$ are bounded.

(2) Suppose $J = \{1\}$.  Then $b_0 = 1$, $a_1 = p$ and $\langle \vec{c} \rangle_{\{1\}} = (\pi^{p-1},1)$
and it is straightforward to check that $\iota [B_1]$ and $\iota [B_{\rm nr}]$ are bounded.
Therefore it suffices to prove that $\iota [B_{\rm cyc}]$ is not bounded.
So suppose that $B$ is a coboundary such that $\iota (B_{\rm cyc}+B)$
has $\mu_\varphi\in \F[[\pi]]^S$ and $\mu_\xi\in \pi\F[[\pi]]^S$.
Then 
$$\mu_\varphi(B_{\rm cyc}+B) = \mu_\varphi(B) = (b_1(\pi^p)-b_0(\pi),b_0(\pi^p)-b_1(\pi))$$
for some $b_0(\pi),b_1(\pi)\in\F((\pi))$, and $(A\pi^{p-1},\pi^{p(p-1)})\mu_\varphi(B)\in\F[[\pi]]^S$.
Letting $v_0 =\val(b_0(\pi))$ and $v_1=\val(b_1(\pi))$, we see that $v_0\ge 1-p$
and $v_1 \ge 0$.  Therefore
$$\mu_\xi(B_{\rm cyc}+B) = (-\nu + (\xi-1)b_0(\pi), -\nu+(\xi-1)b_1(\pi)),$$
and since $-\nu+(\xi-1)b_1(\pi)$ has constant term $-\nu$, we arrive at
a contradiction.

The case $J=\{0\}$ is the same.

(3) Suppose again that $J=\{1\}$.  Now we have $b_0 = p$, $a_1=1$ and
$\langle \vec{c} \rangle_{\{1\}} = (1,\pi^{1-p})$
and it is clear that $\iota [B_{\rm nr}]$ is bounded.  Therefore it suffices to prove
that if $E = \beta_0 B_0 + \beta_1 B_1$ with $\beta_0,\beta_1$ is such that $\iota [E]$
is bounded, then $\beta_0 = \beta_1 = 0$.  The argument in the proof of Lemma~\ref{trv}
shows that $\beta_0 = \beta_1$, so we are reduced to proving that $\iota [B_{\rm cyc}]$
is not bounded.  The proof of this similar to part (2).

The case $J=\{0\}$ is the same.

(4) Now we have $\langle \vec{c} \rangle_\emptyset = (\pi^{1-p},\pi^{1-p})$, and if $\iota [E]$
is bounded then $\mu_\varphi(E+B) \in \pi^{p-1}\F[[\pi]]^S$ for some coboundary $B$.
The proof of Lemma~\ref{trv} then shows that $E$ is cohomologous to a multiple of
$B_{\rm cyc}$, and the boundedness of $\iota [B_{\rm cyc}]$ yields a contradiction
as above.
\end{proof}

\subsection{$p=2$}

We assume $p=2$ throughout this section.  Now $\Gamma$ is not pro-cyclic; we
write $\Gamma = \Delta \times \Gamma_2$ where $\Delta = \langle \eta \rangle$
with $\chi(\eta) = -1$, so $\Delta$ has order $2$, and we choose a topological 
generator $\xi$ of $\Gamma_2$.

\begin{lem} \label{p2lambda} We have $\lambda_\eta \equiv 1 + \pi \bmod \pi^{2^f}\F[[\pi]]$.
If $\gamma \in \Gamma_2$, then $\lambda_\gamma\equiv 1 \bmod \pi^3\F[[\pi]]$.
\end{lem}
\begin{proof} The first assertion follows from the fact that 
$$\lambda_\eta^{2^f-1}= \eta(\pi)/\pi = (1+\pi)^{-1}.$$
For the second assertion, note that if $\gamma\in \Gamma_2$, then
$\chi(\gamma) \equiv 1 \bmod 4$, so $\gamma(\pi)/\pi \equiv 1 \bmod \pi^3\F[[\pi]]$.
\end{proof}

Let $C \in \F^\times$ and $\vec{c}=(c_0,\ldots,c_{f-1}) \in \{0,1\}^S$ with some $c_j=0$ be given.
First assume that $C\neq 1$ if $\vec{c}=\vec{0}$, so that $C\pi^{\Sigma_j\vec{c}}\Phi-1 : \F[[\pi]]\to\F[[\pi]]$
defines a valuation-preserving bijection for all $j \in S$.   As in the case $p > 2$,
we will define for each $i\in S$ an element $H_i(\pi)\in \F((\pi))$ such that
$$(\lambda_\gamma^{\Sigma_i\vec{c}}\gamma - 1) H_i(\pi) \in \F[[\pi]]$$
for all $\gamma\in \Gamma$.   If $c_i = 0$, we let $H_i(\pi) = \pi^{-1}$;
otherwise we use the following lemma:
\begin{lem}  \label{p2H} Suppose that $c_i = 1$, and $r \in 0,\ldots,f-1$ is
such that $c_{i+1}=\cdots=c_{i+r}=0$ and $c_{i+r+1}=1$.  Let
$$H_i(\pi) = \pi^{1-2^{r+2}} + \pi^{1+2^r-2^{r+2}}.$$
Then $(\lambda_\gamma^{\Sigma_i\vec{c}}\gamma - 1) H_i(\pi) \in \F[[\pi]]$
for all $\gamma\in \Gamma$.
\end{lem}
\begin{proof} Note that we can assume $f \ge 2$.  We have
 $$\lambda_\gamma^{\Sigma_i}\gamma \pi^{1-2^{r+2}}
 = \lambda_\gamma^{\Sigma_i}\left(\frac{\gamma(\pi)}{\pi}\right)^{1-2^{r+2}} \pi^{1-2^{r+2}}
 = \lambda_\gamma^{\Sigma_i+ (2^f-1)(1-2^{r+2})}\pi^{1-2^{r+2}}.$$
Note that $\Sigma_i = 1$ if $r=f-1$ and $\Sigma_i \equiv 1 + 2^{r+1}\bmod 2^{r+2}$ otherwise.
In either case we have $\Sigma_i+ (2^f-1)(1-2^{r+2}) \equiv 2^{r+1} \bmod 2^{r+2}$.
It follows that
$$(\lambda_\gamma^{\Sigma_i}\gamma -1)(\pi^{1-2^{r+2}})
  \equiv (\lambda_\gamma^{2^{r+1}}-1)\pi^{1-2^{r+2}} \bmod \F[[\pi]].$$
Similarly we find that
$$(\lambda_\gamma^{\Sigma_i}\gamma -1)(\pi^{1+2^r-2^{r+2}})
  \equiv (\lambda_\gamma^{2^r}-1)\pi^{1+2^r-2^{r+2}} \bmod \F[[\pi]].$$

Lemma~\ref{p2lambda} gives 
$\lambda_\eta^{2^s} \equiv 1 + \pi^{2^s}\bmod \pi^{2^{s+f}}$ for $s\ge 0$,
and it follows that
$$(\lambda_\eta^{2^{r+1}}-1)\pi^{1-2^{r+2}}
  \equiv (\lambda_\eta^{2^r}-1)\pi^{1+2^r-2^{r+2}}
  \equiv \pi^{1+2^{r+1}-2^{r+2}} \bmod \F[[\pi]].$$
Therefore the lemma holds for $\gamma = \eta$.  We also get that
$\lambda_\gamma^{2^s} \equiv 1 \bmod \pi^{3\cdot 2^s}$ for
$\gamma\in\Gamma_2$, from which it follows that
$(\lambda_\gamma^{2^{r+1}}-1)\pi^{1-2^{r+2}}$ and
$(\lambda_\gamma^{2^r}-1)\pi^{1+2^r-2^{r+2}}$ are in $\F[[\pi]]$.
The lemma therefore holds for $\gamma\in \Gamma_2$ as well, and
we deduce from Lemma~\ref{val} that it holds for all $\gamma\in \Gamma$.
\end{proof}

By the bijectivity of $C\pi^{\Sigma_0\vec{c}}\Phi-1$, for each $\gamma \in \Gamma$
we have a unique $G_i(\pi)=G_{i,\gamma}(\pi) \in \F[[\pi]]$ such that
$(C\pi^{\Sigma_0\vec{c}}\Phi-1)(G_i(\pi))=(\lambda_\gamma^{\Sigma_i\vec{c}}\gamma-1)(H_i(\pi))$.
Then letting
$$\begin{aligned}
\mu_\varphi(B_i) &= (0,\ldots,0,H_i(\pi),0,\ldots,0), \\
\mu_\gamma(B_i) &= (G_0(\pi),\ldots,G_i(\pi),\ldots, G_{f-1}(\pi)),
\end{aligned}$$
where
$$\begin{aligned}
G_0(\pi)     &= C\pi^{c_0+2c_1+\cdots+2^{i-1}c_{i-1}}G_i(\pi^{2^i}), \\
G_1(\pi)     &= \pi^{c_1+2c_2+\cdots+2^{i-2}c_{i-1}}G_i(\pi^{2^{i-1}}), \\
             &\vdots \\
G_{i-1}(\pi) &= \pi^{c_{i-1}}G_i(\pi^2), \\
G_{i+1}(\pi) &= C\pi^{c_{i+1}+2c_{i+2}\cdots+2^{f-2}c_{i-1}}G_i(\pi^{2^{f-1}}), \\
             &\vdots \\
G_{f-1}(\pi) &= C\pi^{c_{f-1}+2c_0+\cdots+2^ic_{i-1}}G_i(\pi^{2^{i+1}}),
\end{aligned}$$
gives rise to an extension $[B_i] \in \Ext^1(M_{\vec{0}},M_{C\vec{c}})$.
By almost identical arguments to the case $p>2$, one finds that $[B_0],\ldots,[B_{f-1}]$ are 
linearly independent, so that they form a basis.

Now suppose $C=1$ and $\vec{c}=\vec{0}$. We can define, similarly to the $p>2$ case,
$[B_0], \ldots, [B_{f-2}], [B_{f-1}]$ such that
$$\begin{aligned}
\mu_\varphi(B_0) &=(\pi^{-2}+ \pi^{-1}, 0,\ldots, 0), \\
\mu_\varphi(B_1) &=(0, \pi^{-2}+ \pi^{-1}, 0,\ldots, 0), \\
                 & \vdots                            \\
\mu_\varphi(B_{f-1}) &=(0,\ldots, 0, \pi^{-2}+\pi^{-1}). \\
\end{aligned}$$
As before each $B_i$ is cohomologous to $B_i'$ with
$$\mu_\varphi(B_i) =(0,\ldots,0,\pi^{-1},\pi^{-1}, 0,\ldots, 0),$$
the non-zero entries being in the $i,i+1$ coordinates  (unless $f=1$,
in which case $\mu_\varphi(B_0)=0$).  We again set 
$B_{\rm cyc} =  \sum_{i=0}^{f-1} B_i'$, and define a cocycle
$B_{\rm nr}$ by setting
$$\begin{aligned}
\mu_\varphi(B_{\rm nr}) &= (1, 0,\ldots,0), \\
\mu_\gamma(B_{\rm nr}) &= (0,0,\ldots,0)
\end{aligned}$$
for all $\gamma \in \Gamma$.

The difference now is that if $p=2$, then $\dim_\F\Ext^1(M_{\vec{0}},M_{\vec{0}}) = f+2$,
so we need one more basis element.  We define $B_{\rm tr}$ by
$$\begin{aligned} \mu_\varphi(B_{\rm nr}) &= (0,0,\ldots,0), \\
\mu_\gamma(B_{\rm nr}) &= n_\gamma(1,1,\ldots,1)\end{aligned}$$
where $n_\gamma = 0$ if $\gamma \in \Gamma_3 \cup \eta\Gamma_3$, and
$n_\gamma = 1$ otherwise (so $\gamma \mapsto n_\gamma$ defines a
homomorphism $\Gamma \to \F$).  One checks as in the case $p>2$
that the elements $[B_{\rm nr}], [B_0], [B_1],\ldots,[B_{f-1}], [B_{\rm tr}]$
are linearly independent, hence form a basis for $\Ext^1(M_{\vec{0}},M_{\vec{0}})$.

Finally we assume $f=2$ and compute the spaces of bounded extensions.
There are three possibilites to consider:
\begin{enumerate}
\item $\vec{c} = (0,1)$ or $(1,0)$;
\item $\vec{c} = (0,0)$ and $C\neq 1$;
\item $\vec{c} = (0,0)$ and $C=1$.
\end{enumerate}

We omit the proofs of the following which are essentially the same as for $p>2$:

\begin{prop} \label{p201} If $\vec{c} = (0,1)$ or $(1,0)$,
then 
\begin{itemize}
\item $V_S = \Ext^1(M_{\vec{0}},M_{C\vec{c}})$;
\item if $\vec{c}= (0,1)$, then $V_{\{0\}} = V_{\{1\}} = \F [B_0]$;
\item if $\vec{c}= (1,0)$, then $V_{\{0\}} = V_{\{1\}} = \F [B_1]$;
\item $V_\emptyset = 0$.
\end{itemize}
\end{prop}

\begin{prop} \label{p200} If $\vec{c} = (0,0)$ and $C\in \F^\times$ with $C\neq 1$,
then 
\begin{itemize}
\item $V_S^+ = V_S^- = \Ext^1(M_{\vec{0}},M_{C\vec{0}})$;
\item $V_{\{1\}}^+ = \F[B_0+B_1]$;
\item $V_{\{0\}}^+ = \F[CB_0+B_1]$;
\item $V_{\{1\}}^-= V_{\{0\}}^-=V_\emptyset^+=V_\emptyset^- = 0$.
\end{itemize}
\end{prop}

\begin{prop} \label{p2triv} If $\vec{c} = (0,0)$ and $C=1$,
then 
\begin{itemize}
\item $V_S^+ = V_S^- = \Ext^1(M_{\vec{0}},M_{\vec{0}})$;
\item $V_{\{i\}}^+ = \F [B_{\rm nr}] \oplus \F [B_i]$ for $i=0,1$;
\item $V_{\{i\}}^- = \F [B_{\rm nr}]$ for $i=0,1$;
\item $V_\emptyset^+=V_\emptyset^- = 0$.
\end{itemize}
\end{prop}

\begin{rem} \label{rem:p2} With a view towards relating bounded extensions to crystalline ones,
we would have liked  $V_S^- = \F [B_{\rm nr}] \oplus \F [B_0] \oplus \F [B_1]$ in the
trivial case.  This could have been achieved with a more restrictive definition
of boundedness, requiring for example that $\mu_\gamma \in \pi^2\F[[\pi]]^S$
for $\gamma\in \Gamma_2$ if $p=2$.  However we opted instead for the definition
we found most uniform and easiest to work with.
\end{rem}

\section{Crystalline $\Rightarrow$ bounded}\label{sec:crys}
The paper \cite{BDJ05} formulates conjectures concerning
weights of mod $p$ Hilbert modular forms in terms of the associated
local Galois representations $G_K \to \gl_2(\F)$.  When the
local representation is reducible, i.e., of the form 
$\left(\begin{array}{cc}\chi_1&*\\0&\chi_2\end{array}\right)$,
the set of weights is determined by the associated class in $H^1(G_K,\F(\chi_1\chi_2^{-1}))$,
or more precisely whether the class lies in certain distinguished subspaces.
These subspaces are defined in terms of reductions of crystalline extensions
of crystalline characters.  Our aim is to relate these to the spaces
of bounded extensions we computed in the preceding sections.
The idea is to show that Wach modules over $\A_{K,F}^+$ associated to crystalline
extensions have bounded reductions.  This is easily seen to be true
when the Wach module itself is the extension of two Wach modules;
the problem is that this is not always the case.  Recall that Theorem~\ref{berger}
establishes an equivalence of categories between crystalline representations
and Wach modules over $\B_K^+$.  We note however that $\N$ does not define
an exact functor from $G_K$-stable lattices to
$\A_K^+$-modules.

\begin{ex}  \label{nonext} Let $K=\Q_p$ and $V = \Q_p(1-p) \oplus \Q_p$.  The corresponding
Wach module is $\N(V) = \B_{\Q_p}^+ e_1 \oplus \B_{\Q_p}^+e_2$ with 
\begin{itemize}
\item $\varphi(e_1) = q^{p-1}e_1$ and $\gamma(e_1) = (\gamma(\pi)/\chi(\gamma)\pi)^{p-1}e_1$
 for $\gamma \in \Gamma$;
\item $\varphi$ and $\Gamma$ acting trivially on $e_2$.
\end{itemize}
Let $f_1 = p^{-1}(e_1 - \pi^{p-1} e_2)$ and consider the $\A_{\Q_p}^+$-lattice
$N = \A_{\Q_p}^+f_1 \oplus \A_{\Q_p}^+e_2$ in $\N(V)$.
Then it is straightforward to check that $N$ is a Wach module over $\A_{\Q_p}^+$,
hence corresponds to a $G_{\Q_p}$-stable lattice $T$ in $V$.  Such a lattice
necessarily fits into an exact sequence
$$0 \to \Z_p(1-p) \to T \to \Z_p \to 0$$
of $\Z_p$-representations of $G_{\Q_p}$, but there is no 
surjective morphism $\alpha: N \to \A_{\Q_p}^+$.  Indeed the image would 
have to be generated over $\A_{\Q_p}^+$ by elements $\alpha(f_1)$ and
$\alpha(e_2)$ satisfying $p\alpha(f_1) = - \pi^{p-1}\alpha(e_2)$, and
hence could not be free over $\A_{\Q_p}^+$.   This example is somewhat
special since $V$ is split and $T$ can also be written as an extension
$$0 \to \Z_p \to T \to \Z_p(1-p) \to 0,$$
which does correspond to an extension of Wach modules.  However
it illustrates the problem, which we shall see also occurs for lattices
in non-split extensions of $\Q_p$-representations.
\end{ex}

We will prove under certain hypotheses that the relevant extensions of
$\Z_p$-representations do in fact correspond to extensions of Wach modules.
In particular we will show this holds in the generic case, and in all but a few
special cases when $f=2$.  As a result, we will be able to give a complete
description of the distinguished subspaces in \cite{BDJ05} in terms of 
$(\varphi,\Gamma)$-modules in the generic case and the case $f=2$.

\subsection{The extension lemma} \label{extlemma}

We first establish a general criterion for a Wach module over $\A_{K,F}^+$ to arise
from an extension of two Wach modules.  We consider extensions of 
crystalline representations of arbitrary dimension since it is no more
difficult than the case of one-dimensional representations.

Suppose that we have an exact sequence 
$$0 \to V_1 \to V \to V_2 \to 0$$
of crystalline $\Q_p$-representations of $G_K$ with Hodge-Tate weights in $[0,b]$
for some $b\ge 0$.  We shall identify $V_1$ with a subrepresentation of $V$.
By Theorem~\ref{berger}, we have an exact sequence of corresponding Wach modules
over $\B_K^+$:
$$0 \to M_1 \to M \to M_2 \to 0$$
where $M = \N(V)$, $M_1 = \N(V_1)=\N(V)\cap\D(V_1)$ and $M_2$
is the image of $\N(V)$ in $\D(V_2)$.
Now suppose that $T$ is a $G_K$-stable lattice in $V$.
Letting $T_1 = T \cap V_1$ and $T_2=T/T_1$, we have an exact sequence 
$$0 \to T_1 \to T \to T_2 \to 0$$
of $\Z_p$-representations of $G_K$.
Letting $N = \N(T) = M \cap \D(T)$ be the Wach module in $M=\N(V)$ corresponding to $T$,
we see that $N_1 := N \cap M_1 = \N(T_1)$ since
$$N\cap M_1 = \N(T)\cap \D(V_1) =
                        \D(T) \cap \N(V) \cap \D(V_1) =
                        \D(T) \cap \N(V_1) = \D(T_1)\cap \N(V_1).$$
The quotient $N_2:= N/N_1$ is a finitely generated torsion-free $\A_K^+$-module
with an action of $\varphi$ and $\Gamma$ such that $q^bN_2\subset\varphi^*(N_2)$
and $\Gamma$ acts trivially on $N_2/\pi N_2$.  (Note that $N_2$ is torsion-free
since $N/N_1 \hookleftarrow M_2$, but $N_2$ is not necessarily free as we will
see below.)  Furthermore $\N(T)\to\N(T_2)$
induces an injective homomorphism $N_2 \to \N(T_2)$ which becomes an
isomorphism on tensoring with $\B_K^+$.  

Letting $\E_K^+ = \A_K^+/p\A_K^+$, $\overline{N} = N/pN$ and $\overline{N}_i
 = N_i/pN_i$, we know also that
$$\overline{N}[1/\pi] = \E_K \otimes_{\E_K^+} \overline{N}\quad{\mbox{and}}\quad
  \overline{N}_i[1/\pi] = \E_K \otimes_{\E_K^+} \overline{N}_i$$
for $i=1,2$ are the $(\varphi,\Gamma)$-modules over $\E_K$
corresponding to the reductions mod $p$ of the corresponding $G_K$-stable
lattices.  Moreover $\overline{N}_1$ and $\overline{N}$ are
free over $\E_K^+$ and the homomorphism $\overline{N}_1 \to \overline{N}$
is injective; we identify $\overline{N}_1$ with a submodule of $\overline{N}$.

\begin{lem} \label{lem:ext} The following are equivalent:
\begin{enumerate}
\item the homomorphism $\N(T) \to \N(T_2)$ is surjective;
\item $N_2 = \N(T)/\N(T_1)$ is free over $\A_K^+$; 
\item $\overline{N}_1 = \overline{N}\cap \D(T_1/pT_1)$.
\end{enumerate}
\end{lem}
\begin{proof}  If $\N(T) \to \N(T_2)$ is surjective, then $N_2 \cong \N(T_2)$
is free over $\A_K^+$.  Conversely if $N_2$ is free, then $\N(T)$ maps onto
a Wach module over $\A_K^+$ in $\N(V_2)$, which by Theorem~\ref{berger} is
of the form $\N(T_2')$ for some $G_K$-stable lattice $T_2'$ in $V_2$;
moreover $\N(T_2') \subset \N(T_2)$ implies that $T_2'\subset T_2$.
On the other hand, since $\N(T)$ maps to $\N(T_2')$, $\D(T)$ maps to
$\D(T_2')$, hence $T$ maps to $T_2'$, and therefore $T_2 =  T_2'$.

Since $\B_K^+\otimes_{\A_K^+}N_2 \cong \N(V_2)$ is free of rank $d_2:=\dim_{\Q_p}V_2$
over $B_K^+$, it follows from Nakayama's Lemma that $N_2$ is free over $\A_K^+$
if and only if $N_2/pN_2 = \overline{N}/\overline{N}_1$ is free of rank $d_2$
over $\E_K^+$.   Since $\overline{N}$ and $\overline{N}_1$ are free over $\E_K^+$
and the difference of their ranks is $d_2$, this in turn is equivalent to
$\overline{N}/\overline{N}_1$ being torsion-free over $\E_K^+$, which in turn
is equivalent to $\overline{N}_1 = \overline{N} \cap \overline{N}_1[1/\pi]$.
\end{proof}

\begin{ex} Returning to Example~\ref{nonext}, note that since $e_1-\pi^{p-1}e_2 \in pN$,
we have $\pi^{p-1}\overline{e}_2 =  -\overline{e}_1 \in \overline{N}$, so
$\overline{e}_2 = -\pi^{1-p}\overline{e}_1\in \overline{N}_1'$,
where $\overline{N}_1' = \overline{N} \cap \overline{N}_1[1/\pi]$. 
Thus we find in this case that $\overline{N}_1 = \F_p[[\pi]]\overline{e}_1$, but
$\overline{N}_1'=\pi^{1-p}\F_p[[\pi]]\overline{e}_1$,
so the criterion of the lemma is not satisfied.
\end{ex}

We remark that everything above holds with coefficients; in particular if
$$0 \to T_1 \to T \to T_2 \to 0$$
is an exact sequence of $G_K$-stable $\CO_F$-lattices
in crystalline representations, then the sequence
$$0\to \N(T_1) \to \N(T) \to \N(T_2) \to 0$$
of $\A_{K,F}^+$-modules is exact if and only if
$$\N(T_1)/\varpi_F\N(T_1)  = (\N(T)/\varpi_F\N(T)) \cap \D(T_1/\varpi_F T_1).$$

\subsection{Extensions of rank one modules}
We now specialize to the case where $V_1$ and $V_2$ are one-dimensional over $F$,
with labelled Hodge-Tate weights $(b_{f-1},b_0,\ldots,b_{f-2})$ and 
$(a_{f-1},a_0,\ldots,a_{f-2})$ where each $a_i,b_i \ge 0$.
Suppose that we have an exact sequence
$$0 \to V_1 \to V \to V_2 \to 0$$
of crystalline $F$-representations of $G_K$, and $T$ is a $G_K$-stable
$\CO_F$-lattice in $V$.
We thus have exact sequences
$$0 \to T_1 \to T \to T_2 \to 0\quad\mbox{and}\quad 0 \to \overline{T}_1
     \to \overline{T} \to \overline{T}_2 \to 0$$
where each $T_i$ is a $G_K$-stable $\CO_F$-lattices in $V_i$ and
$\bar{\cdot}$ denotes reduction modulo $\varpi_F$.
We let $N = \N(T)$ be the Wach module over $\A_{K,F}^+$ corresponding
to $T$, and $\overline{N}$ its reduction modulo $\varpi_F$.
Thus $\overline{N}$ is a free rank two $\E_{K,F}^+$-module with
an action of $\varphi$ and $\Gamma$ such that $\Gamma$ acts
trivially modulo $\overline{N}/\pi\overline{N}$.
Furthermore $\E_{K,F}\otimes_{\E_{K,F}^+}\overline{N} \cong
\D(\overline{T})$ as $(\varphi,\Gamma)$-modules over $\E_{K,F}$.
Letting $\overline{N}_1' = \D(\overline{T}_1) \cap \overline{N}$
and $\overline{N}_2' = \overline{N}/\overline{N}_1'$, we see that
each $\overline{N}_i'$ is an $\E_{K,F}^+$-lattice in $\D(\overline{T}_i)$,
stable under $\varphi$ and $\Gamma$ with $\Gamma$ acting trivially modulo
$\pi$.  

From the classification of rank one $(\varphi,\Gamma)$-modules
over $\E_{K,F}$, we know that 
$\D(\overline{T}_1)\cong M_{C\vec{c}}= \E_{K,F}e$
for some $C\in\F^\times$ and $\vec{c}\in\Z^S$.  Under this
isomorphism, $\overline{N}_1'$ corresponds to a submodule
of the form $(\pi^{r_0},\pi^{r_1},\ldots,\pi^{r_{f-1}})\E_{K,F}^+e$.
Since $\Gamma$ acts trivially on $\E_{K,F}^+e/\pi\E_{K,F}^+e$ and
on $\overline{N}_1'/\pi\overline{N}_1'$, we see that 
$(p-1)|r_i$ for $i=0,\ldots,f-1$.  Moreover
$$\varphi^*(\overline{N}_1') = (\pi^{(p-1)b'_0},\ldots,\pi^{(p-1)b'_{f-1}})\overline{N}_1'$$
for some $b_0',\ldots,b_{f-1}'$, all non-negative since $\overline{N}_1'$
is stable under $\varphi$.  Similarly we have
$$\varphi^*(\overline{N}_2') = (\pi^{(p-1)a'_0},\ldots,\pi^{(p-1)a'_{f-1}})\overline{N}_2'$$
for some $a_0',\ldots,a_{f-1}'\ge 0$.

For the following proposition, recall that $\Sigma_j(\vec{c}) = \sum_{i=0}^{f-1}c_{i+j}p^i$ where
$c_k$ is defined for $k\in \Z$ by setting $c_k=c_{k'}$ if $k\equiv k' \bmod f$.  We also
define a partial ordering on $\Z^S$ by $\vec{c}\le \vec{c}'$ if $c_i \le c_i'$ for all $i$.

\begin{prop}  \label{prop:rk2} With the above notation, we have:
\begin{enumerate}
\item $\min(a_i,b_i) \le a_i' \le \max(a_i,b_i)$, $\min(a_i,b_i) \le b_i' \le \max(a_i,b_i)$ 
  and $a_i' + b_i' = a_i + b_i$ for $i=0,\ldots,f-1$;
\item If $\vec{a}\le \vec{b}$ or $\vec{b}\le \vec{a}$, then $\{\vec{a},\vec{b}\} = \{\vec{a}',\vec{b}'\}$;
\item $\Sigma_j(\vec{a}') \ge \Sigma_j(\vec{a})$, $\Sigma_j(\vec{b}') \le \Sigma_j(\vec{b})$,
       $\Sigma_j(\vec{a}')\equiv \Sigma_j(\vec{a}) \bmod (p^f-1)$ and 
        $\Sigma_j(\vec{b}') \equiv \Sigma_j(\vec{b})\bmod (p^f-1)$ for $j=0,\ldots,f-1$;
\item $\vec{a} = \vec{a}'$ if and only if $\vec{b}=\vec{b}'$ if and only if $\N(T) \to \N(T_2)$
   is surjective.
\end{enumerate}
\end{prop}
\begin{proof} (1) We first prove that $a_i' + b_i' = a_i + b_i$ for $i=0,\ldots,f-1$.
The $\A_{K,F}^+$-module $\wedge^2_{\A_{K,F}^+}\N(T)$ inherits actions of $\varphi$
and $\Gamma$ making it a Wach module in $\wedge^2_{\B_{K,F}^+}\N(V)
\cong \N(\wedge^2_F V)$, hence it corresponds to an $\CO_F$-lattice in $\wedge^2_F V$.  The same is
true of $\N(T_1)\otimes_{\A_{K,F}^+} \N(T_2)$; since any two such lattices are scalar multiples of each other,
it follows that the corresponding Wach modules over $\A_{K,F}^+$ are isomorphic, and hence that
$$\overline{N}_1'\otimes_{\E_{K,F}^+}\overline{N}_2' \cong 
\wedge^2_{\E_{K,F}^+}\overline{N} \cong (\N(T_1)/\varpi_F\N(T_1))\otimes_{\E_{K,F}^+}(\N(T_2)/\varpi_F\N(T_2))$$
as $\E_{K,F}^+$-modules.  Moreover the isomorphisms are compatible with the action of $\varphi$, so
 $a_i'+b_i' = a_i+b_i$ for all $i$.

For the inequalities,  suppose first that $\min(a_i,b_i)=0$ for each $i$.  Since
$a_i + b_i = a_i' + b_i'$,  we know that $a_i'+b_i'\le\max(a_i,b_i)$
for each $i$, and the result follows.  The general case follows by twisting $T$ by a character with the
correct Hodge structure and $\N(T)$ by the corresponding Wach module.

(2)  By twisting we can again reduce to the case where $\min(a_i,b_i)=0$ for each $i$.  The
condition $\vec{a}\le \vec{b}$ or $\vec{b}\le\vec{a}$ becomes $\vec{a}$ or $\vec{b}=\vec{0}$, and we
must show that $\vec{a}'$ or $\vec{b}' = \vec{0}$.  The result then follows from the
equality $\vec{a}+\vec{b} = \vec{a}'+\vec{b}'$ proved in (1).

If $\vec{a}'\neq \vec{0}$ and $\vec{b}'\neq\vec{0}$, 
then $\varphi^f(\overline{N}'_i) \subset \pi\overline{N}_i'$ for $i=1,2$, so that
$\varphi^{2f}(\overline{N})\subset \pi\overline{N}$.  This means that $\varphi$ is
topologically nilpotent on $N$ in the sense that $\varphi(N) \subset (\pi,\varpi_F)N$
for some $n> 0$.

On the other hand, the $F$-representation $V$ of $G_K$ is {\em ordinary} in the sense
that there is an exact sequence
$$0 \to V_0 \to V \to V/V_0 \to 0$$
where $V_0$ is unramified, $V/V_0$ is positive crystalline, and each is one-dimensional
over $F$.  (If $\vec{b} = \vec{0}$, then take $V_0 = V_1$; if $\vec{b}\neq\vec{0}$ and
$\vec{a}=\vec{0}$, then the sequence $0\to V_1 \to V \to V_2 \to 0$ splits
and we can take $V_0$ to be the image of $V_2$.)  Since $\D_\crys(V_0) \subset
\D_\crys(V) \cong \N(V)/\pi\N(V)$ and $N/\pi N$ is a $\varphi$-stable lattice
in $\N(V)/\pi\N(V)$, we see that there is an element $e_0 \in N/\pi N$ such
that $e_0 \not\in \varpi_F(N/\pi N)$ and $\phi(e_0) = ue_0$ for some $u\in
(\CO_F\otimes\CO_K)^\times$.  Choosing a lift $\tilde{e}_0 \in N$ of $e_0$,
we have that $\varphi(\tilde{e}_0) \in u\tilde{e}_0 + (\pi,\varpi_F)N$,
contradicting that $\varphi$ is topologically nilpotent on $N$.

(3) Since $\overline{N}_1 = \N(T_1)/\varpi_F\N(T_1)$ is contained in $\overline{N}_1'$,
  we can write $\overline{N}_1 = (\pi^{t_0},\pi^{t_1},\ldots,\pi^{t_{f-1}})\overline{N}_1'$
  for some integers $t_0,t_1,\ldots,t_{f-1}\ge 0$.  We therefore have
  $$\begin{aligned}
    \phi^*(\overline{N_1})&=(\pi^{pt_1},\pi^{pt_2},\ldots,\pi^{pt_{f-1}},\pi^{pt_0})\phi^*(\overline{N}_1')\\
                          &=(\pi^{b_0'+pt_1},\pi^{b_1'+pt_2},\ldots,\pi^{b_{f-2}'+pt_{f-1}},\pi^{b_{f-1}'+pt_0})\overline{N}_1'.
                          \end{aligned}$$
  On the other hand, we also have
    $$\begin{aligned}
    \phi^*(\overline{N_1})&=(\pi^{b_0},\pi^{b_1},\ldots,\pi^{b_{f-1}})(\overline{N}_1)\\
                          &=(\pi^{b_0+t_0},\pi^{b_1+t_1},\ldots,\pi^{b_{f-1}+t_{f-1}})\overline{N}_1'.
                          \end{aligned}$$
  It follows that $b_j + t_j = b_j' + pt_{j+1}$
and thus $\Sigma_j(\vec{b}) + \sum_{i=0}^{f-1}t_{i+j}p^i = 
                   \Sigma_j(\vec{b}') + \sum_{i=0}^{f-1}t_{i+j+1}p^{i+1}$, and therefore that
    $\Sigma_j(\vec{b}) = t_j(p^f-1) + \Sigma_j(\vec{b}')$.  The assertions concerning
    $\Sigma_j(\vec{b}')$ follow, and those concerning $\Sigma_j(\vec{a}')$ then follow using (1).
    
(4)  We see from the proof of (3) that the hypotheses of Lemma~\ref{lem:ext} are satisfied
     if and only if $\overline{N}_1 = \overline{N}_1'$ if and only if $\vec{t}=0$.  On the other
     hand $\vec{b} = \vec{b}'$ if and only if $t_i = pt_{i+1}$ for $i=0,\ldots,f-1$, which implies
     that $t_i = p^ft_i$ for $i=0,\ldots,f-1$, hence is equivalent to $\vec{t}=\vec{0}$.
     That $\vec{a}=\vec{a}'$ if and only if $\vec{b} = \vec{b}'$ follows from (1).
\end{proof}

\subsection{Generic case}
In this subsection, we specialize to the generic case in the sense of \S\ref{sec:generic},
namely $0 < c_i < p-1$ for all $i$.  Recall that if $J\subset S$, then there are
integers $a_i$ and $b_i$ for $i\in S$ such that
\begin{itemize}
\item $1\le a_i \le p$ if $i\in J$, and $a_i = 0$ if $i\not\in J$;
\item $1\le b_i \le p$ if $i\not\in J$, and $b_i = 0$ if $i\in J$;
\item $\sum_{i\in S}b_ip^i - \sum_{i\in S}a_ip^i \equiv \sum_{i\in S}c_ip^i \bmod p^f-1$.
\end{itemize}
Moreover the $a_i$ and $b_i$ are uniquely determined by $\vec{c}$ and $J$ except in
the case where we can take either $a_i=p$ for $i\in J$ and $b_i=1$ for $i\not\in J$,
or $a_i=1$ for $i\in J$ and $b_i=p$ for $i\not\in J$.

\begin{lem}  \label{lem:generic} Suppose that $0 < c_i < p-1$ for all $i$.  Then $a_i < p$ and $b_i < p$
for all $i$ unless $\vec{c}=\vec{1}$, $J=\emptyset$ and $\vec{b} = \vec{p}$,
or $\vec{c}=\overrightarrow{p-2}$, $J=S$ and $\vec{a}=\vec{p}$.  In particular,
$\vec{a}$ and $\vec{b}$ are uniquely determined by
$\vec{c}$ and $J$ except in the above two cases where we can also
have $\vec{b}$ or $\vec{a} = \vec{1}$ instead of $\vec{p}$. 
\end{lem}
\begin{proof}  Suppose that $b_i = p$ and consider $\Sigma_i(\vec{c})$.  We have
$\Sigma_i(\vec{c}) \equiv \Sigma_i(\vec{b}) - \Sigma_i(\vec{a}) \bmod (p^f-1)$ and
$$1+p+\cdots+ p^{f-1} \le \Sigma_i(\vec{c}) \le (p^f-1) - (1+p+\cdots+ p^{f-1}).$$
If $\Sigma_i(\vec{b}) - \Sigma_i(\vec{a}) \in [0,p^f-1)$, then 
$\Sigma_i(\vec{c}) = \Sigma_i(\vec{b}) - \Sigma_i(\vec{a}) \equiv 0 \bmod p$, so 
$c_i=0$, giving a contradiction.  If  $\Sigma_i(\vec{b}) - \Sigma_i(\vec{a}) \in [1-p^f,0)$,
then $\Sigma_i(\vec{c}) = p^f-1 + \Sigma_i(\vec{b}) - \Sigma_i(\vec{a}) \equiv p-1  \bmod p$,
so $c_i = p-1$, giving a contradiction.  If $\Sigma_i(\vec{b}) - \Sigma_i(\vec{a}) \ge p^f-1$,
then $0 \le \Sigma_i(\vec{b}) - \Sigma_i(\vec{a}) - (p^f-1) \le 1+\cdots + p^{f-1}$, giving
$\Sigma_i(\vec{c}) = 1+\cdots + p^{f-1}$, so that $\vec{c}=\vec{1}$, $J=\emptyset$
and $\vec{b}=\vec{p}$.  If $\Sigma_i(\vec{b}) - \Sigma_i(\vec{a}) \le 1-p^f$, then
similar considerations give a contradiction.  The proof in the case $a_i = p$ is similar
(in fact, one can exchange $\vec{c}$ with $\overrightarrow{p-1} - \vec{c}$, $J$ with its complement
and $\vec{a}$ with $\vec{b}$), giving $\vec{c}=\overrightarrow{p-2}$, $J=S$ and $\vec{a}=\vec{p}$.
\end{proof}

Suppose that $V_1 = F(\chi_1)$ and $V_2 = F(\chi_2)$ where $\chi_1$ and $\chi_2$
are crystalline characters of $G_K$ with labelled Hodge-Tate weights $(b_{f-1},b_0,\ldots,b_{f-2})$ and 
$(a_{f-1},a_0,\ldots,a_{f-2})$
respectively, $V$ is an extension
$$0 \to V_1 \to V \to V_2 \to 0$$
of representations of $G_K$ over $F$, and $T$ is a $G_K$-stable $\CO_F$-lattice in $V$.
Letting $T_1 = T \cap V_1$ and $T_2 = T/T_1$, we have
$$0 \to T_1 \to T \to T_2 \to 0.$$
\begin{lem}  \label{lem:genexact} Suppose that $\vec{c}\in\Z^S$ is generic and $\vec{a},\vec{b}\in\Z^S$
are as above.  If $V$ is crystalline, then 
$$0 \to \N(T_1) \to \N(T) \to \N(T_2) \to 0$$
is exact.
\end{lem}
\begin{proof}  Since $V$ is crystalline, there is a Wach module $N=\N(T)$
over $\A_{K,F}^+$ corresponding to $T$.  Since $\vec{c}$ is generic, we have
$\max(a_i,b_i) \le p-1$ for all $i$, unless $\{\vec{a},\vec{b}\}=\{\vec{0},\vec{p}\}$.
If $\max(a_i,b_i) \le p-1$ for all $i$, then by Proposition~\ref{prop:rk2}~(1) and~(3),
we have
\begin{itemize}
\item $0 \le a_i' \le \max(a_i,b_i) \le p-1$ for all $i$, and
\item $\sum_{i=0}^{f-1} a_i'p^i \equiv \sum_{i=0}^{f-1}a_ip^i \bmod (p^f-1)$.
\end{itemize}
These conditions imply that $\vec{a} = \vec{a}'$ (unless $\{\vec{a},\vec{a}'\} =
\{\vec{0},\overrightarrow{p-1}\}$, which would give $\{\vec{a},\vec{b}\} = \{\vec{0},\overrightarrow{p-1}\}$
and hence that $\vec{c}=\vec{0}$ is not generic).  If 
$\{\vec{a},\vec{b}\}=\{\vec{0},\vec{p}\}$, then we instead use parts~(2) and~(3)
of Proposition~\ref{prop:rk2} to conclude that $\vec{a} = \vec{a}'$.
Thus in either case, we conclude from part~(4) of the proposition that
$\N(T)\to\N(T_2)$ is surjective, and therefore the sequence of Wach modules
is exact.
\end{proof}

Now consider a character $\psi:G_K\to \F^\times$.  By the classification of rank one
$(\varphi,\Gamma)$-modules over $\E_{K,F}$, there is a unique pair $C\in\F^\times$,
$\vec{c}\in\Z^S$ with $0\le c_i \le p-1$ and some $c_i < p-1$, such that
$\D(\F(\psi))\cong M_{C\vec{c}}$.
Suppose that $J\subset S$ and $\vec{a},\vec{b}\in\Z^S$ satisfying the usual
conditions, and that $A,B\in\F^\times$ with $BA^{-1}=C$.  Recall then that
we have defined a subspace $\Ext^1_{\bdd}(M_{A\vec{a}},M_{B\vec{b}})$ of
$\Ext^1(M_{A\vec{a}},M_{B\vec{b}})$ and an isomorphism
$$\iota: \Ext^1(M_{\vec{0}},M_{C\vec{c}}) \to \Ext^1(M_{A\vec{a}},M_{B\vec{b}}),$$
well-defined up to an element of $\F^\times$.   We then define $V_J$ as the
preimage of $\Ext^1_{\bdd}(M_{A\vec{a}},M_{B\vec{b}})$.  This space is independent of the
choices of $A$ and $B$ such that $BA^{-1}=C$, but for certain $J$ there are two
choices for the pair $\vec{a},\vec{b}$; we denote by $V_J^+$ the space gotten by
taking $a_i = p$ for all $i\in J$ and $b_i = 1$ for all $i\not\in J$, and by $V_J^-$
the one gotten by taking $a_i=1$ for all $i\in J$ and $b_i=p$ for all $i\not\in J$.

We now also recall the definition of the subspaces of $H^1(G_K,\F(\psi))$ used in \cite{BDJ05},
but we modify the notation from there to be more consistent with this paper.  (For the translation
between the notations, see the remark below.)
For $\psi$, $J$, $\vec{a}$, $\vec{b}$ as above, we consider a crystalline lift
$\tilde{\psi}_J:G_K \to F^\times$ of $\psi$ with labeled Hodge-Tate weights $(h_{f-1},h_0,\ldots,h_{f-2})$
where $h_i=-a_i$ if $i\in J$ and $h_i=b_i$ if $i \not\in J$.  Such a character $\tilde{\psi}_J$ is uniquely
determined up to an unramified twist, which we specify by requiring that $\tilde{\psi}_J(g)$
be the Teichm\"uller lift of $\psi(g)$ for $g\in G_K$ corresponding via local class field
theory to the uniformizer $p\in K^\times$.  When $(\vec{a},\vec{b})$ is not uniquely
determined by $J$, we adopt the notation $\tilde{\psi}_J^\pm$ as usual.  Recall that
$H^1_f(G_K,F(\tilde{\psi}_J))$ denotes the space of cohomology classes corresponding to
crystalline extensions
$$0 \to F(\tilde{\psi}_J) \to V \to F \to 0.$$
We then define the space $L_J'$  as the image in $H^1(G_K,\F(\psi))$ of the preimage
in $H^1(G_K,\CO_F(\tilde{\psi}_J))$ of $H^1_f(G_K,F(\tilde{\psi}_J))$.  We set $L_J=L_J'$
except in the following two cases:
\begin{itemize}
\item If $\psi$ is cyclotomic, $J=S$ and $\vec{a}=\vec{p}$, we let $L_J=H^1(G_K,\F(\psi))$.
\item If $\psi$ is trivial and $J\neq S$, we let $L_J$ be the span of $L_J'$ and the unramified
 class.
\end{itemize}
As usual we disambiguate using the notation $L_J^\pm$.  More precisely, we define $\tilde{\psi}_J^\pm$
as above, taking all $a_i = p$ and $b_j = 1$ for $\tilde{\psi}_J^+$, and all $a_i = 1$ and $b_j = p$
for $\tilde{\psi}_J^-$.  We then define $(L_J')^{\pm}$ as the image in $H^1(G_K,\F(\psi))$ of the preimage
in $H^1(G_K,\CO_F(\tilde{\psi}_J^\pm))$ of $H^1_f(G_K,F(\tilde{\psi}_J)^\pm)$.  We then make the same
modifications as above in the same exceptional cases to obtain the space $L_J^\pm$.
In particular, the first exceptional case above actually only applies to $L_J^+$.
We identify $L_J$ (or $L_J^\pm$) with subspaces of
$\Ext^1(M_{\vec{0}},M_{C\vec{c}})$ via the isomorphisms
$$H^1(G_K,\F(\psi)) \cong \Ext^1_{\F[G_K]}(\F,\F(\psi)) \cong \Ext^1(\D(\F),\D(\F(\psi)))
                    \cong \Ext^1(M_{\vec{0}},M_{C\vec{c}}),$$ 
the last of these given by an isomorphism $\D(\F(\psi))\cong M_{C\vec{c}}$ which is
unique up to an element of $\F^\times$.

\begin{rem} \label{rmk:compare} The article \cite{BDJ05} (after Lemma~3.9) defines spaces
$L_\alpha \subset H^1(G_K,\overline{\F}_p(\psi))$ for certain pairs $(V,J)$  where $J\subset S$ and
$V$ is an irreducible representation of $\gl_2(k)$.  The relation between the spaces is that 
$L_{(V,J')} = L_J \otimes_{\F}\overline{\F}_p$ where $J = \{\, i  \,|\, i - 1 \in J' \,\}$
and if 
$V\cong \otimes_{i\in S} \left(\det^{m_i}\otimes_k {\mathrm{Sym}}^{n_i-1} k^2 \otimes_{k,\tau_i}\overline{\F}_p\right)$,
then we take $a_i = n_{i-1}$ if $i\in J$ and $b_i = n_{i-1}$ if $i\not\in J$.
(The space $L_{(V,J')}$ is in fact independent of $\vec{m}$, and when there
are two choices of $\vec{n}$ compatible with $\psi$ and $J'$, the resulting
spaces $L_{(V,J')}$ are gotten from $L_J^\pm$ in the evident way.)
\end{rem}

We now prove our main result in the generic case.

\begin{thm} \label{thm:generic} Suppose that $\vec{c}$ is generic.
\begin{enumerate}
\item Suppose that $J\neq S$ (resp.~$J\neq\emptyset$) if $\vec{c}=\overrightarrow{p-2}$
 (resp.~$\vec{c}=\vec{1}$).  Then $V_J = L_J$, so $L_J = \oplus_{i\in J}L_{\{i\}}$.
\item If $\vec{c}=\overrightarrow{p-2}$ and $J=S$, then $V_J^\pm = L_J^\pm$, so
 $L_J^- = \oplus_{i\in J}L_{\{i\}}$ if $f > 1$.
\item If $\vec{c}=\vec{1}$ and $J=\emptyset$, then $V_J^\pm = L_J^\pm = \{0\}$.
\end{enumerate}
\end{thm}
\begin{proof} We first prove (1).
 Suppose that $x\in L_J$, so $x$ is a class of extensions
$$0 \to M_{C\vec{c}} \to E \to M_{\vec{0}} \to 0$$
corresponding via $\D$ to a class of extensions of Galois representations
$$0 \to \F(\psi) \to \overline{T} \to \F \to 0.$$
The assumption that $x\in L_J$ means that there is an extension
$$0 \to \CO_F(\tilde{\psi}_J) \to T \to \CO_F \to 0$$
whose reduction mod $\varpi_F$ is $\overline{T}$ and such that $F\otimes_{\CO_F}T$
is crystalline.  Let $\psi_2:G_K \to F^\times$ be a crystalline character
with labeled Hodge-Tate weights $(a_{f-1},a_0,\ldots,a_{f-2})$ and
let $\psi_1=\tilde{\psi}_J\psi_2$.  (Recall that $a_i = 0$ if $i\not\in J$ and
$b_i = 0$ if $i\in J$.)  Then $\psi_1$ is crystalline with
Hodge-Tate weights $(b_{f-1},b_0,\ldots,b_{f-2})$ and we have an exact sequence
$$0 \to T_1 \to T(\psi_2)\to T_2 \to 0$$
where $T_i = \CO_F(\psi_i)$ and $F\otimes_{\CO_F}T(\psi_2)$ is crystalline. 
By Lemma~\ref{lem:genexact},
the corresponding sequence of Wach modules over $\A_{K,F}^+$
$$0 \to \N(T_1) \to N \to \N(T_2)\to 0$$
is exact.  Reducing mod $\varpi_F$, we obtain an exact sequence
of free $\E_{K,F}^+$-modules with commuting $\varphi$ and $\Gamma$ actions
such that $\Gamma$ acts trivially mod $\pi$.  Tensoring with $\E_{K,F}$
yields an exact sequence
$$0 \to M_{B\vec{b}} \to E' \to M_{A\vec{a}} \to 0$$
of $(\varphi,\Gamma)$-modules, bounded with respect
to a basis for $\overline{N}$.  It follows that $E'$ defines
an element of $\Ext^1_{\bdd}(M_{A\vec{a}},M_{B\vec{b}})$.
Moreover this exact sequence is gotten from the one defining $x$
by twisting with $M_{A\vec{a}}$, so we have shown that $\iota(x)$
is bounded, and hence that $x\in V_J$.  Thus $L_J \subset V_J$.

By Proposition~\ref{gen} of this paper and Lemma~3.10 of \cite{BDJ05}, 
we have that $\dim_\F V_J = |J|= \dim_\F L_J = |J|$; therefore
$L_J = V_J$.  The assertion that $L_J = \oplus_{i\in J} L_{\{i\}}$
then also follows from Proposition~\ref{gen}.

The proof of (2) and (3) is exactly the same as (1), except that for
(2) in the cyclotomic case one uses Proposition~\ref{prop:cyclo}.
\end{proof}

\begin{rem}  We see from the proof of the theorem that in the definition of
$L_J$, $\tilde{\psi}_J$ can be replaced by its twist by any unramified
character $G_K\to \CO_F^\times$ with trivial reduction mod $\varpi_F$.
This can also be proved using Fontaine-Laffaille theory.

However in the case where $\psi$ is cyclotomic, $J=S$ and $\vec{a}=\vec{p}$,
we defined $L_J$ as $H^1(G_K, \F(\psi))$ rather than $L_J'$.  In fact $L_J'$
has codimension one and depends on the unramified twist, as the next proof shows.
\end{rem}

As a further application, we show that in the generic case, bounded extensions ``lift''
to extensions of Wach modules.
\begin{cor}  \label{cor:generic} Suppose that $\vec{c}\in\Z^S$ is generic and $\vec{a},\vec{b}\in\Z^S$
are as above and that
$$0 \to M_{B\vec{b}} \to E \to M_{A\vec{a}} \to 0$$
is a bounded extension of $(\varphi,\Gamma)$-modules over $\E_{K,F}$.
In the case $A=B$, $\vec{c}=\overrightarrow{p-2}$ and $\vec{a}=\vec{p}$, assume $F$ is ramified.
Then the extension $E$ arises by applying 
$\E_{K,F}\otimes_{\A_{K,F}^+}$ to an exact sequence over $\A_{K,F}^+$
of Wach modules of the form
$$0 \to \N(\psi_1) \to N \to \N(\psi_2) \to 0$$
where $\psi_1$ (resp.~$\psi_2$) is a crystalline character with
labeled Hodge-Tate weights $(b_{f-1},b_0,\ldots,b_{f-2})$
(resp.~$(a_{f-1},a_0,\ldots,a_{f-2})$).
\end{cor}
\begin{proof} First assume we are not in the exceptional case where
$A=B$, $\vec{c}=\overrightarrow{p-2}$ and $\vec{a}=\vec{p}$.  Since the extension class
defined by $E$ is bounded, the equality $V_J = L_J$ of the preceding
theorem shows that $E$ arises by applying $\D$ to the reduction mod
$\varpi_F$ of a crystalline extension
$$0 \to \CO_F(\psi_1) \to T \to \CO_F(\psi_2) \to 0$$
where $\psi_1$ and $\psi_2$ have the required Hodge-Tate weights.
Lemma~\ref{lem:genexact} then gives the desired extension of
Wach modules over $\A_{K,F}^+$.

Suppose now that $A=B$, $\vec{c}=\overrightarrow{p-2}$ and $\vec{a}=\vec{p}$.
Consider the class $x := \iota(E) \in \Ext^1(M_{\vec{0}},M_{\overrightarrow{p-2}})
\cong H^1(G_K,\F(\chi))$ where $\chi$ denotes the cyclotomic character.
We claim that there is an unramified character $\mu:G_K\to\CO_F^\times$
with trivial reduction mod $\varpi_F$ so that $x$ is in the
image of $H^1(G_K,\CO_F(\chi^p\mu))$.  (This is essentially proved
in Proposition~3.5 of \cite{kw_annals} or Section~3.2.7 of \cite{kw_inv2},
but there it is assumed that $x$ is tr\`es ramifi\'e, so we recall the
argument here.)
The long exact sequence associated to
$$0 \to \CO_F(\chi^p\mu) \stackrel{\varpi_F}{\longrightarrow}
         \CO_F(\chi^p\mu) \to \F(\chi) \to 0$$
shows that the image of $H^1(G_K,\CO_F(\chi^p\mu))$ is
the kernel of the connecting homomorphism 
$$H^1(G_K,\F(\chi)) \to H^2(G_K,\CO_F(\chi^p\mu)).$$
By Tate duality this is the space orthogonal to the image of
the connecting homomorphism
$$H^0(G_K,(F/\CO_F)(\chi^{1-p}\mu^{-1})) \to H^1(G_K,\F)$$
arising from the dual short exact sequence.  Letting $\alpha$
denote the homomorphism $G_K \to \F$ defined by
$(\chi^{1-p}-1)/p$, and $\beta$ the unramified homomorphism
sending ${\mathrm{Frob}}_K$ to $1$, we find that the image of the
connecting homomorphism is spanned 
by $\beta$ if $\mu \not\equiv 1 \bmod p\CO_F$ (which is possible as
$F$ is ramified over $\Q_p$) and by  $\alpha + \lambda\beta$
if $\mu({\mathrm{Frob}}_K) \equiv 1 + p\lambda \bmod p\varpi_F\CO_F$.
If $x\cup\beta=0$ then we can take $\mu \not\equiv 1 \bmod p\CO_F$,
and if $x\cup\beta\neq 0$ then there is a unique $\lambda$
so that $\lambda(x\cup \beta) = - x \cup\alpha$ and we
choose $\mu$ accordingly.  Now since $H^1(G_K,F(\chi^p\mu))
 = H^1_f(G_K,F(\chi^p\mu))$, we see that $E$ arises from the
reduction of a crystalline extension of the required form,
and the result again follows from Lemma~\ref{lem:genexact}.
\end{proof}

\subsection{$f=2$}
In this subsection we will show that if $f=2$, then $L_J=V_J$ (or $L_J^\pm = V_J^\pm$)
unless $\vec{c}=\vec{0}$; in other words, the space of bounded extensions coincides
with the one gotten from reductions of crystalline extensions of the corresponding
weights unless the ratio of the characters is unramified.  Furthermore, we give
a complete description in this exceptional case.

Before treating the case $f=2$, we note what happens in the case $f=1$.  The case
$\vec{c}\neq\vec{0}$ is already treated by the results of the preceding section.
Assume for the moment that $p>2$.  Then
the proof goes through just the same if $\vec{c}=\vec{0}$ and $J=S=\{0\}$.
Suppose then that $\vec{c}=\vec{0}$ and $J=\emptyset$.  If $C\neq 1$, then
$V_\emptyset = L_\emptyset = \{0\}$, so there is nothing to prove.  If
$C = 1$, then we have $V_\emptyset = \{0\}$, but $L_\emptyset = H^1(G_{\Q_p},\F)$.
Indeed all such classes arise as reductions of lattices in representations of the
form $\Q_p\oplus\Q_p(\chi^{1-p}\mu)$ with $\mu$ unramified; the corresponding
Wach module is described just as in Example~\ref{nonext} and so does not give
rise to a bounded extension.  If $p=2$, there are differences in the case $C=1$
(see Remark~\ref{rem:p2}).  In that case
$$V_S^+ = V_S^- = L_S^+ = H^1(G_{\Q_2},\F) = \langle B_\ur, B_\cyc, B_\tr \rangle$$
and $V_\emptyset^+ = V_\emptyset^- = \{0\}$, but 
$L_S^-=L_\emptyset^- = \langle B_\ur,B_\cyc\rangle$
and $L_\emptyset^+ = \langle B_\ur,B_\tr\rangle$.  (For the explicit descriptions,
note that the extensions of Galois representations are unramified twists of ones on which
$H_K$ acts trivially, and if $H_K$ acts trivially on $T$, then 
$\D(T)=\E_K\otimes T$.)

We now turn our attention to $f=2$.  We maintain the notation of the preceding
section, without the assumption that $\vec{c}$ is generic.  In particular $J\subset S$
and $\vec{a}$, $\vec{b}$ satisfy the usual conditions, $V_1$ and $V_2$ are
one-dimensional crystalline representaions with labelled Hodge-Tate weights 
$(b_1,b_0)$ and $(a_1,a_0)$, $V$ is an extension
of $V_2$ by $V_1$, $T$ is a $G_K$-stable $\CO_F$-lattice in $V$, $T_1 = T \cap V_1$ and
$T_2 = T/T_1$.  The refinement of Lemma~\ref{lem:genexact} is the following:
\begin{lem}  \label{lem:f2exact} Suppose that $f=2$ and $\vec{c}\neq\vec{0}$.  
If $V$ is crystalline, then 
$$0 \to \N(T_1) \to \N(T) \to \N(T_2) \to 0$$
is exact.
\end{lem}
\begin{proof}  Since the generic case is covered by Lemma~\ref{lem:genexact}, we
can assume (interchanging embeddings if necessary) that $\vec{c}=(i,0)$ for some
$i\in \{ 1,\ldots,p-2\}$ or $\vec{c} = (i,p-1)$ for
some $i\in \{0,\ldots,p-2\}$.  The cases where $J=\emptyset$ or $J=S$ are covered
by the same argument (using parts (2), (3) and (4) of Proposition~\ref{prop:rk2}),
as are the cases where $\vec{c}=(i,0)$ or $J=\{1\}$ (using parts (1), (3) and (4)
of the proposition).  We are thus left with the case where $\vec{c}=(i,p-1)$ for
some $i\in \{0,\ldots,p-2\}$ and $J=\{0\}$, in which case $\vec{a}=(p-i,0)$ and
$\vec{b}=(0,p)$.  In the notation of Proposition~\ref{prop:rk2}, the possible values
of $\vec{b}'$ are $(0,p)$ and $(1,0)$.  To complete the proof, we must rule out the
latter possibility, which we accomplish by considering the reduction of $\N(T)$ modulo
$p^2$.  From the exact sequence
$$0 \to \D(T_1) \to \D(T) \to \D(T_2) \to 0$$
and the description in \cite{Dou07} of rank one $(\varphi,\Gamma)$-modules
recalled in \S\ref{sec:rk1}, we see that there is a basis $\{e_1,e_2\}$
for $\D(T)$ over $\A_{K,F}$ in terms of which the matrices describing the
actions of $\varphi$ and $\gamma \in \Gamma$ are
$$P = \left(\begin{array}{cc} (\tilde{B},q^p)  &  *   \\
                    0    &  (\tilde{A}q^{p-i},1)  \end{array}\right)\quad\mbox{and}\quad
  G_\gamma = \left(\begin{array}{cc}  (\varphi(\Lambda_\gamma^p),\Lambda_\gamma^p)  & *    \\
                     0    &  (\Lambda_\gamma^{p-i}, \varphi(\Lambda_\gamma)^{p-i})     \end{array}\right)$$
for some $\tilde{A},\tilde{B}\in\CO_F^\times$.
On the other hand, since $V$ is crystalline, there is a basis $\{e_1',e_2'\}$
for $\D(T)$ over $\A_{K,F}$  in terms of which the matrices $P'$ and $G_\gamma'$
describing these actions lie in $\gl_2(\A_{K,F}^+)$, with $G_\gamma' \equiv I\bmod \pi\M_2(A_{K,F}^+)$.
If we assume that further that $\vec{b}'=(1,0)$ (and so $\vec{a}'=(p-i-1,p)$), then we
can choose $e_1',e_2'$ that
$$\overline{P}' \equiv \left(\begin{array}{cc} (B\pi^{p-1},1)  &  *   \\
                    0    &  (A\pi^{(p-i-1)(p-1)},\pi^{p(p-1)})  \end{array}\right)\quad\mbox{and}\quad
  \overline{G}'_\gamma = \left(\begin{array}{cc}  (\lambda_\gamma,\lambda_\gamma^p)  & *    \\
                     0    &  (\lambda_\gamma^{p^2+p-i-1}, \lambda_\gamma^{p^2-ip})     \end{array}\right),$$
where $\bar{\cdot}$ denotes reduction modulo $\varpi_F$.
Since $\D(T)\cong\A_{K,F}\otimes_{\A_{K,F}^+}\N(T)$, we can
write $(e_1',e_2') = (e_1,e_2)Q$ for some $Q\in\gl_2(A_{K,F})$,
and then we have 
$$P' = Q^{-1}P\varphi(T)\quad\mbox{and}\quad
  G_\gamma' = Q^{-1}G_\gamma\gamma(Q)\quad\mbox{for all $\gamma\in\Gamma$.}$$

\noindent{\bf Claim:} {\em $Q \equiv RS \bmod p\A_{K,F}$
for some matrices $R = \left(\begin{array}{cc}  \alpha (q^{-1},1)  & *    \\
                     0    &  \beta (q,1)     \end{array}\right) 
                     \in\gl_2(\A_{K,F})$
with $\alpha,\beta\in\CO_F^\times$,
and $S \in I + \varpi_F\M_2(A_{K,F}^+)$.}

Since $F$ may be ramified over $\Q_p$, we prove the claim by showing
inductively  that $Q \equiv R_mS_m \bmod \varpi_F^m\A_{K,F}$
for some matrices $R_m,S_m$ of the prescribed form
for $m=1,\ldots,e$ where $e=e(F/\Q_p)$.

To prove the statement for $m=1$, note that setting
$R_0 = \left(\begin{array}{cc}(q^{-1},1)&0\\0&(q,1)\end{array}\right)$
gives 
$$\overline{R}_0^{-1}\overline{P}\varphi(\overline{R}_0)
  = \left(\begin{array}{cc}(B\pi^{p-1},1)&*\\
              0&(A\pi^{(p-i)(p-1)},\pi^{p(p-1)})\end{array}\right).$$
So if we write $R=R_0S_0$, then
$$\overline{S}_0\left(\begin{array}{cc}(B\pi^{p-1},1)&*\\
              0&(A\pi^{(p-i)(p-1)},\pi^{p(p-1)})\end{array}\right)
    = \left(\begin{array}{cc}(B\pi^{p-1},1)&*\\
               0&(A\pi^{(p-i)(p-1)},\pi^{p(p-1)})\end{array}\right)
     \varphi(\overline{S}_0).$$
It follows easily that $\overline{S}_0= 
\left(\begin{array}{cc}\bar{\alpha}&\bar{\delta}\\0&\bar{\beta}\end{array}\right)$ for some $\bar{\alpha},\bar{\beta}\in\F^\times$, $\bar{\delta}
\in\E_{K,F}$.  Choosing lifts $\alpha,\beta\in\CO_F^\times$ and
$\delta\in\A_{K,F}$ and setting $R_1=R_0\left(\begin{array}{cc}
\alpha&\delta\\ 0&\beta\end{array}\right)$ gives the result for $m=1$.

Suppose now that $m\in \{1,\ldots,e-1\}$ and that
$Q \equiv R_mS_m \bmod \varpi_F^m\A_{K,F}$ with $R_m,S_m$ of
the prescribed form.  Setting $Q_m = R_m^{-1}QS_m^{-1}$, we have
$Q_m = I + \varpi_F^m Q_m'$ for some $Q_m' \in \M_2(\A_{K,F})$.
Define
$$\begin{array}{rlcl} & P_m = R_m^{-1} P \varphi(R_m), && 
                        G_{\gamma,m} = R_m^{-1} G_\gamma \gamma(R_m),\\
                  & P_m' = S_m^{-1} P' \varphi(S_m) &\mbox{and}& 
                        G_{\gamma,m}' = S_m^{-1} G_\gamma' \gamma(S_m),\\
       \mbox{so that\ }& P_m' = Q_m^{-1} P_m \varphi(Q_m)&\mbox{and}&
                        G_{\gamma,m}' = Q_m^{-1} G_{\gamma,m}' \gamma(Q_m).\end{array}$$
Note that $P_m'\in \M_2(\A_{K,F}^+)$,
$G_{\gamma,m}' \in I + \pi\M_2(\A_{K,F}^+)$, $P_m\equiv P_m' \bmod \varpi_F^m\M_2(\A_{K,F})$,
$G_{\gamma,m} \equiv G_{\gamma,m}' \bmod \varpi_F^m\M_2(\A_{K,F})$,
$$\begin{aligned} P_m & \equiv \left(\begin{array}{cc} (\tilde{B}\pi^{p-1},1)& * \\
                                           0 & (\tilde{A}\pi^{(p-i-1)(p-1)},\pi^{p(p-1)})
                                           \end{array}\right) \bmod p\M_2(\A_{K,F})\\
        \mbox{and\ } G_{\gamma,m} & \equiv
         \left(\begin{array}{cc} (\lambda_\gamma,\lambda_\gamma^p)& * \\
                                           0 & (\lambda_\gamma^{p^2+p-i-1},\lambda_\gamma^{p^2-ip})
                                           \end{array}\right) \bmod p\M_2(\A_{K,F}).\end{aligned}$$
Note that since $m+1 \le e$, the last two congruences hold mod $\varpi_F^{m+1}$,
and that $Q_m^{-1} \equiv I - \varpi_F^m Q_m' \bmod \varpi_F^{m+1}\M_2(\A_{K,F})$.
It follows that 
$$\begin{aligned} P_m' & \equiv (I-\varpi_F^m Q_m')P_m(I+\varpi_F^m \varphi(Q_m'))\\
                    &\equiv P_m + \varpi_F^m(P_m\varphi(Q_m') - Q_m'P_m)\bmod \varpi_F^{m+1}\M_2(\A_{K,F}),
                    \end{aligned}$$
and therefore that 
$$\varpi_F^m(P_m\varphi(Q_m') - Q_m'P_m) \equiv P_m'-P_m \equiv \varpi_F^m
    \left(\begin{array}{cc}x&y\\z&w\end{array}\right)\bmod \varpi_F^{m+1}\M_2(\A_{K,F})$$
with $x,z,w\in\A_{K,F}^+$.  Note that $\overline{P}_m = \overline{P}'$, so
we have $\overline{P}'\varphi(\overline{Q}_m')-\overline{Q}_m'\overline{P}'
           =    \left(\begin{array}{cc}\bar{x}&\bar{y}\\\bar{z}&\bar{w}\end{array}\right)$
for some $\bar{x},\bar{z},\bar{w}\in\E_{K,F}^+$, $\bar{y}\in\E_{K,F}$.  Similarly we find that
$\overline{G}_{\gamma}'\gamma(\overline{Q}_m')-\overline{Q}_m'\overline{G}_{\gamma}'
           =    \left(\begin{array}{cc}\bar{x}_\gamma&\bar{y}_\gamma\\
                   \bar{z}_\gamma&\bar{w}_\gamma\end{array}\right)$
for some $\bar{x}_\gamma,\bar{z}_\gamma,\bar{w}_\gamma\in\pi\E_{K,F}^+$, $\bar{y}_\gamma\in\E_{K,F}$.

Writing 
$$\overline{Q}_m' = \left(\begin{array}{cc}(r_0,r_1)&(s_0,s_1)\\(t_0,t_1)&(u_0,u_1)\end{array}\right)$$
with $r_0,r_1,\ldots,u_0,u_1\in \F((\pi))$, the condition that $\bar{z}\in\E_{K,F}^+$
becomes $\pi^{p(p-1)}t_0(\pi)^p - t_1(\pi), A\pi^{(p-i-1)(p-1)}t_1(\pi^p)-B\pi^{p-1}t_0(\pi)
\in \F[[\pi]]$, from which one deduces that $\val(t_0) \ge 1-p$ and $\val(t_1)\ge 0$.
The condition that $\bar{z}_\gamma\in\pi\E_{K,F}^+$ then becomes that
$\gamma(t_0)\lambda_\gamma^{p^2+p-i-2}-t_0 \in \pi\F[[\pi]]$.  Lemma~\ref{delta}
rules out the possibility that $1-p < \val(t_0) < 0$, and Lemma~\ref{gamma}
rules out the possibility that $\val(t_0) = 1-p$.  Therefore $(t_0,t_1)\in\E_{K,F}^+$.
Since $\overline{P}'\in \M_2(\E_{K,F}^+)$, the condition that $\bar{x}\in\E_{K,F}^+$
then becomes that $\pi^{p-1}(r_1(\pi^p)-r_0(\pi)),r_0(\pi^p)-r_1(\pi)\in \F[[\pi]]$,
which implies that $(r_0,r_1)\in\E_{K,F}^+$.   The condition that $\bar{w}\in\E_{K,F}^+$
becomes that $\pi^{(p-i-1)(p-1)}(u_1(\pi^p)-u_0(\pi)),
\pi^{p(p-1)}u_0(\pi^p)-u_1(\pi)\in \F[[\pi]]$, which implies that $\val(u_0) \ge 1-p$
and $\val(u_1) \ge i+2-p$.  Since $(t_0,t_1)\in \E_{K,F}^+$ and $\overline{G}_\gamma'
\equiv I \bmod \pi\M_2(\E_{K,F}^+)$, the condition that $\overline{w}_\gamma\in
\pi\E_{K,F}^+$ becomes that $\gamma(u_i) - u_i \in \pi\E_{K,F}^+$ for $i=0,1$, so that Lemmas~\ref{delta}
and~\ref{gamma} again imply that $(u_0,u_1)\in\E_{K,F}^+$.
We can thus lift $\overline{Q}_m'$ to a matrix 
$\left(\begin{array}{cc}r&s \\t&u\end{array}\right)\in\M_2(\A_{K,F})$ with $r,t,u\in\A_{K,F}^+$.  Setting 
$R_{m+1} = R_m\left(\begin{array}{cc}1&\varpi_F^m s \\0&1\end{array}\right)$
and $S_{m+1} = \left(\begin{array}{cc}1+\varpi_F^m r &0 \\
  \varpi_F^m t &1 + \varpi_F^m u \end{array}\right)S_m$
then gives $Q\equiv R_{m+1}S_{m+1} \bmod \varpi_F^{m+1}\M_2(\A_{K,F})$
with $R_{m+1},S_{m+1}$ of the prescribed form, and completes
the proof of the claim.

To derive a contradiction from the claim, we proceed as in the proof of the
induction step above, but with $m=e$ and working modulo $\varpi_F^{m+1}$.
More precisely, we define $Q_e$, $Q_e'$, $P_e$, $G_{\gamma,e}$, $P_e'$ and
$G_{\gamma,e}'$ as above; the difference now is that the congruences satisfied
by $P_e$ and $G_{\gamma,e}$ modulo $p$ are not satisfied modulo $p\varpi_F$.
In particular, the upper-left hand entry of $P_e$ is $(\tilde{A}q,q^p/\varphi(q))$,
and a straightforward calculation shows that
$$\frac{q^p}{\varphi(q)} \equiv 1 + p(g(\pi^{-p})- g(\pi^{-1})
    +f(\pi)) \bmod p^2\A_{\Q_p}$$
where $g(X) = \sum_{i=1}^{p-1} (-X)^i/i$ and $f(\pi)\in\A_{\Q_p}^+$.
As before, we have
$\overline{P}'\varphi(\overline{Q}_e')-\overline{Q}_e'\overline{P}'
           =    \left(\begin{array}{cc}\bar{x}&\bar{y}\\\bar{z}&\bar{w}\end{array}\right)$
with $\bar{z}\in \E_{K,F}^+$ since $P_e$ is upper-triangular, but now
$\bar{x} \in (0,c(\bar{g}(\pi^{-p})- \bar{g}(\pi^{-1}))) + \E_{K,F}^+$
for some $c\in\F^\times$ (the reduction of $p/\varpi_F^e$).
Similarly we have
$\overline{G}_{\gamma}'\gamma(\overline{Q}_e')-\overline{Q}_e'\overline{G}_{\gamma}'
           =    \left(\begin{array}{cc}\bar{x}_\gamma&\bar{y}_\gamma\\
                   \bar{z}_\gamma&\bar{w}_\gamma\end{array}\right)$
with $\bar{z}_\gamma\in \pi\E_{K,F}^+$.
So just as before we get $(t_0,t_1)\in \E_{K,F}^+$, but this implies that
$\pi^{p-1}(r_1(\pi^p)-r_0(\pi))\in \F[[\pi]]$ and 
$r_0(\pi^p)-r_1(\pi)\in c(\bar{g}(\pi^{-p})- \bar{g}(\pi^{-1}))) + \F[[\pi]]$,
which leads to a contradiction and completes the proof of the 
lemma.
\end{proof}

\begin{thm}  \label{thm:f2} Suppose that $f=2$ and $\vec{c}\neq\vec{0}$.  Then $V_J = L_J$
(or $V_J^\pm = L_J^\pm$) for all $J\subset S$.  In particular $L_{\{0\}} = L_{\{1\}}$
if and only if $\vec{c} = (i,p-1)$ or $(p-1,i)$ for some $i\in\{1,\ldots,p-2\}$.
\end{thm}
\begin{proof}  The proof of the first assertion is exactly the same as for 
Theorem~\ref{thm:generic}.  The second then follows from the corresponding result
for $V_J$ in \S\ref{sec:f2}.
\end{proof}

Theorem~\ref{thm:intro2} of the introduction now follows in view of Corollary~\ref{cor:fun}.

\begin{rem}  \label{rmk:indepf2} Again we see that in the definition of
$L_J$, $\tilde{\psi}_J$ can be replaced by its twist by any unramified
character with trivial reduction; the cases where some $a_i$ or $b_i$
is $p$ (with $J = \{0\}$ or $\{1\}$) are outside the range of
Fontaine-Laffaille theory.

Note also that the case where we had to work the hardest in the proof
of Lemma~\ref{lem:f2exact} is precisely the one where $L_{\{0\}} = L_{\{1\}}$.
\end{rem}

By the same proof as Corollary~\ref{cor:generic}, we obtain:
\begin{cor} \label{cor:f2} Suppose that $f=2$ and
$\vec{c}\neq\vec{0}$ and $\vec{a},\vec{b}\in\Z^S$
are as above and that
$$0 \to M_{B\vec{b}} \to E \to M_{A\vec{a}} \to 0$$
is a bounded extension of $(\varphi,\Gamma)$-modules over $\E_{K,F}$.
In the case $A=B$, $\vec{c}=\overrightarrow{p-2}$ and $\vec{a}=\vec{p}$, assume $F$ is ramified.
Then the extension $E$ arises by applying 
$\E_{K,F}\otimes_{\A_{K,F}^+}$ to an exact sequence over $\A_{K,F}^+$
of Wach modules of the form
$$0 \to \N(\psi_1) \to N \to \N(\psi_2) \to 0$$
where $\psi_1$ (resp.~$\psi_2$) is a crystalline character with
labeled Hodge-Tate weights $(b_1,b_0)$
(resp.~$(a_1,a_0)$).
\end{cor}

We now say what we can in the case $\vec{c} = \vec{0}$.
First note that the proof of Lemma~\ref{lem:f2exact} goes through
in the following cases:
\begin{itemize}
\item $J=S$, in which case $\vec{a}=\overrightarrow{p-1}$ (or $\vec{2}$ if $p=2$);
\item $J=\{0\}$, $\vec{a} = (p,0)$, $\vec{b} = (0,1)$ (the $+$ case);
\item $J=\{1\}$, $\vec{a} = (0,p)$, $\vec{b} = (1,0)$ (the $+$ case).
\end{itemize}
The proof of Theorem~\ref{thm:f2} goes through in these cases
unless $J=S$, $p=2$, $\vec{a} = 1$, $C=1$ where we get $L_S^- \subset V_S^-$, 
but $\dim L_S^- = 3 \neq \dim V_S^- = 4$ (see Remark~\ref{rem:p2}). 
In this case however we
know that $L_S^-$ consists of the peu ramifi\'ee extensions.
To compute the corresponding $(\varphi,\Gamma)$-modules, note that
$V_{\{0\}}^+ = L_{\{0\}}^+$ contains the classes arising
from reductions of Galois stable lattices in 
$F(\mu\psi^2) \oplus F(\psi^\sigma)$ where $\psi:G_K \to \CO_F^\times$
is a crystalline character with labeled Hodge-Tate weights $(0,1)$,
$\sigma$ is the non-trivial element of $\gal(K/\Q_2)$, and 
$\mu:G_K \to \CO_F^\times$ is an unramified character with trivial
reduction mod $\varpi_F$.  These classes correspond to homomorphisms
$G_K \to \F$ whose restriction to inertia is a multiple of the
reduction of $1/2(\psi^\sigma\psi^{-2}-1)|_{I_K}$.
One can compute these explicitly using class field theory and
check that they are peu ramifi\'ee.  It follows that $L_{\{0\}}^+\subset L_S^-$,
and similarly $L_{\{1\}}^+\subset L_S^-$, so 
that $L_S^- = \langle B_\ur,B_0,B_1 \rangle$.

If $J=\emptyset$, we have $V_\emptyset = \{0\}$, and $L_\emptyset = \{0\}$
unless $C=1$.  If $C=1$, one can compute the extensions and associated
$(\varphi,\Gamma)$-modules explicitly since they are unramified twists
of representations on which $H_K$ acts trivially.  If $p\neq 2$, 
one gets $L_\emptyset = \langle B_\ur, B_\cyc \rangle$.  If $p=2$,
one gets $L_\emptyset^+ = \langle B_\ur, B_\cyc \rangle$ (with $\vec{b}=\vec{1}$)
and $L_\emptyset^- = \langle B_\ur, B_\tr \rangle$ (with $\vec{b}=\vec{2}$).

The most interesting is the $-$ case when $S=\{0\}$ or $\{1\}$.  For example
if $S=\{0\}$, $\vec{a} = (p,0)$ and $\vec{b} = (0,1)$,
the proof of Lemma~\ref{lem:f2exact} breaks down, but we see that if
the associated sequence of Wach modules is not exact, then $\vec{a}'=(0,1)$
and $\vec{b}'=(p,0)$, so the extension of $(\varphi,\Gamma)$-modules
associated to $\overline{T}$ is in $V^+_{\{1\}}$.
Since $V^-_{\{0\}} \subset V^+_{\{1\}}$, it follows that 
$L^-_{\{0\}} \subset V^+_{\{1\}}$, and dimension counting
implies equality.  We therefore have that $V^-_{\{0\}}$ is
contained in $L^-_{\{0\}} = V^+_{\{1\}} = L^+_{\{1\}}$
with codimension one.  Similarly $V^-_{\{1\}}$ is
contained in $L^-_{\{1\}} = V^+_{\{0\}} = L^+_{\{0\}}$ 
with codimension one.

Putting everything together we get:
\begin{thm}  \label{thm:unramified} Suppose that $f=2$, $\vec{c}=\vec{0}$.
\begin{enumerate}
\item If $C\neq 1$, then:
\begin{itemize}
\item if $p>2$ then $L_S = V_S = \Ext^1(M_{\vec{0}},M_{C\vec{0}})$;
\item if $p=2$ then $L_S^\pm = V_S^\pm = \Ext^1(M_{\vec{0}},M_{C\vec{0}})$;
\item $V_{\{0\}}^- = V_{\{1\}}^- = \{0\}$, and
  $L_{\{0\}}^- = L_{\{1\}}^+ = V_{\{1\}}^+  \neq 
    V_{\{0\}}^+ = L_{\{0\}}^+ = L_{\{1\}}^-$;
\item if $p>2$ then $L_\emptyset = V_\emptyset = \{0\}$;
\item if $p=2$ then $L_\emptyset^\pm = V_\emptyset^\pm = \{0\}$.
\end{itemize}
\item If $C = 1$, then:
\begin{itemize}
\item if $p>2$ then $L_S = V_S = \Ext^1(M_{\vec{0}},M_{\vec{0}})$;
\item if $p=2$ then $L_S^+ = V_S^\pm = \Ext^1(M_{\vec{0}},M_{\vec{0}})$ and
 $L_S^- = \langle B_\ur,B_0,B_1\rangle$;
\item $V_{\{0\}}^- = V_{\{1\}}^- = \langle B_\ur\rangle$, 
  $L_{\{0\}}^- = L_{\{1\}}^+ = V_{\{1\}}^+ = \langle B_\ur,B_1\rangle$, and 
  $L_{\{1\}}^- = L_{\{0\}}^+ = V_{\{0\}}^+ = \langle B_\ur,B_0\rangle$;
\item if $p>2$ then $V_\emptyset = \{0\}$ and $L_\emptyset = \langle B_\ur,B_\cyc\rangle$;
\item if $p=2$ then $V_\emptyset^\pm = \{0\}$, $L_\emptyset^+ = \langle B_\ur,B_\cyc\rangle$
 and $L_\emptyset^- = \langle B_\ur,B_\tr\rangle$.
\end{itemize}
\end{enumerate}
\end{thm}

Note that the strict inclusion
$V^-_{\{0\}}\subset L^-_{\{0\}}$ implies the existence
of {\em non-split} crystalline extensions
$0 \to F(\psi_1) \to V \to F(\psi_2) \to 0$
with Galois stable $\CO_F$-lattices $T$ 
such that the corresponding sequence of Wach
modules over $\A_{K,F}^+$ is not exact (with
$\psi_1$ and $\psi_2$ of labeled Hodge-Tate
weights $(p,0)$ and $(0,1)$ respectively).
 
As in Remark~\ref{rmk:indepf2},
we see that the definitions of $L_J$ are
independent of the choice of unramified
twist, unless $C=1$, $J=\emptyset$ and
$F$ is ramified, in which
case twisting by an unramified character 
which is trivial mod $\varpi_F$ but not mod $p$
would give $L_J' = L_J = \langle B_\ur \rangle$.

Finally we remark that the proof of Corollary~\ref{cor:f2} goes
through when $\vec{c}=\vec{0}$ except in the following two cases where $C=1$:
\begin{itemize}  
\item If $p=2$ and $\vec{a}=\vec{1}$, then only classes in $L_S^-$
 lift (see Remark~\ref{rem:p2}).
\item If $\vec{a}=(1,0)$ and $\vec{b}=(0,p)$ (or $\vec{a}=(0,1)$ and $\vec{b}=(p,0)$),
 then we have not determined whether $B_\ur$ lifts.
\end{itemize}

\end{document}